\documentclass[12pt]{amsart}

\usepackage{etoolbox}
\usepackage{booktabs}

\makeatletter
\let\old@tocline\@tocline
\let\section@tocline\@tocline
\newcommand{\subsection@dotsep}{4.5}
\newcommand{\subsubsection@dotsep}{4.5}
\patchcmd{\@tocline}
{\hfil}
{\nobreak
	\leaders\hbox{$\m@th
		\mkern \subsection@dotsep mu\hbox{.}\mkern \subsection@dotsep mu$}\hfill
	\nobreak}{}{}
\let\subsection@tocline\@tocline
\let\@tocline\old@tocline

\patchcmd{\@tocline}
{\hfil}
{\nobreak
	\leaders\hbox{$\m@th
		\mkern \subsubsection@dotsep mu\hbox{.}\mkern \subsubsection@dotsep mu$}\hfill
	\nobreak}{}{}
\let\subsubsection@tocline\@tocline
\let\@tocline\old@tocline

\let\old@l@subsection\l@subsection
\let\old@l@subsubsection\l@subsubsection

\def\@tocwriteb#1#2#3{%
	\begingroup
	\@xp\def\csname #2@tocline\endcsname##1##2##3##4##5##6{%
		\ifnum##1>\c@tocdepth
		\else \sbox\z@{##5\let\indentlabel\@tochangmeasure##6}\fi}%
	\csname l@#2\endcsname{#1{\csname#2name\endcsname}{\@secnumber}{}}%
	\endgroup
	\addcontentsline{toc}{#2}%
	{\protect#1{\csname#2name\endcsname}{\@secnumber}{#3}}}%

\newlength{\@tocsectionindent}
\newlength{\@tocsubsectionindent}
\newlength{\@tocsubsubsectionindent}
\newlength{\@tocsectionnumwidth}
\newlength{\@tocsubsectionnumwidth}
\newlength{\@tocsubsubsectionnumwidth}
\newcommand{\settocsectionnumwidth}[1]{\setlength{\@tocsectionnumwidth}{#1}}
\newcommand{\settocsubsectionnumwidth}[1]{\setlength{\@tocsubsectionnumwidth}{#1}}
\newcommand{\settocsubsubsectionnumwidth}[1]{\setlength{\@tocsubsubsectionnumwidth}{#1}}
\newcommand{\settocsectionindent}[1]{\setlength{\@tocsectionindent}{#1}}
\newcommand{\settocsubsectionindent}[1]{\setlength{\@tocsubsectionindent}{#1}}
\newcommand{\settocsubsubsectionindent}[1]{\setlength{\@tocsubsubsectionindent}{#1}}

\renewcommand{\l@section}{\section@tocline{1}{\@tocsectionvskip}{\@tocsectionindent}{}{\@tocsectionformat}}%
\renewcommand{\l@subsection}{\subsection@tocline{1}{\@tocsubsectionvskip}{\@tocsubsectionindent}{}{\@tocsubsectionformat}}%
\renewcommand{\l@subsubsection}{\subsubsection@tocline{1}{\@tocsubsubsectionvskip}{\@tocsubsubsectionindent}{}{\@tocsubsubsectionformat}}%
\newcommand{\@tocsectionformat}{}
\newcommand{\@tocsubsectionformat}{}
\newcommand{\@tocsubsubsectionformat}{}
\expandafter\def\csname toc@1format\endcsname{\@tocsectionformat}
\expandafter\def\csname toc@2format\endcsname{\@tocsubsectionformat}
\expandafter\def\csname toc@3format\endcsname{\@tocsubsubsectionformat}
\newcommand{\settocsectionformat}[1]{\renewcommand{\@tocsectionformat}{#1}}
\newcommand{\settocsubsectionformat}[1]{\renewcommand{\@tocsubsectionformat}{#1}}
\newcommand{\settocsubsubsectionformat}[1]{\renewcommand{\@tocsubsubsectionformat}{#1}}
\newlength{\@tocsectionvskip}
\newcommand{\settocsectionvskip}[1]{\setlength{\@tocsectionvskip}{#1}}
\newlength{\@tocsubsectionvskip}
\newcommand{\settocsubsectionvskip}[1]{\setlength{\@tocsubsectionvskip}{#1}}
\newlength{\@tocsubsubsectionvskip}
\newcommand{\settocsubsubsectionvskip}[1]{\setlength{\@tocsubsubsectionvskip}{#1}}

\patchcmd{\tocsection}{\indentlabel}{\makebox[\@tocsectionnumwidth][l]}{}{}
\patchcmd{\tocsubsection}{\indentlabel}{\makebox[\@tocsubsectionnumwidth][l]}{}{}
\patchcmd{\tocsubsubsection}{\indentlabel}{\makebox[\@tocsubsubsectionnumwidth][l]}{}{}

\newcommand{\@sectypepnumformat}{}
\renewcommand{\contentsline}[1]{%
	\expandafter\let\expandafter\@sectypepnumformat\csname @toc#1pnumformat\endcsname%
	\csname l@#1\endcsname}
\newcommand{\@tocsectionpnumformat}{}
\newcommand{\@tocsubsectionpnumformat}{}
\newcommand{\@tocsubsubsectionpnumformat}{}
\newcommand{\setsectionpnumformat}[1]{\renewcommand{\@tocsectionpnumformat}{#1}}
\newcommand{\setsubsectionpnumformat}[1]{\renewcommand{\@tocsubsectionpnumformat}{#1}}
\newcommand{\setsubsubsectionpnumformat}[1]{\renewcommand{\@tocsubsubsectionpnumformat}{#1}}
\renewcommand{\@tocpagenum}[1]{%
	\hfill {\mdseries\@sectypepnumformat #1}}

\let\oldappendix\appendix
\renewcommand{\appendix}{%
	\leavevmode\oldappendix%
	\addtocontents{toc}{%
		\protect\settowidth{\protect\@tocsectionnumwidth}{\protect\@tocsectionformat\sectionname\space}%
		\protect\addtolength{\protect\@tocsectionnumwidth}{2em}}%
}
\makeatother

\makeatletter
\settocsectionnumwidth{2em}
\settocsubsectionnumwidth{2.5em}
\settocsubsubsectionnumwidth{3em}
\settocsectionindent{1pc}%
\settocsubsectionindent{\dimexpr\@tocsectionindent+\@tocsectionnumwidth}%
\settocsubsubsectionindent{\dimexpr\@tocsubsectionindent+\@tocsubsectionnumwidth}%
\makeatother

\settocsectionvskip{10pt}
\settocsubsectionvskip{0pt}
\settocsubsubsectionvskip{0pt}

\settocsectionformat{\bfseries}
\settocsubsectionformat{\mdseries}
\settocsubsubsectionformat{\mdseries}
\setsectionpnumformat{\bfseries}
\setsubsectionpnumformat{\mdseries}
\setsubsubsectionpnumformat{\mdseries}

\let\oldtableofcontents\tableofcontents
\renewcommand{\tableofcontents}{%
	\vspace*{-\linespacing}
	\oldtableofcontents}

\usepackage[margin=1in,marginparwidth=0.8in, marginparsep=0.1in]{geometry}

\usepackage{setspace}

\usepackage{times}

\usepackage{soul}

\usepackage{amsfonts,amssymb,latexsym,amsmath,graphicx, bbm,tikz-cd}
\usepackage{stmaryrd} 
\usepackage{mathrsfs}

\usepackage{enumitem}

\usepackage{amscd}
\usepackage{colonequals}

\usetikzlibrary{decorations.pathmorphing}

\usepackage[bookmarks=true, bookmarksopen=true,%
bookmarksdepth=3,bookmarksopenlevel=2,%
colorlinks=true,%
linkcolor=blue,%
citecolor=blue,%
filecolor=blue,%
menucolor=blue,%
urlcolor=blue]{hyperref}

\newtheorem{theorem}{Theorem}[subsection]
\newtheorem{lemma}[theorem]{Lemma}
\newtheorem{proposition}[theorem]{Proposition}
\newtheorem{corollary}[theorem]{Corollary}
\newtheorem{theoremdef}[theorem]{Theorem/Definition}

\theoremstyle{definition}
\newtheorem{definition}[theorem]{Definition}
\newtheorem{convention}[theorem]{Convention}
\newtheorem{hypothesis}[theorem]{Hypothesis}

\theoremstyle{remark} 
\newtheorem{remarque}[theorem]{Remark}
\newtheorem{remark}[theorem]{Remark}
\newtheorem{caution}[theorem]{Caution}
\newtheorem{example}[theorem]{Example}

\newcommand{\A}{\mathbb{A}}

\newcommand{\C}{\mathbb{C}}
\newcommand{\D}{\mathbb{D}}

\newcommand{\G}{\mathbb{G}}
\renewcommand{\H}{\mathrm{H}}

\renewcommand{\L}{\mathbb{L}}

\renewcommand{\P}{\mathbb{P}}
\newcommand{\Q}{\mathbb{Q}}
\newcommand{\R}{\mathbb{R}}

\newcommand{\U}{\mathbb{U}}

\newcommand{\Z}{\mathbb{Z}}

\newcommand{\cF}{\mathcal{F}}

\newcommand{\cI}{\mathcal{I}}
\newcommand{\cJ}{\mathcal{J}}
\newcommand{\cK}{\mathcal{K}}
\newcommand{\cL}{\mathcal{L}}

\newcommand{\bE}{\mathbf{E}}

\newcommand{\bS}{\mathbf{S}}

\newcommand{\ed}{E(\Gamma)}
\newcommand{\ver}{V(\Gamma)}

\newcommand{\qh}{$(G, M)$}
\newcommand{\qhm}{\qh-space} 
 
\newcommand{\mom}{\A}

\newcommand{\uu}{\mathbb{U}_1}

\newcommand{\YBet}{Y}
\newcommand{\YDol}{X}
\newcommand{\starGG}{\Gm}
\newcommand{\starUU}{\U_1}

\newcommand{\Gm}{\C^*}

\newcommand{\Dol}{\mathfrak{D}}
\newcommand{\Bet}{\mathfrak{B}}

\DeclareMathOperator{\res}{res}

\newcommand{\GApoint}{\mathbf{0}}

\newcommand{\closedpart}{K}
\newcommand{\openpart}{Q}

\DeclareMathOperator{\im}{Im}

\newcommand{\CKS}{\mathfrak{C}}

\newcommand{\reg}{\operatorname{reg}}

\newcommand{\CKSmap}{\frak{c}}

\newcommand{\delfilt}{\mathrm{D}}
\newcommand{\redC}{\overline{C}}
\newcommand{\versD}{\frak{V}}
\newcommand{\gr}{\operatorname{gr}}
\newcommand{\boing}{\mbox{$ \bigcirc \!\! \bullet$}}
\newcommand{\homsep}{,}

\newcommand{\UUCstar}{\star_{\uutimesC}}
\newcommand{\basichodge}{\frak{F}}
\newcommand{\dmom}{\pi_\Dol}
\newcommand{\dmomD}{q}

\newcommand{\dmomabrevU}{\mu_{\Dol}}
\newcommand{\baseopen}{\frak{U}}
\newcommand{\comap}{\operatorname{co}}
\newcommand{\relmap}{\operatorname{rel}}
\newcommand{\momenthomomorphism}{\mathfrak{f}}

\newcommand{\uutimesD}{\D \times \uu}
\newcommand{\uutimesC}{\C \times \uu}
\newcommand{\hyperbasepoint}{0 \times \mathbf{1}}

\newcommand{\stratum}{S}

\newcommand{\subedges}{J}

\newcommand{\openset}{U}
\newcommand{\generalstratum}{S}
\newcommand{\generalsubedges}{J}

\author{Zsuzsanna Dancso, Michael McBreen, and Vivek Shende}

\title{Deletion-contraction triangles for Hausel-Proudfoot varieties}

\begin{document}
	
	\begin{abstract} To a graph, Hausel and Proudfoot associate two complex manifolds, $\Bet$ and $\Dol$, which behave, 
		respectively
		like moduli of local systems on a Riemann surface, and moduli of Higgs bundles.   For instance, $\Bet$ is a moduli space of 
		microlocal sheaves, which generalize local systems, and $\Dol$ carries the structure of a complex integrable system. 
		
		We show the Euler characteristics of these varieties count spanning subtrees of the graph, and the point-count over a finite field
		for $\Bet$ is a generating polynomial for spanning subgraphs.  This polynomial satisfies a deletion-contraction relation,
		which we lift to a deletion-contraction exact triangle for the cohomology of $\Bet$.   There is a corresponding
		triangle for $\Dol$.  
		
		Finally, we prove $\Bet$ and $\Dol$ are diffeomorphic, that the diffeomorphism carries the
		weight filtration on the cohomology of $\Bet$ to the perverse Leray filtration on the 
		cohomology of $\Dol$, and that all
		these structures are compatible with the deletion-contraction triangles. 
	\end{abstract}
	
	\maketitle
	
	\newpage
	
	\singlespacing
	\tableofcontents
	
	\newpage

	\section{Introduction}
	
	\vspace{2mm}
	
	\begin{quote} 
		$\Bet$, the locus $ \C^2 \setminus \{1 + xy = 0\}$, the first nontrivial multiplicative quiver variety, the moduli of 
		microlocal sheaves on a singular Lagrangian torus.  
	\end{quote} 
	
	\vspace{2mm}
	
	\begin{quote}
		$\Dol$, a neighborhood of the nodal elliptic curve in its versal deformation, the simplest degeneration in a 
		complex integrable system, a local model for 4-dimensional hyperk\"ahler geometry.  
	\end{quote}
	
	\vspace{2mm}

	Each of the above spaces is the progenitor of 
	a family, with one member 
	for each $\Gamma$ a connected multigraph with loops.  The initial examples
	are those associated to the graph $\boing$. 
	These families were introduced by Hausel and Proudfoot \cite{HP}, where they 
	observed that $\Bet$ and $\Dol$ are analogous to moduli of local systems and the moduli of Higgs bundles on an algebraic curve, respectively, and conjectured
	the existence of diffeomorphisms $\Bet(\Gamma) \cong \Dol(\Gamma)$, analogous to the nonabelian
	Hodge correspondence.

	That correspondence \cite{Sim, Sim-loc, Sim-loc2}
	relates three perspectives on nonabelian Lie-group valued 
	cohomology: locally constant sheaves (Betti), bundles with
	connection (de Rham), and Higgs bundles (Dolbeault).  We are most interested in the case where the underlying
	variety is an algebraic curve $C$, and in the (non-complex-analytic!) diffeomorphism between the 
	moduli $\mathcal{M}_{B}(C,n)$ of rank $n$ locally constant simple sheaves, 
	i.e. simple representations $\pi_1(C) \to GL_n(C)$, and the moduli  $\mathcal{M}_{\overline\partial}(C,n)$ of 
	stable rank $n$ Higgs bundles.  The Higgs bundle moduli carries Hitchin's integrable system,  
	$H \colon \mathcal{M}_{\overline\partial}(C,n) \to \A$, where $\A$ parameterizes $n$-multisections of $T^*C$ (`spectral curves')  \cite{Hit, Hit2}.  
	The fiber over the point corresponding to a smooth spectral curve $\Sigma$ is its Jacobian $J(\Sigma)$.  
	
	\vspace{2mm}
	
	We believe $\Bet(\Gamma)$ and $\Dol(\Gamma)$ are in some sense {\em microlocal} versions of the 
	nonabelian cohomology spaces, and give some ideas in this direction in Remark \ref{rem: microlocal}.  
	In any case, Hausel and Proudfoot conjectured the following relationship between them, which 
	we establish:

	\begin{theorem} \label{thm:introdolbet} (\ref{thm:basicnaht})
		For any graph $\Gamma$, there is a canonical homotopy equivalence induced by 
		a non-canonical open embedding $\Dol(\Gamma) \subset \Bet(\Gamma)$.
	\end{theorem}
	
	\begin{remark} \label{rem: lamentations}
		The results of the present article are cohomological in nature, so do not depend
		on the precise geometric details of the embedding above; in fact we construct
		a family of such embeddings depending on various parameters, and any of these 
		may be used.  For purely motivational purposes (to sharpen the analogy with 
		the nonabelian Hodge correspondence), we note the following possibilities. 
		$\Dol(\Gamma)$ carries a (not complete) hyperk\"ahler metric, and 
		in particular a twistor sphere of complex structures. 
		Using Theorem \ref{li embedding}, it is possible to choose the
		embedding $\Dol(\Gamma) \subset \Bet(\Gamma)$ such that the complex structure on 
		$\Bet(\Gamma)$ restricts to one of the complex structures on $\Dol(\Gamma)$, 
		different from the one in which $\Dol$ admits a
		holomorphic integrable system. 
		It is also possible 
		to deform the open embedding to a (not especially natural) diffeomorphism, by using Proposition \ref{prop:betretract}. 
	\end{remark} 
	
	Recall that the Grothendieck ring of varieties is formally generated by varieties, subject 
	to the relation $|X| = |X \setminus Y| + |Y|$ when $Y$ is a subvariety of $X$, and
	we write $|X|$, $|Y|$, etc. to denote the class in the Grothendieck ring.  The ring 
	structure descends from the cartesian product of varieties.

	\begin{theorem} \label{thm:motivicdc}
		The following identities hold in the Grothendieck ring of varieties: 
		\begin{equation} \label{eq: motive of boing} 
			|\Bet(\boing)| = \L^2 - \L + 1
		\end{equation} 
		
		\begin{equation} \label{eq:motivicdc} 
			|\Bet(\Gamma)| = \begin{cases}  |\Bet(\Gamma / e)| & \mbox{$e$ is a bridge} \\ 
				|\Bet(\Gamma \setminus e)| \cdot |\Bet(\boing)| & \mbox{$e$ is a loop} \\ 
				|\Bet(\Gamma \setminus e)| \cdot \L +  |\Bet(\Gamma / e)| & \mbox{otherwise}  \end{cases}
		\end{equation}
		
		\begin{equation} \label{eq:mbetmotive} |\Bet(\Gamma)| = \sum_{\Gamma' \in Span(\Gamma)} (\L - 1)^{2 b_1(\Gamma')} \L^{b_1(\Gamma) - b_1(\Gamma')}\end{equation} 
		In particular, the Euler characteristic of $\Bet(\Gamma)$ is the number of spanning subtrees of $\Gamma$. 
	\end{theorem} 
	
	Here, \eqref{eq: motive of boing} is elementary, and \eqref{eq:mbetmotive} follows from \eqref{eq: motive of boing}, \eqref{eq:motivicdc}, and 
	general facts about the Tutte polynomial \cite[Chap. 10, Thm. 2]{Bol}.  Theorem \ref{thm:motivicdc}  is proven in Section \ref{sec: motivic interlude}.

	\vspace{2mm}
	Let us recall the relationship between the Grothendieck ring of varieties and cohomology.  
	The cohomology of any algebraic variety $X$ carries two filtrations; a decreasing `Hodge' filtration
	and an increasing `weight' filtration \cite{D1, D2, D3}.  
	We are interested here in the latter; its $i$'th step on the $j$'th cohomology group 
	is denoted $W_i H^j(X)$, and the associated graded spaces are denoted by $gr^W_i H^j(X)$.  One records these dimensions 
	in the mixed Poincar\'e polynomial:
	$$P^X(q, t) = \sum_{i,j} q^i t^j \dim gr^W_i H^j(X)$$ 
	Under specializing $q \to 1$, one recovers the usual Poincar\'e polynomial.  There is an analogous 
	construction with compactly supported cohomology, $P^X_c(q, t)$.  When $X$ is smooth, one has
	the Poincar\'e duality $P^X_c(q, t) = (qt^2)^{\dim X} P^X(q^{-1}, t^{-1})$.  In any case, 
	$P^X_c(q, -1)$ factors through the Grothendieck ring. 
	
	In \cite{HRV}, the quantity $P^X_c(q, -1)$ was determined for the character varieties $\mathcal{M}_B(C)$, and also for
	twisted versions corresponding to Higgs 
	bundles of nonzero degree.\footnote{\cite{HRV} does not compute classes in the Grothendieck ring, 
		and in fact these remain unknown.  Instead, they determine $P^X_c(q, -1)$ by counting points over finite fields.} 
	These explicit formulas, together with the complete description of 
	the cohomology for $GL(2)$, led to a conjectural formula for the full mixed Hodge polynomial.  
	Inpsection of the conjectural formula suggested certain curious properties of the cohomology \cite{HRV}.  
	
	The remarkable ``$P = W$'' conjecture of \cite{dCHM} was proposed to explain these
	curiousities.  The setup is as follows.
	Given any 
	map of algebraic varieties $f\colon X \to A$, there is a filtration on $\H^{\bullet}(X) \cong \H^{\bullet}(A, Rf_* \C_X)$ arising from
	truncation of $Rf_* \C_X$ in the (middle) perverse $t$-structure on $A$; this is termed the perverse 
	Leray filtration.  The $P = W$ conjecture asserts that under Simpson's correspondence, the weight filtration on the character variety goes to the perverse Leray filtration
	associated to Hitchin's integrable system on the moduli of Higgs bundles.  
	This conjecture was established in \cite{dCHM} in the $GL(2)$ case, and very recently for any rank 
	on a genus 2 curve \cite{dCMS}.  One of its original motivations, the `curious Hard Lefschetz' conjecture, 
	is now established \cite{Mel}.  Some additional special cases have been
	verified \cite{SZ, Szi}, and some tests of structural predictions verified \cite{dCM}.  
	A certain limit of the conjecture appears to be related to a comparison of limiting behavior
	of the Hitchin fibration with the geometry of the boundary complex of the character variety \cite{Sim-geom}. 
	Relationships between
	perverse and weight filtrations have also been found in other settings of hyperk\"ahler geometry \cite{dCHM2, Harder, HLSY}.  
	In particular, the 4-real-dimensional examples of the spaces under investigation here were studied in \cite{Z}. 
	A similar sounding (but at present
	not directly related) statement has been found in homological mirror symmetry \cite{HKP}.  The original
	conjecture remains open in the general case. It is unclear what is the natural setting or generality for this conjecture.  
	
	In our setting, the space $\Dol(\Gamma)$ is the central fiber of a certain natural family; we may correspondingly
	equip its cohomology with a perverse filtration.  Meanwhile, $\Bet(\Gamma)$ carries a weight filtration, due to being
	an algebraic variety.  We will prove: 
	
	\begin{theorem} \label{thm:intropw} (\ref{thm:wequals2p})
		The homotopy equivalence $\Dol(\Gamma) \hookrightarrow \Bet(\Gamma)$ carries the
		the weight filtration on $\H^{\bullet}(\Bet(\Gamma))$ to (twice) the perverse Leray filtration
		on $\H^{\bullet}(\Dol(\Gamma))$. 
	\end{theorem} 
	
	To our knowledge, all other known instances of $P=W$ are proven by computing both sides; 
	e.g. by finding generators and relations of the cohomology ring, matching filtrations on generators,
	and proving multiplicativity of the perverse filtration \cite{dCHM, dCMS}.  
	By contrast, we do not know generators for the cohomology ring, much less relations.  
	While we know $P^{\Bet(\Gamma)}_c(q, -1)$, we
	do not have even a conjecture for $P^{\Bet(\Gamma)}_c(q, t)$, or of the analogous perverse 
	Poincar\'e polynomial of $\Dol(\Gamma)$.  
	
	Instead we proceed by 
	upgrading the deletion-contraction relations of Theorem \ref{thm:motivicdc} to deletion-contraction exact triangles.  
	By the end we will have shown: 
	
	\begin{theorem} \label{thm:introdc} (\ref{thm:weightdcs}, \ref{td:doldcs}, \ref{Hodgeintertwinesdelcon}, \ref{cor:dolbdelconstricpreserves}) 
		For any edge $e$ which is neither a loop nor a bridge, there are deletion-contraction long-exact sequences, intertwined by pullback along $\Dol(\Gamma) \hookrightarrow \Bet(\Gamma)$.  
		\begin{equation*} 
			\begin{tikzcd} 
				\ \arrow[r] &  \ \H^{\bullet-2}(\Bet(\Gamma \setminus e) \homsep \Q)(-1) \arrow[d]
				\arrow[r]  &  \H^{\bullet}(\Bet(\Gamma)\homsep \Q) \arrow[d] \arrow[r]   &  \H^{\bullet}( \Bet(\Gamma / e)\homsep \Q) \arrow[r]  \arrow[d]
				& \  \\
				\ \arrow[r]  &  \H^{\bullet-2}(\Dol(\Gamma \setminus e) \homsep \Q) \{-1\} \arrow[r] 
				& \H^{\bullet}(\Dol(\Gamma) \homsep \Q) \arrow[r] 
				& \H^{\bullet}(\Dol(\Gamma / e) \homsep \Q) \arrow[r]  & \ \\
			\end{tikzcd}
		\end{equation*}
		The sequences are strictly compatible with the weight and perverse Leray filtrations, respectively.  The $(-1)$ and $\{-1\}$ 
		indicate shifts of these filtrations. 
	\end{theorem} 
	
	The existence of the intertwined long exact sequences is nontrivial, but in some sense it is proven by pure thought, 
	using the excision triangle 
	on the top,  the nearby-vanishing triangle at the bottom, and geometric arguments for commutativity of the diagram.  
	One would like to conclude compatibility
	with filtrations by induction on the size of the graph.  This does not immediately work, for two reasons. 
	The first: we do not know a pure thought argument that 
	the Dolbeault sequence is {\em strictly} compatible with the perverse Leray filtration; in fact, we will only learn
	this at the very end of the paper.  The second: even had we known this, there is the following
	difficulty: consider two short exact sequences of filtered vector spaces, maps strictly compatible with the filtration. 
	Suppose given an isomorphism of the underlying short exact sequences, which respects the filtration save on 
	the middle term.  Must it respect the filtrations on the middle term?  Alas, no. 
	
	To deal
	with these difficulties we introduce yet a third filtration, which is defined only in terms
	of the deletion maps.  
	
	\begin{definition} (Deletion filtration) \label{def:delfiltration}
		Let $\mathrm{Graph}^\circ$ be the category whose objects are connected oriented graphs and whose morphisms are inclusions
		whose complement contains no loop.  Let $grab$ be the category whose objects are graded abelian groups, and whose morphisms
		are arbitrary (not graded) abelian group morphisms.  Let
		$$A^\bullet: \mathrm{Graph}^{\circ} \to grab$$
		be a covariant functor such that 
		$A^\bullet(\Gamma' \to \Gamma)$ has degree 
		$2|\Gamma| - 2 |\Gamma'|$, i.e., 
		the corresponding map $A^{\bullet}(\Gamma')[2|\Gamma'|] \to A^{\bullet}(\Gamma)[2|\Gamma|]$ has degree zero.  
		
		If $\Gamma$ has only loops and bridges, then we define 
		$0 = \mathrm{D}_{i-1}(A^i(\Gamma)) \subset \mathrm{D}_{i} (A^i (\Gamma)) = A^i(\Gamma)$. 
		Otherwise, we set
		$$\mathrm{D}_{i-k} A^i(\Gamma) = \mathrm{Span}(\{\mathrm{image}(A^{i-2k}(\Gamma'))\, | \, |\Gamma \setminus \Gamma'| = k\})$$
		Here the span is over all maps $\Gamma' \to \Gamma$. 
	\end{definition}
	It is immediate from the definition that $\mathrm{D}_\bullet$ is (not necessarily strictly) preserved by all maps 
	$A(\Gamma \subset \Gamma'): A(\Gamma')\{|\Gamma'|\} \to A(\Gamma)\{|\Gamma|\}$
	where $\{ \cdot \}$ indicates a shift of the filtration.  It is also evident that it is the minimal
	such filtration.

	Once we have shown the deletion maps act identically on the cohomology of the $\Bet(\Gamma)$ and $\Dol(\Gamma)$,
	it follows that the corresponding deletion filtrations must agree on the $\Bet$ and $\Dol$ sides.  Thus it remains to show
	the deletion filtration agrees with the weight and perverse filtrations.  This proves to be rather involved; our argument 
	depends on introducing a combinatorial model in which the third filtration is manifest, and then arguing on
	each side that this combinatorial model can be realized by some (rather different on the two sides) geometric
	construction. 
	
	\subsection{Outline}  
	We begin in Section \ref{sec:CKS} by recalling from \cite{MSV} the combinatorial description of a certain complex $\CKS(\Gamma)$ associated to any graph.  This complex 
	will turn out to have geometric interpretations both as the cohomology of $\Bet(\Gamma)$, and of $\Dol(\Gamma)$.  
	Nevertheless in Section \ref{sec:CKS} we confine ourselves to a purely combinatorial discussion.  We construct explicitly
	the deletion-contraction filtration exact sequence, and note some of its properties.  In particular, we observe that
	the deletion-contraction sequences themselves induce a filtration on the cohomology.  A key point about $\CKS(\Gamma)$
	is that the resulting filtration is easy to describe.  
	
	In Section \ref{sec:moment} we adapt the formalism of moment maps and symplectic reduction to situations when
	no symplectic structure is present.  Symplectic reduction applies in the situation of a group $\G$ 
	acting on a symplectic manifold with moment map $\mu\colon X \to \frak{g}^*$ (which, together with the symplectic
	structure, encodes the group action).  Here we consider arbitrary spaces $X$ with an action of a group $\G$ preserving a 
	map $\mu\colon X \to \mom$ to an abelian group $\mom$; we call such things $(\G, \mom)$-spaces.  The map $\mu$ 
	in no way encodes the group action. 
	
	Nevertheless, given a $(\G, \mom)$-space $X$, we can define its reduction $X \sslash_{\eta \in \mom} \G \colonequals \mu^{-1}(\eta)/\G$. Given two $(\G, \mom)$-spaces $X,Y$, we can form a product $(\G, \mom)$-space $X \bullet Y$. Similarly, we can form the quotient $X \star Y = X \bullet Y \sslash \G$. This construction is `functorial', meaning that a map $Y' \to Y$ induces a map $X \star Y' \to X \star Y$. 
	
	In Section \ref{sec:constructmoduli} we build spaces from graphs.  From any $(\G, \mom)$-space $X$ 
	together with a graph $\Gamma$ and an element $\eta_v \in \mom$ for each vertex of $\Gamma$, we construct 
	a space $X(\Gamma\homsep \eta)$ in Section \ref{subsec:spacesfromgraphs}.  In particular, $X$ is recovered from the one-loop graph: $X = X(\boing)$.  
	
	Given an edge $e$ in $\Gamma$, we  form new graphs $\Gamma / e$ and $\Gamma \setminus e$ by contracting (resp.\ deleting) $e$.  Our main tools for studying $X(\Gamma)$ are the two relations of the form
	$X(\Gamma / e) \star X = X(\Gamma)$ and $X(\Gamma / e) \star \mathbf{point} = X(\Gamma \setminus e)$.

	In Section \ref{sec:bet} we turn our attention to the spaces $\Bet(\Gamma)$.  They are built from the basic space 
	$\Bet = \C^2 \setminus \{xy + 1 = 0\}$.  
	Using functoriality of the $\star$ product, we turn properties of $\Bet$ into properties of $\Bet(\Gamma)$. In Section \ref{sec:bettidelcon}, we use this to obtain the  {\em Betti deletion-contraction sequence}
	$$\to \H^{\bullet-2}(\Bet(\Gamma \setminus e), \Q(-1)) \to \H^{\bullet}(\Bet(\Gamma), \Q) \to \H^{\bullet}(\Bet(\Gamma / e), \Q) \to $$  The key geometric construction is an embedding of a line bundle over $\Bet(\Gamma \setminus e)$ into $\Bet(\Gamma)$, with complement $\Bet(\Gamma / e)$. The resulting long exact sequence of a pair is our deletion-contraction sequence.  The same geometry immediately implies Equation \eqref{eq:motivicdc} above.
	
	The deletion maps equip  the cohomology of $\Bet(\Gamma)$ with a deletion filtration.  
	The deletion maps are induced by maps of algebraic varieties, hence respect the weight filtration;  minimality
	of the deletion filtration implies it is bounded by the weight filtration.  In fact, they are 
	equal; to prove this we construct an explicit complex of differential forms, which on the one hand is sensitive to the
	weight filtration, and on the other, can be identified with the complex $\CKS(\Gamma)$, compatibly with
	deletion-contraction.  The deletion filtration is explicit on $\CKS(\Gamma)$, allowing us to conclude.

	In Section \ref{sec:dol}, we turn to the Dolbeault space $\Dol(\Gamma)$.  The special case $\Dol = \Dol(\boing)$ is
	the Tate curve, and the more general spaces are degenerating families of abelian varieties built as subquotients of powers
	of $\Dol(\boing)$.   We study these spaces in families; in particular, there is a family of spaces over the unit disk $\D \subset \C$ whose general fiber is $\Dol(\Gamma / e)$, whose special fiber is homotopic to $\Dol(\Gamma)$, and whose singular locus is homotopic to $\Dol(\Gamma \setminus e)$.  The nearby-vanishing triangle gives rise, ultimately, to the deletion-contraction
	sequence in this setting.  
	
	As with the moduli of Higgs bundles, the spaces $\Dol(\Gamma)$ have the structure of complex analytic integrable systems. 
	We explore this structure further in Section \ref{sec:system}, in particular describing the fibers and characterizing the monodromy. 
	We need these results to show  the deletion maps preserve the perverse filtration, hence that the deletion filtration is bounded
	by the perverse filtration.
	Additionally, borrowing a calculation of \cite{MSV}, we show that $\CKS(\Gamma)$ also computes the cohomology of the spaces 
	$\Dol(\Gamma)$, compatibly with the perverse filtrations. 
	
	Finally in Section \ref{sec:hodge}, we begin comparing $\Bet$ and $\Dol$.  First we construct a smooth embedding
	and homotopy equivalence 
	between the basic spaces, $\Dol \subset \Bet$.  Due to the similarity of the constructions of these spaces, 
	this induces a similar inclusion $\Dol(\Gamma) \subset \Bet(\Gamma)$, thus proving 
	Theorem \ref{thm:introdolbet} (\ref{thm:basicnaht}).  
	
	We show in Section \ref{sec:hodgedeletionintertwines} that the deletion maps are intertwined by $\Dol(\Gamma) \to \Bet(\Gamma)$.  
	The key geometric input is a relation between the subspace used in the long exact sequence of a pair (on the Betti side) and the vanishing thimble for the degenerating family (on the Dolbeault side).  It follows immediately that the Betti and Dolbeault deletion filtrations are identified. In particular, dimensions of the associated graded pieces of the Dolbeault deletion filtration equal those of the $\CKS$-filtration. Then since Dolbeault deletion filtration is bounded by the perverse Leray filtration, but both these
	have associated 
	graded dimensions matching that of the $\CKS$ filtration, we conclude that in fact these filtrations must be equal. 
	Having identified the deletion filtrations with the weight and perverse Leray filtrations on the respective sides, we 
	deduce Theorem \ref{thm:intropw} (\ref{thm:wequals2p}).  Some further geometric considerations give 
	the full intertwining of Theorem \ref{thm:introdc}.

	\subsection{Some additional remarks}
	
	\begin{remarque} \label{rem:moregeneral}
		More generally, the spaces $\Bet(\Gamma)$ and $\Dol(\Gamma)$ can be (and were originally \cite{HP}) 
		defined with an arbitrary integer matrix in place of the adjacency matrix of the graph; in this generality
		they have orbifold points (see Remark \ref{rem:matroid}.)  Deletion-contraction relations have a well known generalization 
		to this matroidal setting.  We expect that in fact
		all the results of the paper generalize as well, with proofs complicated only by increased bookkeeping. 
	\end{remarque}
	
	\begin{remarque} \label{rem: microlocal} 
		Here we will explain that $\Bet(\Gamma)$ and $\Dol(\Gamma)$ in some sense model
		the nonabelian cohomology spaces `near' a nodal spectral curve with dual graph $\Gamma$, 
		and that this, along with our results, hints at the existence of a `microlocal' version of nonabelian
		Hodge theory. 
		
		By microlocal, we mean as always 
		`locally in the cotangent bundle', i.e. locally around the spectral curve $\Sigma \subset T^*C$, and correspondingly 
		locally around the corresponding Hitchin fiber, itself a multisection of a cotangent bundle 
		$T^*Bun_{GL_n}(C)$.  
		
		Consider a smooth spectral curve.  
		A neighborhood of $J(\Sigma)$ inside $\mathcal{M}_{\overline\partial}(C,n)$ will be diffeomorphic (and in fact symplectomorphic) to a neighborhood of 
		the zero section of $T^* J(\Sigma) = \mathcal{M}_{\overline\partial}(\Sigma,1)$, which in turn is -- by the abelian case of the nonabelian Hodge correspondence --
		diffeomorphic to $\mathcal{M}_{B}(\Sigma,1)$.  That is, there is a diffeomorphism between a neighborhood of 
		the Dolbeault data for $\Sigma$ and the moduli space of Betti data on $\Sigma$.  
		
		We turn to singular spectral curves.  Assuming $\Sigma$ is a reduced, possibly reducible, curve, the Hitchin fiber is a compactification of its Jacobian; we denote it $\overline{J}(\Sigma)$.  
		We will be interested in the cohomology of this singular fiber. 
		Denoting by $\widetilde{\Sigma}$ the normalization of the curve, it is known that 
		$H^*(\overline{J}(\Sigma)) \cong H^*(J(\widetilde{\Sigma})) \otimes D(\Sigma)$ for some graded vector space $D(\Sigma)$ 
		depending only on the singularities of $\Sigma$.\footnote{When $\Sigma$ not irreducible, the compactification of 
			the Jacobian depends on the choice of a stability condition.  
			However, it follows from \cite{MSV} that $H^*(\overline{J}(\Sigma))$ is in fact independent of a generic such choice,
			and genericity is known to follow from smoothness of the total space of the Hitchin fibration.}
		
		Here we focus on the simplest case, when $\Sigma$ has only nodes.  Let $\Gamma_\Sigma$ be the dual 
		graph: it has vertices for the irreducible components of $\Sigma$, and edges for the nodes. 
		Let us explain how the $\Dol(\Gamma_\Sigma)$ is similar to a neighborhood in the Hitchin system of the fiber over $[\Sigma]$. 
		The space $\Dol(\Gamma_\Sigma)$ is smooth; $\Dol(\Gamma_\Sigma) \times J(\widetilde{\Sigma})$ has the same dimension as $\mathcal{M}_{\overline\partial}(C,n)$, 
		and it follows its construction that $\Dol(\Gamma_\Sigma)$  carries the structure of an integrable system.  
		The data defining $\Dol(\Gamma)$ did not depend on complex structure parameters, so it cannot 
		be expected that the central fiber 
		of $\Dol(\Gamma)$ is analytically related to the Hitchin fiber $\overline{J}(\Sigma)$.  In fact, even when $\Sigma$
		has rational components, the corresponding central fibers need not be homeomorphic, and even if they are, the corresponding integrable systems need not be fiberwise homeomorphic.  Nevertheless, we
		will see $H^*(\Dol(\Gamma_\Sigma)) \cong D(\Sigma)$, in fact compatibly with the perverse Leray filtration (see Remark \ref{rem:dolmeaning}).  In this sense, $\Dol(\Gamma_\Sigma)$ is a model (or replacement) for the local topology in
		$\mathcal{M}_{\overline\partial}(C,n)$ around $\overline{J}(\Sigma)$. 
		
		There is also a sense in which $\Bet(\Gamma_\Sigma)$ captures `Betti information near $\Sigma$'.  More precisely, one can view $\Sigma$ as a (singular)
		Lagrangian and study the moduli space $\mathcal{M}_B(\Sigma,1)$ of rank one microlocal sheaves on $\Sigma$.\footnote{Equivalently \cite[Sec. 6.2]{GPS3}, of the wrapped Fukaya category of a completion of a neigborhood of $\Sigma$.
			From this point of view, one sees an embedding $\mathcal{M}_B(\Sigma,1) \to \mathcal{M}_{B}(C,n)$ is induced by 
			pullback of pseudo-perfect modules under a non-exact Viterbo restriction, modulo convergence issues.  The case of smooth spectral curve is  \cite{GMN2}.  What is by no means clear is if or why $\overline{J}(\Sigma)$ lies in the image, 
			let alone why it should be a deformation retract thereof. We will not use or discuss this further here.}  
		Were $\Sigma$ smooth, 
		this would be the space of rank one local systems we encountered above.  In the nodal case, 
		moduli of microlocal sheaves is shown in \cite{BK} to match certain multiplicative Nakajima varieties \cite{CBS, Y}; 
		comparing the results there to the definitions here, it is immediate that
		there is an algebraic isomorphism 
		$$\mathcal{M}_B(\Sigma,1) \cong \mathcal{M}_B(\widetilde{\Sigma},1) \times \Bet(\Gamma_\Sigma)$$
		
		In fact the relationship between microlocal sheaves on a spectral cover and the neighborhood of 
		the corresponding Hitchin fiber should hold in some greater generality.   In particular, 
		at least for spectral curves the links of whose singularities are torus knots, a similar statement can be tortured
		out of  the identification in \cite{STWZ} of moduli of Stokes data as moduli of sheaves microsupported along 
		a Legendrian, plus the nonabelian Hodge correspondence in the presence of irregular singularities \cite{BB}.
		As explained in the introduction of \cite{STZ}, a comparison of the numerics of that article with those 
		of \cite{OS, GORS} reveals a faint shadow of a `$P=W$' phenomenon here as well. 
		
		On the other hand, 
		while we construct an embedding $\Dol(\Gamma) \subset \Bet(\Gamma)$, we do not know a 
		category of which the former space is a moduli space, much less a functor between categories
		inducing this map.  It would be preferable to have such a category and functor.  Relatedly, 
		we have described how the spaces $\Dol(\Gamma)$ and $\Bet(\Gamma)$ are cohomologically related to Hitchin fibers, 
		but not given maps of spaces, much less of categories. 
		
		We end with the question: is there a microlocal nonabelian Hodge theory? 
	\end{remarque}
	
	\begin{remarque}
		Recall from \cite{SW} and subsequent developments 
		that if one considers $N=2$ super Yang-Mills for $U(n)$ with $g$ adjoint matter fields, 
		then the vacua in $\R^4$ form the base of a Hitchin system corresponding to the moduli of Higgs bundles 
		over a base curve of genus $g$.  
		At low energy, in a vacuum corresponding to a spectral curve with (for convenience) rational components, the theory
		is described by an abelian gauge theory with gauge fields corresponding to the components of the spectral
		curve, and bifundamentals or adjoints corresponding to the nodes.  That is, it corresponds to the dual graph
		of the curve.  We expect there should be some physical account of why the cohomology of the corresponding
		multiplicative quiver variety is identical to the cohomology of the Hitchin fiber, and more optimistically, why 
		this should identify weight and perverse Leray filtrations (as we have mathematically proven is the case). 
		
		There is a string-theoretic account of why the perverse filtration on the cohomology of the Higgs moduli space should lead
		to the bigraded numbers guessed by \cite{HRV} for the weight filtration on the character variety (see \cite{CDP1, CDP2, CDDP, DDP, Dia, CDDNP}). 
		It is not immediately clear how this relates to the above notions, but it would be interesting to make such a
		connection.  In particular, unlike \cite{HRV}, 
		we have not been able to compute (or guess) the mixed Poincar\'e polynomials of the 
		multiplicative hypertoric varieties.  
	\end{remarque}

	\begin{remarque}
		The embedding $\Dol(\Gamma) \subset \Bet(\Gamma)$ has the flavor of a hyperk\"ahler rotation.  
		In particular it carries the central fiber of the integrable
		system $\Dol(\Gamma)$ to a non-holomorphic Lagrangian subvariety of $\Bet(\Gamma)$, which should be the 
		Lagrangian skeleton of an appropriate Weinstein structure.  This fact, which we do not prove here, suggests a way to calculate 
		the Fukaya category of $\Bet(\Gamma)$, using the approach of \cite{K, N, GPS1, GPS2, GPS3, GS}. 
		This idea is explored in \cite{GMcW}, building on calculations of \cite{McW}. 
	\end{remarque}

	\begin{remarque}
		A shadow of Theorem \ref{thm:intropw} can be seen by comparing Equation \eqref{eq:mbetmotive} above to Theorem 1.1 of \cite{MSV}, 
		after specializing $\L \to 1$ in the latter.  We do not know what parameter should be introduced in our formula to recover
		the $\L$ of \cite{MSV}; this corresponds to asking how to characterize the filtration on the {\em Betti} moduli space
		which corresponds to the weight filtration on the {\em Dolbeault} moduli space.  This question does not arise 
		in the setting of the original P=W conjecture, because in that situation, the cohomology of the Dolbeault space is 
		pure, i.e. the weight filtration arises from the cohomological grading.  In the present case, the (central fiber of the) 
		Dolbeault space does not have pure cohomology.  
	\end{remarque}

	\begin{remarque}
		The deletion-contraction relation enjoys various connections with the skein relation of knot theory; it may be expected
		that deletion-contraction exact sequences enjoy similar connections with the skein exact sequences in knot homology
		theories such as \cite{Kho}.  Indeed, this is true by construction in various extant categorifications of the Tutte polynomial 
		and its specializations \cite{ER, HR, HR2, Sto, Sto2}, though we do not know how these constructions relate to $\H^{\bullet}(\Bet(\Gamma))$.
		
		In this context we recall the relation between knot invariants and the perverse polynomial of Hitchin fibers \cite{ObS, Mau}; 
		and its conjectural lift to the cohomological level \cite{ORS, GORS}.  
	\end{remarque}
	
	\vspace{2mm}
	{\bf Acknowledgements.}
	We thank the following people for helpful conversations: 
	Valery Alexeev, Phil Engel, Michael Groechenig, Tam\'as Hausel, Raju Krishnamoorthy,  Yang Li, Conan Leung, Luca Migliorini, Nick Proudfoot, Michael Thaddeus.  
	
	Z.\ D.\ was supported by NSF grant no.~0932078~000 while in residence
	at the Mathematical Sciences Research Institute in Berkeley, California during Fall 2014, and by an Australian Research Council DECRA grant DE170101128 from 2017.
	M.\ M.\ was supported by the Advanced Grant “Arithmetic and Physics of Higgs moduli spaces” No. 320593 of the European Research Council. while at the {\'E}cole Polytechnique F{\'e}d{\'e}rale de Lausanne, and by the Natural Sciences and Engineering Research Council while at the University of Toronto. Part of this work was completed while M.\ M.\ was at the Massachusetts Institute of Technology and the Hausdorff Institute of Mathematics. 
	V.\ S.\ was supported by the NSF grants DMS-1406871 and CAREER DMS-1654545, 
	and a Sloan fellowship.

	\section{Graph conventions and some linear algebra} \label{sec: graph conventions}
	For us, a {\em graph} will always mean a finite simplicial set with only 0- and 1-simplices, 
	or in other words, what is sometimes called an `oriented multigraph with loops'.  
	That is, we have the data 
	of a finite set of edges $E(\Gamma)$, a finite set of vertices $V(\Gamma)$, and 
	two maps $h, t: E(\Gamma) \to V(\Gamma)$. 
	
	As $\Gamma$ is a simplicial set, given a (for us always abelian) group $A$, we may form the simplicial chains and cochains.  
	There are (by definition) canonical isomorphisms
	$$C_0(\Gamma\homsep A) \simeq A^{\ver} \simeq C^0(\Gamma\homsep A)$$ and $$C_1(\Gamma\homsep A) \simeq A^{\ed} \simeq C^1(\Gamma\homsep A).$$
	We write $\overline{A^{\ver}}$ or $\redC^0(\Gamma\homsep A)$ for the quotient by the subgroup of constant functions, 
	and $\redC_0(\Gamma\homsep A)$ for the subgroup of chains summing to zero via the group law of $A$. 
	
	We denote the differentials by  $d_{\Gamma} \colon C_1(\Gamma\homsep A) \to C_0(\Gamma\homsep A)$ 
	and $d^*_{\Gamma} \colon C^0(\Gamma\homsep A) \to C^1(\Gamma\homsep A)$, respectively.  
	The formula for $d_\Gamma$ is: 
	\begin{eqnarray*} 
		d_\Gamma\colon A^{\ed} & \to & A^{\ver} \\
		\, e & \mapsto & h(e) - t(e)
	\end{eqnarray*}
	We often use the inclusion $\H_1(\Gamma, R) = \mathrm{ker}(d_\Gamma) \subset A^{\ed}$. 
	
	We write: 
	\begin{equation} \label{double h} 
		\mathbb{H}(\Gamma, A) \colonequals \H_1(\Gamma, A) \oplus \H^1(\Gamma, A)
	\end{equation}

	Note that all chain and cochain groups, and all homology and cohomology groups, are canonically independent
	of the choice of orientations on edges of $\Gamma$. 
	
	\vspace{2mm}
	
	For $\eta \in C_0(\Gamma, A) = A^{\ver}$, 
	we will write
	\begin{equation} \label{h1 torsor} \H_1(\Gamma \homsep A)_{\eta} := d_{\Gamma}^{-1}(\eta) \end{equation}  
	When nonempty,  $\H_1(\Gamma \homsep A)_{\eta}$ is the translation of $\H_1(\Gamma \homsep A) \subset C_1(\Gamma, A)$ by any $d_{\Gamma}$-preimage 
	of $\eta$, and is thus a torsor for $\H_1(\Gamma \homsep A) = \H_1(\Gamma \homsep A)_{0}$.

	\begin{definition} \label{defsimplemomentvalue}
		For $a \in A^{\ed}$, we write $\mathrm{Supp}(a) \subset \Gamma$ for the subgraph consisting of all vertices of $\Gamma$, 
		and those edges
		whose corresponding coordinate in $a$ is nonzero.  We say $\eta \in \redC_0(\Gamma, A)$ is {\em generic} 
		if for all $a \in \H_1(\Gamma \homsep A)_{\eta} \subset C_1(\Gamma, A) = A^{\ed}$, the graph $\mathrm{Supp}(a)$ is connected.  
	\end{definition}
	
	\begin{remark} \label{rem: disconnected generic}
		Note that if $\Gamma$ is disconnected, then $\eta$ is generic iff $\H_1(\Gamma \homsep A)_{\eta} = \emptyset$. 
	\end{remark}
	
	The geometric significance of Definition \ref{defsimplemomentvalue} becomes apparent in Prop. \ref{prop:smoothcond} below. 
	
	\begin{lemma} \label{genericvaluesexist}
		Suppose $A$ is a vector space of positive dimension over an infinite field. 
		Then there exist generic $\eta \in \redC_0(\Gamma, A)$. 
		The same holds when $A$ is a (commutative) Lie group of positive dimension. 
	\end{lemma}
	\begin{proof}
		Consider the complement $\Delta$ of the generic locus of $\redC_0(\Gamma \homsep A)$. It is the union over all disconnected subgraphs $\Gamma' \subset \Gamma$ of the image of $d_\Gamma : C_1(\Gamma' \homsep A) \to \redC_0(\Gamma, A)$. 
		
		Since $\redC_0(\Gamma \homsep A) = \redC_0(\Gamma' \homsep A)$, we can identify the latter map with $d_{\Gamma'} : C_1(\Gamma' \homsep A) \to \redC_0(\Gamma' \homsep A)$. The cokernel has dimension one less than the number of connected components of $\Gamma'$, and in particular the image is a proper subspace. It follows that $\Delta$ is a proper subset of $\redC_0(\Gamma, A)$.
		
		For commutative Lie groups, the result follows from the same argument on tangent spaces. 
	\end{proof}
	
	\vspace{2mm}
	
	\begin{lemma} \label{genericvaluesexist2}
		Suppose $A = A_1 \times A_2$ is a product of positive dimensional commutative Lie groups. Then there exist generic $\eta \in \redC_0(\Gamma, 0 \times A_2)$.  
	\end{lemma}
	\begin{proof}
		Given $\Gamma' \subset \Gamma$, the map $d_{\Gamma} : C_1(\Gamma',  A_1 \times A_2) \to \redC_0(\Gamma,  A_1 \times A_2)$ may be identified with the product map $C_1(\Gamma',  A_1) \times C_1(\Gamma', A_2) \to \redC_0(\Gamma,  A_1) \times \redC_0(\Gamma, A_2).$ We can now argue as in Lemma \ref{genericvaluesexist}. 
	\end{proof}

	Let us now consider deletion and contractions. 
	
	\begin{definition}
		Given a graph $\Gamma$, the graph $\Gamma \setminus e$ is defined by 
		deleting the edge $e$.   If $e$ is not a loop, then the graph $\Gamma / e$ is defined by ``contracting'' $e$, i.e., by removing
		it and collapsing $h(e)$ and $t(e)$ to a single vertex $v(e)$.
	\end{definition}
	
	There are evident maps (of simplicial sets) $\Gamma \setminus e \to \Gamma \to \Gamma / e$, the second being
	defined only when $\gamma$ is not a loop.  These 
	induce in particular maps $C_0(\Gamma \setminus e) \cong C_0(\Gamma) \to C_0(\Gamma / e)$. 
	We will use the following notation: 
	
	\begin{definition} \label{def:delconteta}
		For $\eta \in C_0(\Gamma, A)$:
		\begin{enumerate}
			\item We write $\eta \setminus e \in C_0(\Gamma \setminus e, A)$ for the preimage of $\eta$ under  
			$C_0(\Gamma, A) \xleftarrow{\sim} C_0(\Gamma \setminus e, A)$. 
			\item If $e$ is not a loop, we write $\eta / e$ for the image of $\eta$ under the map $C_0(\Gamma) \to C_0(\Gamma / e)$.
		\end{enumerate}
	\end{definition}
	
	\begin{lemma} \label{lem:genericityofdelcon} If $\eta$ is generic, then so are $\eta \setminus e$ and $\eta / e$.
	\end{lemma}
	\begin{proof}
		We have the commutative diagram: 
		\begin{equation} \label{eq: tc1} 
			\begin{tikzcd}
				\H_1(\Gamma, A)_{\eta} \arrow[r] & C_1(\Gamma, A)  \arrow[r, "d_{\Gamma}"] & C_0(\Gamma, A) \\
				\H_1(\Gamma \setminus e , A)_{\eta / e} \arrow[r] \arrow[hookrightarrow, u] & C_1(\Gamma \setminus e, A) \arrow[hookrightarrow, u] \arrow[r, "d_{\Gamma \setminus e}"] & C_0(\Gamma \setminus e, A)  \arrow[u, equals]
			\end{tikzcd}
		\end{equation}
		In particular, the support of any $a \in \H_1(\Gamma \setminus e , A)_{\eta / e}$ is equal to the support of its image in $\H_1(\Gamma, A)_{\eta}$. 
		
		If $e$ is not a loop, then we have the commutative diagram: 
		\begin{equation} \label{eq: tc2} 
			\begin{tikzcd} 
				\H_1(\Gamma, A)_{\eta} \arrow[d, "\sim"] \arrow[r] & C_1(\Gamma, A) \arrow[d] \arrow[r, "d_{\Gamma}"] & C_0(\Gamma, A) \arrow[d] \\
				\H_1(\Gamma /e , A)_{\eta / e} \arrow[r] & C_1(\Gamma / e, A) \arrow[r, "d_{\Gamma / e}"] & C_0(\Gamma / e, A) 
			\end{tikzcd}
		\end{equation}
		Thus, the support of any $a \in \H_1(\Gamma /e , A)_{\eta / e}$ is the image of the corresponding support of its preimage in $\H_1(\Gamma, A)_{\eta}$. 
	\end{proof}
	
	\begin{remark} \label{rem:delconJ}
		Suppose instead of starting with a single edge $e \in \ed$, we are given a subset $\subedges \subset \ed$. Definition \ref{def:delconteta} can be iterated to define $\eta \setminus \subedges$ and $\eta / \subedges$. Iterating Lemma \ref{lem:genericityofdelcon} shows that these are also generic. 
	\end{remark}
	
	Note that combining the left vertical morphisms of \eqref{eq: tc1} and \eqref{eq: tc2}, we obtain:  
	\begin{equation} \label{torsor delcon composition}
		\H_1(\Gamma \setminus e , A)_{\eta / e} \hookrightarrow \H_1(\Gamma, A)_{\eta} \cong \H_1(\Gamma /e , A)_{\eta / e}
	\end{equation}
	
	\begin{corollary} \label{cor: deletingbridges}
		Let $\eta$ be generic, and let $e$ be a bridge, i.e. suppose $\Gamma \setminus e$ is disconnected. Then $\H_1(\Gamma \setminus e, A)_{\eta \setminus e}$ is empty. 
	\end{corollary}
	\begin{proof}
		Follows from Remark \ref{rem: disconnected generic} and Lemma \ref{lem:genericityofdelcon}.
	\end{proof}

	For later use we describe some structures associated to a contracted edge.  
	Any edge $e$ determines a map 
	$$A \to A^{\ed} = C^1(\Gamma, E) \to H^1(\Gamma, E)$$
	(which is injective so long as $e$ is not a bridge).  Assuming $e$ is not a loop, 
	we define the map $\alpha_e: A \to \H^1(\Gamma / e, A)$ by demanding commutativity of the following diagram: 
	\begin{equation} \label{diag:momentrelationships0}
		\begin{tikzcd}
			\H^1(\Gamma \setminus e, A) & \arrow[l]   \H^1(\Gamma, A) & \arrow[l] A \\
			\operatorname{cok(\alpha_e)} \arrow[u, "\sim"] & \arrow[l] \arrow[u, "\sim"] \H^1(\Gamma / e, A) & \arrow[l, "\alpha_e"] A \arrow[u, equals]
		\end{tikzcd}
	\end{equation}
	Similarly, an edge $e$ determines a map $\H_1(\Gamma, A)_\eta \to C^1(\Gamma, A) \to A$; so long as $e$ is not
	a loop, we define the map 
	$\beta_e : \H_1(\Gamma / e, A)_{\eta / e} \to A$ by demanding commutativity of the following diagram. 
	\begin{equation} \label{diag:momentrelationships}
		\begin{tikzcd}
			\H_1(\Gamma \setminus e, A)_{\eta \setminus e} \arrow[d, "\sim"] \rar & \H_1(\Gamma, A)_{\eta} \rar \arrow[d, "\sim"] & A \arrow[d, equals] \\
			\operatorname{ker(\beta_e)}  \rar & \H_1(\Gamma / e, A)_{\eta / e}  \arrow[r, "\beta_e"] & A
		\end{tikzcd}
	\end{equation}

	\section{Combinatorial model} \label{sec:CKS}

	In this section, we will give a purely combinatorial model for the cohomology of our $\Bet(\Gamma)$ or 
	$\Dol(\Gamma)$, equipped with the appropriate filtration.  
	The model was originally introduced in \cite{MSV} to describe the perverse filtration 
	on the cohomology of the compactified Jacobian of a nodal curve, which as we have mentioned above 
	is closely related to $\Dol(\Gamma)$.

	We write our complexes over an arbitrary commutative ring $R$,
	which in the remainder of this article will always be $\Z$, $\Q$ or $\C$.  
	
	\begin{definition}
		Let $J$ be a subset of edges of $\Gamma$. We write 
		\[ \CKS^{2k,l} ( \Gamma, R ) \colonequals \bigoplus_{|J|=k}   \bigwedge^l \mathbb{H}(\Gamma \setminus J, R), \ \   \CKS^{2k+1,l}(\Gamma, R) \colonequals 0. \]
	\end{definition}
	We will often suppress the choice of $R$. Let $\CKS^{\bullet} ( \Gamma ) \colonequals \bigoplus_{2k+l=\bullet} \CKS^{2k,l}( \Gamma )$. We now define a differential $\CKS^{\bullet}( \Gamma ) \to \CKS^{\bullet+1}(\Gamma)$. 
	
	Let $e$ be an edge, viewed as a class in $\H^1(\Gamma, R)$. We have a map $\langle e, - \rangle : H_1(\Gamma, \R) \to R$ given by the composition of $\H_1(\Gamma, \R) \to R^{E(\Gamma)}$ with the projection $R^{E(\Gamma)} \to R$ onto the $e \textsuperscript{th}$ coordinate. We extend this to a linear function $$\langle e, - \rangle : \mathbb{H}(\Gamma, R) \to R$$ by setting $\langle e, f \rangle = 0$ for $f \in \H^1(\Gamma, R)$. Given $e \in E(\Gamma)$, consider the map $e \langle e, - \rangle  \colon \mathbb{H}(\Gamma) \to \mathbb{H}(\Gamma) $ which takes $x$ to $e \langle e, x \rangle$. This map depends on the edge $e$ but not its orientation. Extend, via the Leibniz rule, to a linear map $e \langle e, - \rangle  \colon \bigwedge^l \mathbb{H}(\Gamma) \to \bigwedge^l \mathbb{H}(\Gamma)$.  
	
	\begin{lemma} \label{lem:imageofde}
		The image of $e \langle e, - \rangle$ is the subspace 
		\begin{equation} \label{eq:imagede} e \wedge \bigwedge^{l-1} \bigg( \H_1(\Gamma \setminus e) \oplus  \H^1(\Gamma) \bigg). \end{equation}
	\end{lemma}
	\begin{proof}
		Note that $ \H_1(\Gamma \setminus e) \oplus  \H^1(\Gamma) = \ker(e \langle e, - \rangle)$. Choose any splitting of $\mathbb{H}(\Gamma)$ into $\ker(e \langle e, - \rangle) \oplus \mathbb{F}$ where $\mathbb{F}$ is rank one; then $e \langle e, - \rangle$ restricts to an isomorphism $\mathbb{F} \cong R e$. We have $\bigwedge^{l} \mathbb{H}(\Gamma) = \bigg( \mathbb{F} \otimes \bigwedge^{l- 1} \ker(e \langle e, - \rangle) \bigg) \oplus \bigwedge^{l} \ker(e \langle e, - \rangle)$. The map $e \langle e, - \rangle$ takes the left-hand summand isomorphically onto \ref{eq:imagede} and kills the right-hand summand. 
	\end{proof}
	When $e$ is not a bridge, we have an identification $e \wedge \bigwedge^{l-1} \bigg( \H_1(\Gamma \setminus e) \oplus \H^1(\Gamma) \bigg) = \bigwedge^{l-1} \mathbb{H}(\Gamma \setminus e)$, and thus we obtain a map $$d_e  \colon \bigwedge^l \mathbb{H}(\Gamma) \to \bigwedge^{l-1} \mathbb{H}(\Gamma \setminus e).$$
	
	More explicitly, 
	\[ d_e (x_1 \wedge x_2 \wedge ... \wedge x_{l}) = \sum_{i=1}^l (-1)^{i-1} \langle e, x_i \rangle x_1 \wedge x_2 \wedge ... \wedge \widehat{x}_i \wedge ... \wedge x_l. \]
	\begin{definition}
		Let $d_{\CKS}  \colon \CKS^{2k,l} ( \Gamma ) \to \CKS^{2k+2,l-1} ( \Gamma )$ be the linear map whose restriction to $\bigwedge^{l} \mathbb{H}(\Gamma \setminus J)$ is the direct sum over all non-bridge edges $e$ in $\Gamma \setminus J$ of $d_e  \colon \bigwedge^{l} \mathbb{H}(\Gamma \setminus J) \to \bigwedge^{l-1} \mathbb{H}(\Gamma \setminus J \setminus e)$.  
	\end{definition}
	
	\begin{lemma}
		The map $d_{\CKS}$ makes $\CKS^{\bullet}(\Gamma)$ into a complex, i.e. $ d_{\CKS}^2 = 0$.
	\end{lemma} 
	\begin{proof}
		It is easy to see that $d_e^2=0$. We must check that additionally, $d_{e_1} d_{e_2} = - d_{e_2} d_{e_1}$. The sign arises when passing from $e \langle e, - \rangle$ to $d_e$, which involves reordering the factors of a wedge product so that the factor $e$ comes out in front. Indeed, we may write the image of $d_{e_2} (x_0 \wedge ... \wedge x_N)$ under $e_1 \langle e_1, - \rangle$ as a sum of terms $(-1)^j x_0 \wedge ... \wedge e_1 \langle e_1, x_i \rangle \wedge ... \wedge \widehat{x}_j \wedge ... \wedge x_N$ with $i < j$ and  $(-1)^j x_0 \wedge ... \wedge \widehat{x}_j  \wedge ... \wedge   e_1 \langle e_1, x_i \rangle \wedge ... \wedge x_N$ with $i > j$. Here the hat indicates that a factor has been omitted. Then $d_{e_1} d_{e_2}(x_0 \wedge ... \wedge x_N)$ is the sum of terms $(-1)^{i+j} x_0 \wedge ... \wedge \widehat{x}_i \wedge ... \wedge \widehat{x}_j \wedge ... \wedge x_N$ with $i < j$ and  $(-1)^{i+j-1} x_0 \wedge ... \wedge \widehat{x}_j  \wedge ... \wedge   \widehat{x}_i \wedge ... \wedge x_N$ with $i > j$. Exchanging $e_1$ and $e_2$ exchanges the signs. 
	\end{proof}
	
	\begin{lemma} \label{lem:orientationindependence}
		The complexes $(\CKS^{\bullet}(\Gamma), d_{\CKS})$ associated to different orientations of $\Gamma$ are canonically isomorphic. 
	\end{lemma}
	\begin{proof}
		Suppose the orientations differ at a single edge $e$. Let $\epsilon_e$ be the automorphism of $\CKS^{\bullet}(\Gamma)$ which multiplies $\bigwedge \mathbb{H}(\Gamma \setminus J)$ by $-1$ if $e \in J$, and is the identity on the other sumands. Then $\epsilon_e$ intertwines the differentials associated to the two orientations.
		If the  orientations differ at multiple edges, the differentials are intertwined by a product of such $\epsilon_e$. 
	\end{proof}
	
	\begin{figure}
		\begin{equation*}
			\begin{tikzcd} 
				\bigwedge^4 \mathbb{H}(\subedges)    \\
				\bigwedge^3 \mathbb{H}(\subedges)  \arrow[r, "d_\CKS"]  &  \bigoplus_{i=1}^3 \bigwedge^2 \mathbb{H}(\subedges \setminus e_i)  \\
				\bigwedge^2 \mathbb{H}(\subedges)  \arrow[r, "d_\CKS"]  &   \bigoplus_{i=1}^3 \bigwedge^1 \mathbb{H}(\subedges \setminus e_i)  \arrow[r, "d_\CKS"]  &   \bigoplus_{i,j=1}^3 \bigwedge^0 \mathbb{H}(\subedges \setminus e_i, e_j) &  \\
				\bigwedge^1 \mathbb{H}(\subedges)  \arrow[r, "d_\CKS"]  & \bigoplus_{i=1}^3 \bigwedge^0 \mathbb{H}(\subedges \setminus e_i) &    \\
				\bigwedge^0 \mathbb{H}(\subedges)  &  \\
				\ & 
			\end{tikzcd}
		\end{equation*}
		\caption{
			The complex $\CKS^{\bullet}( \subedges )$, where $\subedges$ is the 
			graph with two vertices joined by three edges $e_1, e_2, e_3$. We have only indicated the groups which are not automatically zero for degree reasons.
			The cohomological grading increases as one moves up or to the right. The cohomology is described in Figure \ref{fig:cohcalc}.} 
	\end{figure}
	
	\begin{figure} \label{fig:cohcalc}
		\begin{tabular}{llll} \toprule
			& 0 & 1 & 2  \\ \midrule
			4 & $R$  &  &  \\
			3 & $R^2$ & $R$  \\
			2 & $R^2$ & $0$ & $R$ \\
			1 & $R^2$ & $R$ &  \\
			0 & $R$ &  &     \ \\ \bottomrule
		\end{tabular}
		\caption{The cohomology of $\CKS^\bullet(\subedges)$. The $\CKS$-grading is indicated on the left, while the number $|J|$ of deleted edges is indicated on the top row. The cohomological grading is the sum of these numbers.} 
	\end{figure}

	By construction, the differential $d_{\CKS}$ takes $\CKS^{2k, l}( \Gamma )$ to $\CKS^{2k+2, l-1}( \Gamma )$, and thus preserves the subspace $\CKS_m( \Gamma ) \colonequals \oplus_{a+2b = m} \CKS^{a,b}( \Gamma )$. We thus have $\H^{i}(\CKS( \Gamma)) = \oplus_{m} \H^{i}(\CKS_m( \Gamma ))$. 
	\begin{definition}
		We call this extra grading on cohomology the $\CKS$-grading, so that $\H^{i}(\CKS_m( \Gamma ))$ has $\CKS$-degree $m$.
	\end{definition}
	
	Fix an edge $e \in \Gamma$ which is neither a loop nor a bridge. We can identify $\CKS^{\bullet - 2}( \Gamma \setminus e )$ with the subcomplex of $\CKS^{\bullet}( \Gamma )$ consisting of summands  $\bigwedge^l \mathbb{H}(\Gamma \setminus J)$ with $e \in J$. If we ignore the differential, the quotient complex is given by the summands $\bigwedge^l \mathbb{H}(\Gamma \setminus J)$ with $e \notin J$. The homotopy equivalence $\Gamma \setminus J \to (\Gamma \setminus J) / e = (\Gamma / e) \setminus J$ identifies each such summand with $\bigwedge^l \mathbb{H}((\Gamma / e) \setminus J)$; the quotient complex therefore has the same underlying graded vector space as $\CKS^{\bullet}(\Gamma / e)$. The differentials also match, and thus we have
	\begin{equation} \label{CKSshortexact}
		0 \to \CKS^{\bullet - 2}(\Gamma \setminus e) \to \CKS^{\bullet}(\Gamma) \to \CKS^{\bullet}(\Gamma / e) \to 0 
	\end{equation}
	\begin{definition} 
		The resulting long exact sequence 
		\begin{equation} \label{CKSsequence}
			\to \H^{\bullet-2}(\CKS( \Gamma \setminus e )) \xrightarrow{a^{\CKS}_e} \H^{\bullet}(\CKS( \Gamma ) )\xrightarrow{b^{\CKS}_e} \H^{\bullet}(\CKS( \Gamma / e ) ) \xrightarrow{c^{\CKS}_e} 
		\end{equation}
		is the $\CKS$ deletion-contraction sequence. 
		
	\end{definition}
	With a view to applying Definition \ref{def:delfiltration}, we consider the following more general situation. Suppose $\Gamma'$ is a connected subgraph of $\Gamma$ whose complement contains no self-edges. Then we likewise have a subcomplex $\CKS^{\bullet - 2|\Gamma \setminus \Gamma'|} \to \CKS^{\bullet}(\Gamma)$, given by all summands $\bigwedge^l \mathbb{H}(\Gamma \setminus J)$ with $\Gamma' \supset \Gamma \setminus  J$. The induced map on cohomology can be written as the composition, in any order, of the maps $a_e^{\CKS}$ for $e \in \Gamma \setminus \Gamma'$. In particular, the compositions in different orders are all equal.

	We can thus make the following special case of Definition \ref{def:delfiltration}.
	\begin{definition} \label{def:cksdeletion}
		The $\CKS$-deletion filtration is the filtration defined by Definition \ref{def:delfiltration}, where the covariant functor $A$ takes $\Gamma$ to $\H^{\bullet}(\CKS(\Gamma))$ and takes $\Gamma' \to \Gamma$ to the composition of $a_e^{\CKS}$ (in any order) for $e \in \Gamma \setminus \Gamma'$. 
	\end{definition}
	By construction, the maps $b_e^{\CKS}$ and $c_e^{\CKS}$ respect the $\CKS$-grading, whereas the map $a_e^{\CKS}$ increases the grading by one. Hence the $k$th step $\delfilt_k \H^i(\CKS ( \Gamma ))$ lies in the subspace of $\CKS$-degree $\geq i-k$. The reverse inclusion is clear, and thus $\delfilt_k \H^i(\CKS( \Gamma )) = \oplus_{m \leq 2i-k} \H^i(\CKS_m( \Gamma ))$. 
	
	\begin{corollary} \label{cor:cksfilfromcksgrad}
		The $\CKS$-deletion filtration is induced by the $\CKS$-grading on $\H^{\bullet}(\CKS ( \Gamma ))$.
	\end{corollary}
	Viewed as a sequence of filtered vector spaces, the maps of a graded sequence strictly preserve the filtrations. 
	\begin{corollary} \label{cor:cksstrictness}
		The $\CKS$ deletion-contraction sequence strictly preserves the $\CKS$-filtration.
	\end{corollary}

	\section{Generalised moment maps} \label{sec:amodg} \label{sec:moment} 
	Let us review moment maps.  
	A symplectic form $\omega$ on a manifold $X$ determines a map $f \mapsto \omega^{\#} df$ 
	from functions to vector fields.  Fixing a Lie algebra $\mathfrak{g}$, 
	we get similarly a map from $\mathfrak{g}^*$-valued functions
	to $\mathfrak{g}^*$-valued vector fields.  When such a resulting vector field integrates to the action
	of a Lie group $G$, said action is termed Hamiltonian, the function is termed
	the moment map, and is $G$-equivariant  (with respect to the coadjoint action on $\mathfrak{g}^*$).  
	In case $f$ is submersive over $0$, and $G$ acts freely on $f^{-1}(0)$, 
	the symplectic reduction $X/\!/G := f^{-1}(0)/G$ inherits a symplectic structure.  More generally
	we may reduce along coadjoint orbits; $X/\!/_O G = f^{-1}(O)/G$. 
	If in addition $X$ carried a K\"ahler structure and $G$ acts by isometries, the reduction
	inherits a K\"ahler structure as well.  If X carries a hyperk\"ahler structure and a $G$ action which
	is independently Hamiltonian for the K\"ahler forms $\omega_I, \omega_J, \omega_K$ 
	with moment maps $f_I, f_J, f_K: X \to \mathfrak{g}^*$, then there is under analogous
	conditions a hyperk\"ahler reduction $X /\!/\!/\!/ G := (f_I^{-1}(0) \cap f_J^{-1}(0) \cap f_K^{-1}(0)) / G$. 
	
	In the present article we will only be concerned with $G$ commutative.  Then a 
	map $X \to \mathfrak{g}^*$ is $G$-equivariant iff it is constant on orbits;  
	the coadjoint orbits are just the points of $\mathfrak{g}^*$. 
	
	There is a related notion of group valued moment map \cite{AMM}; in general this is somewhat
	sophisticated; here we will need only the commutative case, which is much simpler.  
	For $G$ a connected commutative Lie group, note that its Lie algebra $\mathfrak{g}$ is also
	its universal cover.  In this case, given a $G$ action on $X$, we say $f: X \to G$ is a multiplicative moment map 
	if it is constant on fibers and the map $\tilde{f}: X \times_G \mathfrak{g} \to \mathfrak{g} \cong \mathfrak{g}^*$ is a moment
	map for the natural $G$ action on $X \times_G \mathfrak{g}$.  One defines reduction as
	for ordinary moment maps, and the usual proofs of compatibility of reduction
	with K\"ahler or hyperk\"ahler structure go through unchanged (e.g. \cite[Theorem 3.1]{HKLR}). 
	
	The various above notions play a prominent role in the literature on character varieties 
	and related spaces.  The spaces we will consider in this paper also have 
	such Hamiltonian structures. The constructions we perform with them will require and retain such structures, 
	but often at intermediate stages will not  
	be (quasi-)Hamiltonian or hyperk\"ahler, in particular due to the group action being too small or the target of the moment 
	map too large.  E.g., we may have a subgroup $H \subset G$ and be interested in $\mu^{-1}(O) / H$.
	If $G$  and $H$ are abelian, this retains a Hamiltonian action of $G/H$.
	
	What will always be present is the structure of the action of a (commutative) Lie group $G$ and 
	a map to another commutative Lie group $M$, constant on fibers.  
	Here we develop some basic manipulations of such structures.

	\subsection{\qh-spaces} \label{qhm}

	Recall that for a group $G$, we say a space $X$ is a $G$-space if it carries a $G$-action.  
	
	\begin{definition}
		Fix a group $G$ and a $G$-space $M$.  By a \qh-space, we mean a $G$-space $X$ and a 
		$G$-equivariant map $\mu_X: X \to M$.  A morphism of \qh-spaces is a G-equivariant map $f: X \to Y$ 
		such that $\mu_Y = \mu_{X} \circ f$.
	\end{definition}
	
	\begin{remark} \label{rem: appropriate}
		We write `space' above to mean element in some appropriate category which will be clear from context.  
		For us this will always be a category
		of smooth manifolds, possibly with extra structure, e.g. complex manifolds, K\"ahler manifolds, etc.  
		
		For instance, if we say that $X$ is a complex $(G, M)$-variety, we mean
		that $G$ is a complex algebraic group, $M$ is a complex variety, and the $G$-action 
		and map $\mu_{X}$ are algebraic. 
	\end{remark}
	
	\begin{remark}
		We soon (in Convention \ref{conv: commutative and trivial}) impose much stricter requirements on the allowable $G, M$: 
		we will require $G$ commutative and the $G$-action on $M$ trivial. 
		We begin in the present generality only for the sake of making clear when these hypotheses become relevant, 
		namely in Lemma \ref{lem:qhmstronconv}. 
	\end{remark}

	\begin{example}
		If $M$ has a distinguished point $0 \in M$, then we write 
		$\mathbf{0} := \mathbf{0}_{(G, M)}$ for the $(G, M)$-space given by a
		point carrying the trivial $G$ action and whose
		image under the map $\mu$ has image $0 \in M$. 
	\end{example}
	
	\begin{example} 
		If $X$ is a space with a $G$-action, and $O$ is a space with a map to $M$,
		we write $[X \times O]$ for the \qh-space whose underlying space is 
		$X \times O$, on which $G$ acts by multiplication on the first factor and trivially on the second, 
		equipped with the map $\mu\colon X \times O \to O$ via the second projection. 
		
		In particular, we will often consider $[G \times M]$ where $G$ acts by (left) translation on the first factor.
	\end{example}
	
	\begin{example}
		Any $G$-stable subset of a $(G, M)$-space inherits a $(G, M)$-structure. In particular, given a $(G, M)$-space $X$ and any subset $O \subset M$, 
		the space $\mu_X^{-1}(O) \subset X$ carries a natural $(G, M)$-structure.
	\end{example}
	
	\begin{remark}
		Recall that a symplectic manifold carries a canonical Poisson structure, and that there is a 
		standard Poisson structure defined on $\mathfrak{g}^*$. 
		A $(G, \mathfrak{g}^*)$-structure on a manifold $X$ arises from a Hamiltonian $G$-action 
		for the symplectic form $\omega$
		iff the $G$-equivariant $\mu: X \to \mathfrak{g}^*$ is Poisson.  
	\end{remark}

	\begin{definition} \label{def:basechange}
		Fix a group $G$ and a $G$-space $M$.  
		Suppose given a group homomorphism $\rho: H \to G$, an $H$-space $N$, and a morphism of $H$-spaces 
		$\tau: M \to N$.   Then composition with $(\rho, \tau)$ determines a functor from $(G, M)$-spaces to 
		$(H, N)$-spaces. 
	\end{definition}
	
	\begin{example}
		If $H$ is a Lie subgroup of $G$, then the natural maps $H \to G$ and $\mathfrak{g}^* \to \mathfrak{h}^*$ 
		determine as above a functor from $(G, \mathfrak{g}^*)$-spaces to $(H, \mathfrak{h}^*)$-spaces.  
		This functor evidently takes Hamiltonian $G$-structures to Hamiltonian $H$-structures.

		However, we also have a functor from $(G, \mathfrak{g}^*)$-spaces to $(H, \mathfrak{g}^*)$-spaces, just by restriction
		of the group action.  
	\end{example}
	
	\subsection{Reduction}
	
	\begin{definition}
		Let $X$ be a \qhm.  For a $G$-invariant point $m \in M$, we define 
		\[ X \sslash_{m} G \colonequals \mu_X^{-1}(m) / G. \]
		If we wish to emphasize $M$, we write 
		$X \sslash_{m \in M} G$. 
	\end{definition}
	\begin{remark}
		Whenever we write $\mu_X^{-1}(m) / G$, we are implicitly asserting that the quotient 
		makes sense.  More precisely, since we are always working with smooth manifolds, 
		we require that $m$ is a regular value of $\mu$ and that the $G$ action on $\mu_X^{-1}(m)$ 
		is free.  
	\end{remark} 
	\begin{remark}
		In our applications, the $G$ action on $M$ is trivial, i.e. every point $m \in M$ is invariant.  
		(When the $G$ action is not trivial, it is natural to also allow invariant subsets in place of $m$.) 
	\end{remark}
	\begin{remark} When $X$ is a Hamiltonian $G$-space and $\mu_X: X \to \mathfrak{g}^*$ is the moment map, 
		this is the classical symplectic reduction. 
	\end{remark}

	Let us note the following elementary compatibilities:

	\begin{lemma} \label{lem:basechange2}
		An injection $H \hookrightarrow G$ determines a surjection 
		$X \sslash_{m} H  \to  X \sslash_{m} G$. 
	\end{lemma}

	\begin{lemma} \label{lem:basechange1}
		Let $\tau: M \to N$ be a morphism of $G$-spaces.  
		Let $m \in M$.  For a $(G, M)$-space $X$, there is a Cartesian square
		\begin{equation} 
			\begin{tikzcd}
				X \sslash_{m} G \arrow[hookrightarrow]{r} \arrow[d] & X \sslash_{\tau(m)} G \arrow[d] \\
				m \arrow[hookrightarrow]{r} & \tau^{-1}\tau(m) 
			\end{tikzcd}
		\end{equation}
	\end{lemma}
	
	\begin{remark}
		It will later be relevant that if $\tau: M \to N$ is a group homomorphism, then 
		$\tau^{-1}\tau(m)$ is a torsor for the kernel. 
	\end{remark}

	\subsection{Convolution of \qhm s} \label{sec:convolution}

	\begin{definition} \label{def: convolution}
		Let $M, N$ be $G$-spaces with a $G$-equivariant morphism $c: M \times N \to N$.  
		Then from a $(G, M)$-space $X$ and a $(G, N)$-space $Y$, we construct a $(G, N)$-space 
		$X \bullet Y$ as follows.  The underlying space is $X \times Y$, and: 
		$$g(x, y) = (gx, gy) \qquad \qquad \mu_{X \bullet Y}(x, y) = c(\mu_X(x), \mu_Y(y))$$
		For $n \in N^G$, we also write
		$$X \star_{G, N, n} Y \colonequals  (X \bullet Y) \sslash_n G = \mu_{X \bullet Y}^{-1}(n) / G$$
	\end{definition}
	
	If the choice of $G$, $N$, or $n$ is clear from context, we may omit them. 
	In particular, we will uniformly employ the abbreviations $\star_{\starGG} \colonequals \star_{\Gm, \Gm}$, $\star_{\starUU} \colonequals \star_{\U_1, \U_1}$ and $\UUCstar\colonequals \star_{\U_1, \uutimesC}$. Here $\uu$ denotes the unit circle.
	
	\begin{remarque}
		Through this text we are working with smooth manifolds.  In this context, 
		the notation $X \star_n Y$ implicitly asserts that $n$ is a regular value for the map 
		$c(\mu_X, \mu_Y) \colon X \times Y \to N$ and that $G$ acts freely on $c(\mu_X, \mu_Y)^{-1}(n)$. 
	\end{remarque}

	\begin{remark}
		As mentioned above, in the remaining chapters we always have $G$ commutative and acting trivially on $M$ and $N$.
		The additional structures needed in the above definition will always come from $M$ being a commutative group,
		and $N$ being an $M$-torsor (and usually $N = M$). 
	\end{remark}

	\begin{lemma} \label{lem:qhmstronconv}
		In the setting of Def. \ref{def: convolution}, 
		assume in addition that $G$ is commutative and acts trivially on $M$.  
		Then we may equip $X \star Y$ with a $(G, M)$ structure by defining
		$g (x, y) = (gx, y)$ and $\mu_{X \star Y}(x, y) = \mu_X(x)$. 
	\end{lemma}
	\begin{proof}
		The main point is to check that the written formulas, which are well defined on $X \bullet Y$,
		descend to the quotient.  We should check that $(gx, y)$ and $(ghx, hy)$ are in the same $G$-orbit
		in $X \bullet Y$; this follows from commutativity of $G$.  We should check that 
		$\mu_X(hx) = \mu_X(x)$; this follows from equivariance of $\mu$ and triviality of the $G$-action on $M$. 
	\end{proof}
	
	\begin{remarque}
		There is an $(G, N)$-space isomorphism $X \bullet Y \cong Y \bullet X$.
		If, in addition to the hypotheses of Lemma \ref{lem:qhmstronconv}, 
		we have $N = M$ and the map $c: M \times M \to M$ is commutative, then 
		there is an isomorphism of spaces $\phi: X \star_m  Y \cong Y \star_m X$ for any $m$.  If the binary operation defined by $c$ has inverses, then the $(G, M)$-structures on $X \star_m Y$ and $Y \star_m X$ are related
		by $\phi(g\cdot (x,y)) = g^{-1} \cdot \phi(x, y)$ and 
		$\mu_{X \star Y} = m - \mu_{Y \star X}$. 
	\end{remarque}
	
	\begin{convention}\label{conv: commutative and trivial} 
		For the remainder of the article, when we discuss $(G, M)$ spaces, $G$ will always be a commutative group,
		and the $G$-action on $M$ will be trivial.  When we discuss convolution, we always additionally require $M = N$
		and ask that the map $c: M \times M \to M$ define the structure of a commutative group. 
	\end{convention}

	\begin{example} \label{starwithapoint}
		Recall that $\mathbf{0} = \mathbf{0}_{(G, M)}$ denotes the point with trivial $G$ action and $\mu = 0 \in M$.  For
		any $(G,M)$-space $X$, and any $m \in M$, 
		we have 
		\[ X \star_{m} \mathbf{0} = X \sslash_{m} G =  \mathbf{0}  \star_{m} X\]
		Lemma \ref{lem:qhmstronconv} asserts that the resulting space should acquire a $(G, M)$-structure. 
		Note the resulting $G$ action is trivial, and the map $\mu$ is the constant map with value $m$ on the left, 
		or $0$ on the right. 
	\end{example}

	\subsection{Some convolution lemmas}
	
	\begin{lemma} \label{lem:functoriality}
		Given a $(G, M)$-space $X$ and a \qh-map $f \colon Y \to Y''$, the formula $x \times y \mapsto x \times f(y)$ restricts
		and descends to give a  \qh-map
		\[ f_X \colon X \star Y \to X \star Y' \]
	\end{lemma}
	\begin{proof}
		If $x \times y \in X \times Y$ lies in $\mu_{X \bullet Y}^{-1}(\zeta)$, then $x \times f(y)$ lies in $\mu_{X \bullet Y'}^{-1}(\zeta)$ since $f$ preserves the moment map. The resulting map $id \times f  \colon \mu_{X \bullet Y}^{-1}(\zeta) \to \mu_{X \bullet Y'}^{-1}(\zeta)$ is $G$-equivariant, and thus descends to the quotients.
	\end{proof}
	\begin{remarque} \label{rem:compactexhaustion}
		Suppose $Y$ and $Y'$ and $X$ have exhaustions by compact $G$-equivariant subsets, compatible with the map $f$. Then $f_X$ is compatible with the 
		exhaustions of $X \star Y$ and $X \star Y'$ obtained from taking cartesian products of compact sets. 
	\end{remarque}
	
	\begin{lemma} \label{lem:injsur}
		If $f \colon Y \to Y'$ is injective (resp surjective), then $f_X \colon X \star Y \to X \star Y'$ is injective (resp surjective).
	\end{lemma}
	\begin{proof}
		If $f$ is injective (resp surjective), then so is $f \times id  \colon X \times Y \to X \times Y'$. Injectivity is clearly preserved by 
		restriction to $\mu_{X \bullet Y}^{-1}(\eta)$ and $\mu_{X \bullet Y'}^{-1}(\eta)$. To see that surjectivity is also preserved, note 
		that if $\mu(x, y') = \eta$, then for any preimage $y$ of $y'$, $\mu(x,y) = \eta$.  Finally, passing to $G$-quotients
		preserves injectivity and surjectivity of $G$-equivariant maps. 
	\end{proof}
	
	\begin{lemma} \label{lem:clopen}
		If $f : Y \to Y'$ is an open (resp closed) embedding, then $f_X \colon X \star Y \to X \star Y$ is an open (resp closed) embedding.
	\end{lemma}
	\begin{proof}
		By Lemma \ref{lem:injsur}, it is enough to check that if $Y \subset Y'$ is open, then $X \star Y \subset X \star Y'$ is open. The opens of $X \star Y'$ are given by intersecting $G$-invariant opens of $X \times Y'$ with $\mu^{-1}_{X \bullet Y'}(\eta)$. In particular $X \star Y = X \times Y \cap \mu^{-1}_{X \bullet Y'}(\zeta) / G$ is open. 
	\end{proof}
	
	\begin{lemma} \label{lem:freequotienttrivialized} \label{lem:freequotienttrivialized2}
		Let $X$ be a $(G, M)$ space.  
		Let $G$ act freely on $T$, and let $O \subset M$.  Then $$X \star_\zeta [T \times O] = \mu_X^{-1}(\zeta - O) \times_G T$$ and is naturally a $\mu_X^{-1}(\zeta-O)$-bundle over $T/G$ with structure group $G$. 
	\end{lemma}
	\begin{proof}
		We have $\mu_{X \bullet [T \times O]}^{-1}(\zeta) = \mu_X^{-1}(\zeta - O) \times T$. Taking the quotient by $G$ on both sides gives the desired identification.
	\end{proof} 
	
	We often use the following very special case: 
	
	\begin{corollary} \label{cor:bulletslice} \label{cor:bullettriv} \label{cor:bulletunit}
		Let $X$ be a $(G, M)$-space.  For any $O \subset M$, there is a canonical isomorphism
		of $(G, M)$-spaces: 

		\[ X \star_{\zeta} [G \times O] \cong \mu_X^{-1}(\zeta - O). \]
		In particular, for any $\zeta \in M$, there is a canonical isomorphism $X \star_{\zeta} [G \times M] = X$. 
	\end{corollary}

	\begin{lemma} \label{lem:abstractfactor2}  
		Let $X_1$ be a $(G_1 \times G_2, M_1 \times M_2)$-space, and let $X_2$ be a $(G_2, M_2)$-space.  View $X_1 \times X_2$ as a 
		$(G_1 \times G_2, M_1 \times M_2)$-space, where 
		$$(g_1, g_2) \cdot (x_1, x_2) := ( (g_1, g_2) \cdot x_1, g_2 \cdot x_2)$$
		$$\mu_{X_1 \times X_2}(x_1, x_2) := \mu_{X_1}(x_1) + (0, \mu_{X_2}(x_2))$$
		
		Then for $\zeta_1 \in M_1$ and $\zeta_2 \in M_2$, 
		\begin{equation} \label{eq:abstractfactor2} (X_1 \times X_2) \sslash_{( \zeta_1, \zeta_2)}  G_1 \times G_2 = 
			(X_1 \sslash_{\zeta_1} G_1) \star_{G_2, M_2, \zeta_2} X_2. \end{equation}
	\end{lemma}

	\section{Spaces from graphs} \label{sec:constructmoduli} \label{sec:graphspace}
	
	Fix commutative connected Lie groups $G, M$.  We regard $G$ as acting trivially on $M$. 
	Let $Z$ be a $(G, M)$-space. 
	Let $\Gamma$ be a connected graph, and $\eta$ an assignment of an element of $M$ to each vertex.  
	From this data, we will produce a new space $Z(\Gamma , \eta)$.

	\subsection{Construction} \label{subsec:spacesfromgraphs}
	Recall from Section \ref{sec: graph conventions} above our conventions and notation for graphs. 
	
	Given a $(G, M)$-space $Z$, there is an action of $C^1(\Gamma, G) = G^{\ed}$ on $Z^{\ed}$, together with a map $\mu_{Z}^{E(\Gamma)} \colon Z^{\ed} \to M^{\ed} = C_1(\Gamma, M)$. Thus $Z^{\ed}$ is a $(C^1(\Gamma, G), C_1(\Gamma, M))$-space. 
	
	We view $Z^{\ed}$ as a $(\redC^0(\Gamma\homsep G), \redC_0(\Gamma\homsep M))$-space by composition with $(d_\Gamma^*, d_\Gamma)$ as in Definition \ref{def:basechange}. We write $\mu_\Gamma : Z^{\ed} \to \redC_0(\Gamma, M)$ for the associated map. 
	By definition, we have $\mu_\Gamma = d_\Gamma \circ \mu_{Z}^{E(\Gamma)}$.

	\begin{definition} \label{definitionofgraphspace}
		Let $\Gamma$ be a graph, let $\eta \in \redC_0(\Gamma, M) \subset  C_0(\Gamma, M) = M^{\ver}$, 
		and let $Z$ be a $(G, M)$-space.  We define: 
		\[ Z(\Gamma , \eta) \colonequals Z^{\ed} \sslash_{\eta} \redC^0(\Gamma\homsep G).\]
	\end{definition} 
	
	When the choice of $\eta$ is clear from context or irrelevant,  we simply write $Z(\Gamma)$. 
	On the other hand, if we wish to emphasize the dependence on $(G, M)$, we write $Z^{G,M}(\Gamma, \eta)$. 
	In Section \ref{sec: smoothness}, we discuss criteria which ensure that $\eta$ is a regular point of $\mu_\Gamma$ and 
	that the $\redC^0(\Gamma\homsep G)$-action is free, hence that $ Z(\Gamma , \eta) $ is a smooth manifold.

	The following diagram summarizes the situation, and defines the map $\mu_{res}$:
	\begin{equation}
		\begin{tikzcd}
			Z(\Gamma , \eta) \arrow[r, equal] \arrow[d, "\mu_{res}"] & 
			\mu_{\Gamma}^{-1}(\eta) / \redC^0(\Gamma, G)   \arrow[d] & 
			\mu_{\Gamma}^{-1}(\eta)  \arrow[r, hook] \arrow[d] \arrow[l] & 
			Z^{\ed} \arrow[d, "\mu_{Z}^{E(\Gamma)}"] \\
			\H_1(\Gamma \homsep M)_{\eta} \arrow[r, equal] & 
			d_{\Gamma}^{-1}(\eta) \arrow[r, equal] &
			d_{\Gamma}^{-1}(\eta)  \arrow[r, hook] &
			C_1(\Gamma\homsep M)
		\end{tikzcd}
	\end{equation}

	\begin{example} Since $d_{\boing}=0$, we have
		$Z = Z(\boing, 0)$.
	\end{example}
	
	\begin{lemma} \label{lem:loopsandotherproducts}
		Suppose $\Gamma$ contains a loop $e$. Then $Z(\Gamma) = Z(\Gamma \setminus e) \times Z(\boing)$.
	\end{lemma}
	\begin{proof}
		This follows  from the fact that $\redC^0(\Gamma, G)$ acts trivially on the factor of $Z$ associated to $e$.
	\end{proof}

	\begin{proposition} \label{prop: h1h1 structure} 
		The action of $C^1(\Gamma\homsep G)$ on $Z^{\ed}$ descends to an 
		action of 

		$\H^1(\Gamma \homsep G)$ on $Z(\Gamma, \eta)$.  Combined with the map $\mu_{res}\colon Z(\Gamma , \eta) \to \H_1(\Gamma \homsep M) _{\eta}$
		above, this defines a $(\H^1(\Gamma \homsep G), \H_1(\Gamma \homsep M)_\eta)$-space structure on $Z(\Gamma , \eta)$. 
		Finally, if $\mu_Z$ is proper, then so is $\mu_{\operatorname{res}}$.
	\end{proposition}
	\begin{proof}
		Regarding properness, note that it is 
		preserved both by restriction to the closed set $\mu_{\Gamma}^{-1}(\eta)$ and by descent to the quotient by $\overline{C}^0(\Gamma, G)$. 
	\end{proof}
	
	\begin{lemma} \label{lem:disconnectedgraphs}
		For $\Gamma$ disconnected and $\eta \in \redC_0(\Gamma,M)$ generic in the sense of Definition \ref{defsimplemomentvalue}, we have $Z(\Gamma, \eta) = \emptyset$.
	\end{lemma} 
	\begin{proof}
		The locus $\H_1(\Gamma, A)_\eta$ is empty by Remark \ref{rem: disconnected generic}. Since $Z(\Gamma, \eta)$ is a quotient of a preimage of this locus in $Z^{\ed}$, it is also empty. 
	\end{proof}
	\begin{remark}
		Since we will always assume that $\eta$ is generic, we will always have $Z(\Gamma, \eta) = \emptyset$ for $\Gamma$ disconnected. Such `empty' graph spaces arise naturally if we start with a connected graph $\Gamma$ with generic $\eta$, and produce a disconnected graph by deleting a bridge, with parameter $\eta \setminus e$ as in Definition \ref{def:delconteta}.
	\end{remark}

	\begin{lemma} \label{ex:cohomologicalquotient} 
		Let $Z = [G \times M]$.  Then $Z(\Gamma, \eta) \cong [\H^1(\Gamma\homsep G) \times  \H_1(\Gamma\homsep M)_\eta]$
	\end{lemma}
	\begin{proof}
		We identify $[G \times M]^{\ed} = [C^1(\Gamma, G) \times C_1(\Gamma, M)]$. Then $\mu_\Gamma^{-1}(\eta) = [G^{\ed} \times H_1(\Gamma, M)_\eta]$, and $\mu^{-1}(\eta)/\redC^0(\Gamma,G) = [H^1(\Gamma, G) \times H_1(\Gamma, M)_\eta]$.

	\end{proof}
	
	\subsection{Independence of orientation} \label{sec:indeporient}
	Pick a subset $J$ of the edges of $\Gamma$, and let $\Gamma'$ be the oriented graph obtained from $\Gamma$ by switching the orientation of each edge in $J$. 
	\begin{proposition} \label{prop:indeporient}
		Let $f: Z \to Z$ be an automorphism of topological spaces which intertwines the $\G$-action with the inverse $\G$ action and the $\A$-map with the inverse $\A$-map. Then $f$ determines an isomorphism of topological spaces
		\[ Z(\Gamma) \xrightarrow{\sim} Z(\Gamma'). \]
		If $f$ is a map of smooth manifolds or algebraic varieties, then so is the induced map.
	\end{proposition}
	\begin{proof}
		The map $Z^{\ed} \to Z^{E(\Gamma')}$ given by $f$ on the factors in $J$ and by the identity everywhere else intertwines the group actions and moment maps for $\Gamma$ and $\Gamma'$, and thus descends to the requisite isomorphism.
	\end{proof}
	
	\subsection{Smoothness} \label{sec: smoothness} 
	As always, we work some category of smooth manifolds, possibly with extra structure.  Thus  $Z$ is a smooth
	manifold, and $\mu_{Z}: Z \to M$ is a $C^\infty$ map.  We impose the following additional hypothesis for the remainder of the section: 
	
	\begin{hypothesis} \label{hyp:bZsmoothness}
		The map $\mu_Z: Z \to M$ restricts to a submersion over $M \setminus 0$, on which $G$ acts freely.  
	\end{hypothesis}
	
	\begin{proposition} \label{prop:smoothcond}
		Suppose that $\eta$ is generic (Def. \ref{defsimplemomentvalue}). Then $\eta$ is a regular value of $\mu_\Gamma$, 
		and the action of $G^{\overline{\ver}}$ on $\mu_\Gamma^{-1}(\eta)$ is free.  In particular, 
		$Z(\Gamma , \eta)$ is smooth, and 
		$$\dim Z(\Gamma, \eta) = |E(\Gamma)|  \dim Z  + (1 - |V(\Gamma)|) (\dim G + \dim M)$$
	\end{proposition}
	\begin{proof}
		We begin by studying the group action. By Hypothesis \ref{hyp:bZsmoothness},  
		the action of $C^1(\Gamma \homsep G)$ on $\big( \mu_Z^{\ed} \big)^{-1}(a)$ restricts to a free action of $C^1(\mathrm{Supp}(a) \homsep G)$. Thus for $\redC^0(\Gamma\homsep G)$ to act freely on $\big( \mu_Z^{\ed} \big)^{-1}(a)$, it is sufficient that the composition 
		\begin{equation} \label{eq:injectiveseq} \redC^0(\Gamma\homsep G) \xrightarrow{d^*_{\Gamma}}  C^1(\Gamma\homsep G) \to C^1(\mathrm{Supp}(a) \homsep G)\end{equation} 
		be injective. Here the second map is the natural projection. The composition is the differential $d^*_{\mathrm{Supp}(a)}$. 
		Thus, the kernel is trivial exactly when $\mathrm{Supp}(a)$ is connected.
		
		We use a similar argument to show that the map $\mu_\Gamma \colon Z^{\ed} \to \redC_0(\Gamma\homsep M)$ is a submersion. 
		We can factor the differential
		$d \mu_\Gamma$ as $d (d_\Gamma) \circ d (\mu^{\ed})$. For any $z \in (\mu_Z^{\ed})^{-1}(a)$, the image of the differential $d \mu^{\ed}$ 
		contains the tangent space of $M^{E(\mathrm{Supp}(a))} = C_1(\mathrm{Supp}(a) \homsep M)$.  
		Dually to \ref{eq:injectiveseq}, we have a surjective composition
		$$C_1(\mathrm{Supp}(a)\homsep M) \to C_1(\Gamma\homsep M)  \to \redC_0(\Gamma\homsep M).$$ Thus $d (d_\Gamma)$ 
		is a surjection even when restricted to the tangent space of $M^{E(\mathrm{Supp}(a))}$. Thus $\mu_\Gamma$ is a submersion, 
		and $\mu_\Gamma^{-1}(\eta)$ is smooth.
	\end{proof}

	\begin{remark}
		If $\dim G + \dim M = \dim Z$, the dimension formula 
		simplifies to: $\dim Z(\Gamma, \eta) =  \dim Z \cdot \dim \H^1(\Gamma) $. 
		In fact below we will always have $\dim Z = 4$. 
	\end{remark}

	\begin{remarque} \label{rem:matroid}
		Definition \ref{definitionofgraphspace} makes sense with $d_{\Gamma}$ replaced by a general integer matrix $D: \Z^n \to \Z^k$.
		The definition of generic $\eta$ can be extended to this case, but in general such $Z(D , \eta)$ will have orbifold singularities. 
	\end{remarque}
	
	\subsection{Stabilisers}
	We consider the action of $\H^1(\Gamma, G)$ on $Z(\Gamma, \eta)$. We are interested in the stabiliser of a point $z \in Z(\Gamma, \eta)$. 
	\begin{lemma}  
		Under Hypothesis \ref{hyp:bZsmoothness}, $\H^1(\mathrm{Supp}(a) \homsep G)$ acts freely on $\mu_{\res}^{-1}(a)$ for $a \in \H_1(\Gamma \homsep \mom)_\eta$. 
	\end{lemma}
	\begin{proof}
		By definition, $\mu_{\res}^{-1}(a)$ is a $\redC^0(\Gamma, G)$-quotient of a subspace of $(M \setminus 0)^{\mathrm{Supp}(a)} \times M^{\ed \setminus \mathrm{Supp}(a)}$. By the hypothesis, $C^1(\mathrm{Supp}(a), G)$ acts freely on this locus. Hence $\H^1(\mathrm{Supp}(a), G)$ acts freely on the quotient. 
	\end{proof}
	\subsection{An open cover}  \label{sec: cover}
	
	\begin{definition} \label{def: open cover}
		Let $D \subset \mu_Z^{-1}(0)$ be a closed subset.   Then for any subgraph $\Gamma' \subset \Gamma$ we define
		\[ \hat{\openset}_{\Gamma'} := (Z \setminus D)^{E(\Gamma')} \times Z^{E(\Gamma \setminus \Gamma')} \subset Z^{\ed} \]
		We write $\openset_{\Gamma'}$ for the image of $\hat \openset_{\Gamma'} \cap \mu_Z^{-1}(0)$ in $Z(\Gamma, \eta)$.  
		Note $\openset_{\Gamma'}$ is an open subset.  
	\end{definition}

	\begin{lemma} \label{lem:pretreecover}
		Let $\mathcal{T}(\Gamma)$ be the set of spanning trees of $\Gamma$.  
		For $\eta$ generic, we have 
		$$\bigcup_{T \in \mathcal{T}(\Gamma)} \openset_T = Z(\Gamma, \eta).$$
	\end{lemma}
	\begin{proof}
		Let $x \in Z(\Gamma, \eta)$ and fix a lift $\hat{x} \in \mu_\Gamma^{-1}(\eta)$. Let $a = \mu_Z^{\ed}(\hat{x})$. By genericity of $\eta$, $\operatorname{Supp}(a)$ contains a spanning tree $T \subset \Gamma$. Thus $\hat{x} \in \hat{\openset}_{T}$ and $x \in \openset_{T}$.
	\end{proof}
	\begin{remark}
		Lemma \ref{lem:pretreecover} gives another proof that $Z(\Gamma, \eta)$ is smooth for generic $\eta$, though not
		that it is Hausdorff. 
	\end{remark}

	We will later be interested in $D$ satisfying $Z \setminus D \cong [G \times M]$.  These have the
	following additional properties: 
	
	\begin{lemma} \label{lem:treecover}
		An isomorphism $Z \setminus D \cong [G \times M]$ determines an isomorphism 
		\begin{equation} \label{eq:chartcoordinates} \openset_{T} \cong Z(\Gamma / T) = Z^{E(\Gamma / T)} \end{equation} 
	\end{lemma}
	\begin{remark}
		Note there is a canonical identification $E(\Gamma / T) = E(\Gamma \setminus T)$, and these sets are canonically identified
		with a basis of $\H^1(\Gamma)$. 
	\end{remark}
	\begin{proof}
		We have 
		\[ \openset_{\Gamma'} = [G \times M]^{E(T)} \times Z^{E(\Gamma / T)} \sslash_{\eta} \redC^0(\Gamma, G) = [G^{E(T)} \times M^{E(T)}] \star Z^{E(\Gamma / T)}. \]
		Applying Corollary \ref{cor:bulletunit} with $G' =  \redC^0(\Gamma, G) = G^{E(T)}$ and $M' = \redC_0(\Gamma, M) = M^{E(T)}$, we get a natural isomorphism
		$\openset_{T} \cong Z^{E(\Gamma / T)}$. 
	\end{proof}

	The isomorphism \eqref{eq:chartcoordinates} is easiest to describe on a slightly smaller open. We have 
	\[\bigcap_{T \in \mathcal{T}(\Gamma)} \openset_T= (Z \setminus D)^{\ed} \sslash_{\eta} \redC^0(\Gamma, G) \cong \H^1(\Gamma\homsep G) \times  \H_1(\Gamma\homsep M). \]
	
	\begin{lemma} \label{lem:coordinates} 
		For any given spanning tree $T \subset \Gamma$, there is a natural isomorphism
		\begin{equation} \label{eq:torustrivialisationingeneral}  \H^1(\Gamma\homsep G) \times  \H_1(\Gamma\homsep M) \cong (G \times M)^{E(\Gamma / T)} \end{equation}
		induced by the contraction $\Gamma \to \Gamma / T = \boing^{E(\Gamma / T)}$. If $(\theta, \gamma)$ is an element of the left-hand side, with $\theta = \sum_{e \notin T} \theta_e e$, this isomorphism takes $(\theta, \gamma)$ to $(\theta_e, \langle e, \gamma \rangle)_{e \notin T}$.
	\end{lemma}
	
	\begin{lemma} \label{lem:coordinatesoncharts}
		The isomorphism \eqref{eq:chartcoordinates} fits into the commutative diagram
		\begin{equation}
			\begin{tikzcd}
				\openset_{T} \arrow[r, "\sim"] & Z^{E(\Gamma / T)} \\
				\H^1(\Gamma\homsep G) \times  \H_1(\Gamma\homsep M) \arrow[u] \arrow[r, "\sim"]  & (G \times M)^{E(\Gamma / T)} \arrow[u] \\
			\end{tikzcd}
		\end{equation}
		where the bottom arrow is \eqref{eq:torustrivialisationingeneral}.
	\end{lemma}

	\subsection{Deletion, contraction, and convolution} \label{sec:deletioncontraction} \label{sec:gammadelcon}
	In this section we will explain how $Z(\Gamma , \eta)$ behaves under contraction (Lemma \ref{lem:uncontract}) and 
	deletion (Lemma \ref{lem:deletionfromcontraction}).  
	
	A key ingredient is a $(G, M)$-structure on $Z(\Gamma / e, \eta / e)$ associated to the newly contracted edge $e$.  
	Recall that by 
	Proposition \ref{prop: h1h1 structure}, $Z(\Gamma / e, \eta / e)$ is a $(\H^1(\Gamma / e, G), \H_1(\Gamma / e, M)_{\eta / e})$-space.  
	Recall from diagrams \ref{diag:momentrelationships0} and \ref{diag:momentrelationships} we have maps:
	$$\alpha_e \colon G \to \H^1(\Gamma / e, G) \qquad \qquad \beta_e \colon \H_1(\Gamma / e, M)_{\eta / e} \to M$$
	
	\begin{definition} \label{def:ghostaction}
		The $(G, M)$-structure on $Z(\Gamma  / e, \eta / e) $ defined by composition (Def. \ref{def:basechange}) with the maps $\alpha_e, \beta_e$ 
		will be called the deletion-contraction $(G, M)$-structure.  We denote the map to $M$ by $\mu_{(e)}: Z(\Gamma  / e, \eta / e) \to M$. 
	\end{definition}
	
	The deletion-contraction structure is related to the splitting 
	\begin{equation} \label{factone} Z^{\ed} = Z^{E(\Gamma / e)} \times Z^e. \end{equation} 
	
	Indeed, let us write $\pi_e: \Gamma \to \Gamma / e$ for the contraction.  Any vertex $v \in \Gamma$ determines 
	natural maps $v_!: G \to \redC^0(\Gamma, G)$ and $v^!: C_0(\Gamma, G) \to G$. We have an isomorphism 
	\begin{equation} \label{facttwo} 
		\pi_e^* \times t(e)_! : \redC^0(\Gamma / e, G) \times G \cong \redC^0(\Gamma, G)
	\end{equation}

	Using \ref{factone} and \ref{facttwo} we transport the $\redC^0(\Gamma, G)$-structure on 
	$Z^{\ed}$ to a $\redC^0(\Gamma / e, G) \times G$-structure on $Z^{E(\Gamma / e)} \times Z^e$.  
	The resulting $\redC^0(\Gamma / e, G)$-structure on $Z^{E(\Gamma / e)}$ and  
	$G$-structure on $Z^e$ are the standard ones.  The resulting 
	$\redC^0(\Gamma / e, G)$-structure on $Z^e$ is trivial, and the resulting $G$-structure on 
	$Z^{E(\Gamma / e)}$ is the deletion-contraction structure. 
	
	We have an isomorphism
	\begin{equation} \label{factthree} 
		(\pi_e)_* \times t(e)^! : C_0(\Gamma, M) \cong C_0(\Gamma / e, M) \times M 
	\end{equation}
	This map carries $\eta \mapsto (\eta / e, \eta_{t(e)})$ where $\eta_{t(e)}$ is the coefficient of $\eta$ at $t(e)$.

	Consider the composition
	$$ \nu :  Z^{E(\Gamma / e)} \times Z^e \cong Z^{\ed}  \xrightarrow{\mu_{\Gamma}} C_0(\Gamma, M) \cong C_0(\Gamma / e, M) \times M$$
	One checks that  
	
	\begin{equation} \label{factfour}
		\nu(w, z) = (\mu_{\Gamma / e}(w), \mu_{(e)}(w))  + (0, \mu(z))
	\end{equation}

	\begin{lemma}   \label{lem:uncontract}
		We have a cartesian diagram of spaces
		\begin{equation} \label{eq:reductiondiagram}
			\begin{tikzcd}
				Z(\Gamma , \eta) \arrow[r] \arrow[d] & (Z(\Gamma  / e, \eta / e) \times  Z)/G \arrow[d] \\
				\H_1(\Gamma, M)_{\eta} \arrow[hook, r] & \H_1(\Gamma / e, M)_{\eta / e} \times M
			\end{tikzcd}
		\end{equation}
		where the right-hand map descends from the product map $Z(\Gamma  / e, \eta / e)\times Z \to \H_1(\Gamma / e, M)_{\eta / e} \times M$ and the bottom map is given by pushforward along $\Gamma \to \Gamma / e$ on the first factor and taking the coefficient of $e$ on the second factor.
		The convolution on the upper right is taken with respect to the deletion-contraction $(G, M)$-structure. 
		
		The image of $Z(\Gamma, \eta)$ in $(Z(\Gamma  / e, \eta / e) \times  Z)/G$ is naturally identified as $Z(\Gamma  / e, \eta / e) \star_{G, M, \eta_{t(e)}}  Z$. 
	\end{lemma}

	\begin{proof} 
		We claim that  
		\begin{eqnarray*} 
			Z(\Gamma , \eta) & := & Z^{E(\Gamma)} \sslash_\eta \redC^0(\Gamma, G) \\
			&  = &
			(Z^{E(\Gamma/e)} \times Z) \sslash_{(\eta/e, \eta_{t(e)})} (\redC^0(\Gamma / e, G) \times G) \\
			& = &  Z(\Gamma  / e, \eta / e) \star_{G, M, \eta_{t(e)}}  Z
		\end{eqnarray*}
		In passing from the second to the third line, we have used Lemma \ref{lem:abstractfactor2}.  
		More precisely, to apply this lemma, we 
		view $Z^{E(\Gamma/e)}$ as a $(\redC^0(\Gamma / e, G) \times G, \redC_0(\Gamma / e, M) \times M)$-space 
		using the product of its
		standard structure and its deletion-contraction structure.  
		The remaining hypothesis of the Lemma is verified by \eqref{factfour}, and we translate the conclusion 
		of the lemma through the definition  $Z(\Gamma / e , \eta / e)  :=  Z^{E(\Gamma / e)} \sslash_{\eta / e} \redC^0(\Gamma / e, G)$. 
		This gives the top isomorphism of \eqref{eq:reductiondiagram}.

		The map $Z^{E(\Gamma)} \to C_1(\Gamma, M)$ descends to the left-hand vertical map, whereas $Z^{E(\Gamma / e)} \times Z \to C_1(\Gamma / e, M) \times M$ descends to the right-hand vertical map. These maps are intertwined by the identifications $Z^{E(\Gamma)} \to Z^{E(\Gamma / e)} \times Z$ and $C_1(\Gamma, M) \to C_1(\Gamma / e,M) \times M$, and this latter identification defines the bottom-map. 
	\end{proof}
	
	\begin{remark}
		While $(Z(\Gamma  / e, \eta / e) \times  Z)/G$ may not be a free quotient, it is free along the locus where we take the fiber product, 
		so this will cause no difficulty. 
	\end{remark}
	
	Let $\mathbf{0}$ be the point with the trivial $(G, M)$-structure.
	\begin{lemma} \label{lem:deletionfromcontraction}  
		We have the commutative diagram 
		\begin{equation} 
			\begin{tikzcd}
				Z(\Gamma  \setminus e, \eta \setminus e) \arrow[r] \arrow[d] &  (Z(\Gamma  / e, \eta / e)  \times \mathbf{0}) / G \arrow[d] \\
				\H_1(\Gamma \setminus e, M)_{\eta \setminus e} \arrow[hook, r] & \H_1(\Gamma / e, M)_{\eta / e} \times M
			\end{tikzcd}
		\end{equation}
		where the bottom row is given by Eq. \ref{torsor delcon composition} on the first factor, and the zero map on the second factor. 
		The image of $Z(\Gamma \setminus e, \eta \setminus e)$ in $(Z(\Gamma  / e, \eta / e) \times  \mathbf{0})/G$ is naturally identified as $Z(\Gamma  / e, \eta / e) \star_{G, M, \eta_{t(e)}}  \mathbf{0}$.
		
	\end{lemma}
	\begin{proof}
		Observe the canonical identifications $E(\Gamma \setminus e) = E(\Gamma / e)$ and 
		and $C_0(\Gamma \setminus e) = C_0(\Gamma)$ and $\redC^0(\Gamma \setminus e) = \redC^0(\Gamma)$. 
		The top isomorphism now follows
		from  $Z(\Gamma / e , \eta / e) = Z^{E(\Gamma / e)} \sslash_{\eta / e} \redC^0(\Gamma / e, G)$ and
		
		$$Z(\Gamma \setminus e , \eta \setminus e) = Z^{E(\Gamma \setminus e)} \sslash_{\eta \setminus e} \redC^0(\Gamma \setminus e, G) = 
		(Z^{E(\Gamma/e)} \times \mathbf{0}) \sslash_{(\eta/e, \eta_{t(e)})} (\redC^0(\Gamma / e, G) \times G)
		$$ 
		The last equality is obtained by applying Lemma \ref{lem:abstractfactor2} to $Z^{E(\Gamma/e)}$ and $\mathbf{0}$. To do so, we view $Z^{E(\Gamma/e)}$ as a $(\redC^0(\Gamma / e, G) \times G, \redC_0(\Gamma / e, M) \times M)$-space using the product of its standard structure and its deletion-contraction structure.  
		
		The compatibility of the bottom map is a direct calculation. 
	\end{proof}
	
	\begin{corollary} \label{cor:deletingbridgeskills}
		Suppose $\Gamma \setminus e$ is disconnected and $\eta$ is generic. Then 
		\[  Z(\Gamma  / e, \eta / e)  \star_{\eta_{t(e)}} \GApoint = \emptyset. \]
		The same holds when $\GApoint$ is replaced by any $(G, M)$-space $\bS$ with $\mu(\bS)=1$.
	\end{corollary}
	
	\begin{proof}
		Apply Lemma \ref{lem:disconnectedgraphs} and Lemma \ref{lem:genericityofdelcon} to $Z(\Gamma \setminus e, \eta \setminus e)$.
	\end{proof}

	We often abbreviate $Z(\Gamma) := Z(\Gamma, \eta)$ and likewise $Z(\Gamma / e) \colonequals Z(\Gamma / e, \eta / e)$ and $Z(\Gamma \setminus e) := Z(\Gamma \setminus e, \eta \setminus e)$. 
	
	\begin{lemma} \label{lem:deletioninclusionfromGMpoint}
		An inclusion of $(G,M)$-spaces $\mathbf{0} \to Z$ defines inclusions
		\begin{equation}
			Z(\Gamma \setminus e, \eta \setminus e) \to Z(\Gamma, \eta). 
		\end{equation}
		for any edge $e \in \Gamma$.
	\end{lemma}
	\begin{proof}
		Consider the embedding
		\[ Z^{E(\Gamma \setminus e)} \times \mathbf{0} \to Z^{E(\Gamma \setminus e)} \times Z = Z^{E(\Gamma)} \]
		This map induces a map of $\redC^0(\Gamma, G)$-reductions. The codomain reduces to $Z(\Gamma, \eta)$. The domain reduces to
		\[Z^{E(\Gamma \setminus e)} \times \mathbf{0} \sslash_{\eta} \redC^{0}(\Gamma, G) = Z^{E(\Gamma \setminus e)} \sslash_{\eta \setminus e} \redC^0(\Gamma \setminus e, G) = Z(\Gamma \setminus e, \eta \setminus e). \]
	\end{proof}

	\section{$\Bet(\Gamma)$} \label{sec:bet} 
	
	\subsection{Construction} \label{subsec:basicbettidef} \label{subsec:graphbettispace} \label{sec:buildblocksbetti}
	
	\begin{definition} \label{def:bettispace}
		We define the space
		$$ \Bet \colonequals \C^2 \setminus \{ 1 + xy  = 0\}. $$
		and the maps
		\begin{eqnarray*}
			\mu_{\Bet}^{\C^*}\colon \Bet & \to & \C^* \\
			(x, y) & \mapsto & 1 + xy
			\\
			\mu_{\Bet}^{\R}\colon \Bet & \to & \R \\
			(x, y) & \mapsto & |x|^2 - |y|^2
		\end{eqnarray*}
	\end{definition}

	\begin{figure}[h]
		\large
		\centering{
			\resizebox{70mm}{!}{\includegraphics{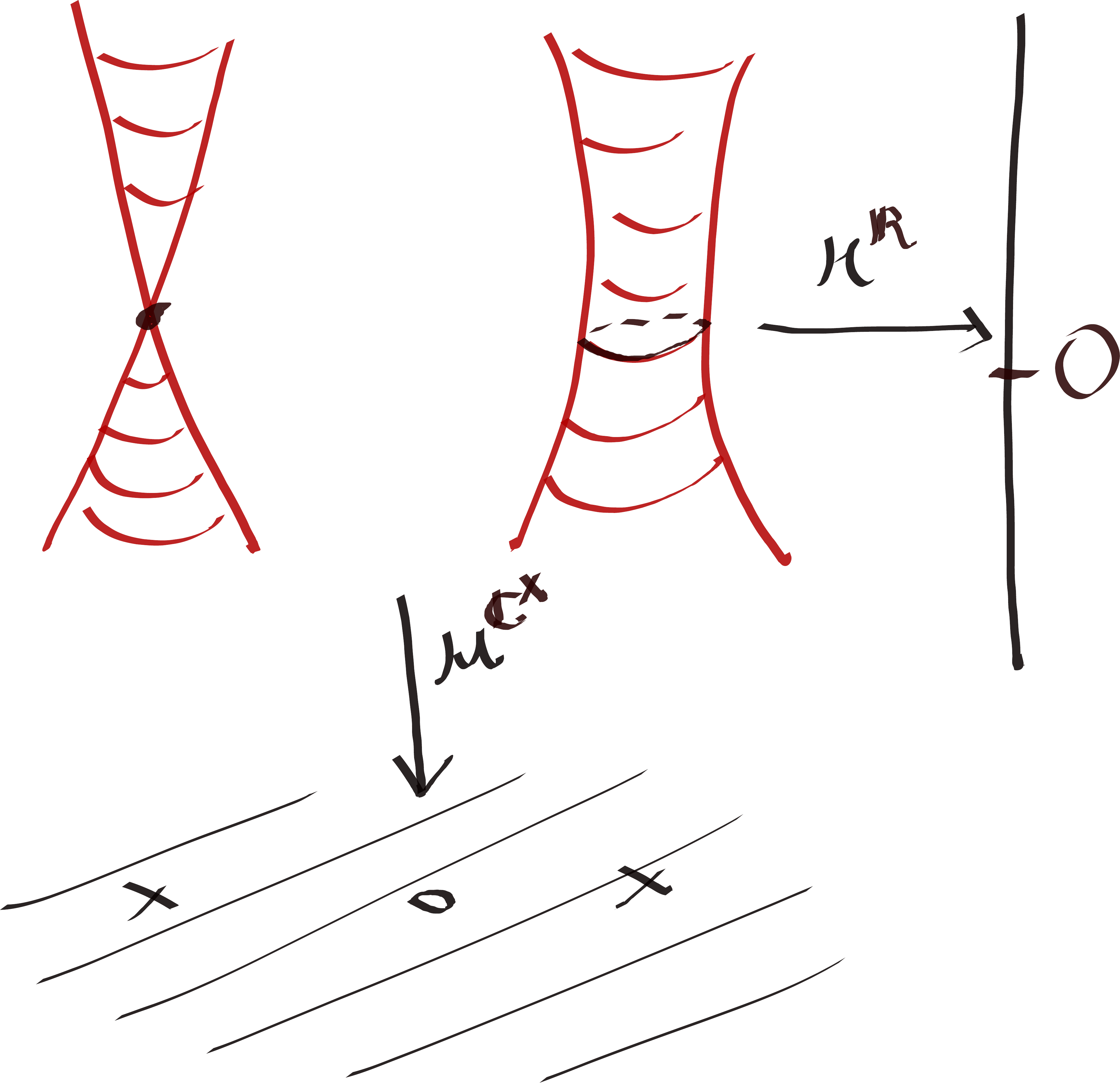}}
			\caption{A schematic picture of $\Bet$ and its various moment maps and their targets. Two fibers of $\mu^{\C^*}$ are shown, and the intersection of each fiber with $(\mu^\R)^{-1}(0)$ is indicated in black.}
		}
	\end{figure}

	\begin{lemma} \label{lem:bettiqhm} \label{lem:bettiqhkhm} \label{lem:betprincipbundle}
		The following properties are easily checked.
		\begin{enumerate} 
			\item The map $\mu_{\Bet}^{\C^*}$ is invariant under the $\C^*$ action on $\Bet$ given by 
			$\tau \cdot (x,y) = (\tau x, \tau^{-1} y)$. 
			\item The map $\mu_{\Bet}^{\R}$ is invariant under the $\uu \subset \C^*$ action. 
			\item The fibers of $\mu_{\Bet}^{\R} \times \mu^{\C^*}_{\Bet}$ over the complement of $0 \times 1$ are free $\U_1$-orbits, thus defining a principle $\U_1$-bundle $\mathcal{P}_{\Bet}$ over $\R \times \C^* \setminus 0 \times 1$. 
			\item Let $\omega = Im(dx \wedge dy)$.  
			The action of $\uu \subset \C^*$ preserves $\omega$.  
		\end{enumerate}
	\end{lemma}
	
	\begin{proposition} \label{prop:betproperties}  \label{rem: bet mu R}
		The space $\Bet = \C^2 \setminus \{ xy+1 =0 \}$ has the following properties: 
		\begin{enumerate}
			\item
			$\Bet$ is a smooth algebraic $(\C^*, \C^*)$-variety, with action $(x,y) \to (\tau x, \tau^{-1} y)$ and map $\mu_{\Bet}^{\C^*}(x,y) = 1+xy$.
			\item It has a single $\C^*$-fixed point at $(0,0)$. Every other point has trivial stabilizer.  
			\item The inclusion  $\GApoint_{(\C^*, \C^*)} \to (0,0) \in \Bet$ is a morphism of $(\C^*, \C^*)$ varieties. 
			\item The attracting cell at this fixed point is $\bS_{\Bet} \colonequals \{ x = 0 \} \cong \C$. 
			\item Note that $\bS_{\Bet} \cong \A^1$ with its natural $\C^*$-action.   
			\item The map 
			\begin{eqnarray*}
				\Bet \setminus \bS_\Bet & \to & [\C^* \times \C^*] \\
				(x, y) & \mapsto & (x, xy+1) 
			\end{eqnarray*}
			is an isomorphism of $(\C^*, \C^*)$-spaces. 
			\item  As a $(\C^*, \C^*)$-space, $\Bet$ satisfies Hypothesis \ref{hyp:bZsmoothness}.
			\item \label{other structure} $\mu_{\Bet}^{\C^*} \times \mu_{\Bet}^\R$ endows $\Bet$ with the structure of a $(\uu, \C^* \times \R)$-manifold.  
		\end{enumerate}
	\end{proposition}

	\begin{definition} Given a graph $\Gamma$ and generic $\mathbf{\eta}$, we abbreviate
		\[ \Bet(\Gamma) \colonequals \Bet^{(\C^*, \C^*)}(\Gamma, \eta). \]
	\end{definition}
	
	We will suppress the dependence on $\eta$ for most of this paper. In coordinates, it is described as follows. Recall that $\Bet^{\ed}$ has coordinates $x_e, y_e$ for $e \in \ed$. The subset $\mu_{\Gamma}^{-1}(\eta) \subset \Bet^{\ed}$ is defined by 
	\[ \prod_{\text{edges exiting } v} (1 + x_e y_e) \prod_{\text{ edges entering } v} (1 + x_e y_e)^{-1} = \eta_v \]
	for each $v \in \ver$. Then 
	\[ \Bet(\Gamma)= \mu_{\Gamma}^{-1}(\eta) / \redC^0(\Gamma\homsep\C^*) \]
	where the factor of $\C^*$ attached to $v$ acts by $\tau x_e, \tau^{-1} y_e$ on incoming edges, and $\tau^{-1} x_e, \tau y_e$ on outgoing edges. Hence $\Bet(\Gamma, \eta)$ is a smooth complex affine variety. 
	
	\subsection{Independence of orientation} \label{sec:Betorientation}
	For this paper, we will work with a fixed orientation of $\Gamma$. However, the dependence on the chosen orientation is quite mild, as shown by the following.
	\begin{proposition}
		If $\Gamma, \Gamma'$ differ only by the choice of orientation, then there is a canonical isomorphism
		\[ \Bet(\Gamma) \xrightarrow{\sim} \Bet(\Gamma)' \]
	\end{proposition}
	\begin{proof}
		By Proposition \ref{prop:indeporient}, it is enough to find an automorphism $\Bet \to \Bet$ intertwining the $\C^*$-action and the $\C^*$-moment map with their inverses. This is given by $(x,y) \to (-y, (1+xy)^{-1}x)$.  \end{proof}
	\begin{remark}
		This is an especially simple case of the proof of independence-of-orientation for multiplicative quiver varieties in \cite{CBS}.
		Under the translation to microlocal sheaves \cite{BK}, the formula $(x,y) \to (-y, (1+xy)^{-1}x)$ describes the behavior under
		Fourier transform of the `canonical' and
		`variation' maps.  
	\end{remark}
	
	\subsection{Deletion-contraction sequence} \label{sec:bettidelcon}
	
	From Proposition \ref{prop:betproperties}, 
	there are natural inclusions and projections of $(\Gm, \Gm)$-spaces  as follows: 
	\begin{equation} \label{bettiinclproj}
		\GApoint \xleftarrow{\pi} \bS_\Bet \xrightarrow{I} \Bet \xleftarrow{J} [ \Gm \times \Gm ]
	\end{equation}
	
	As $I, J$ give a decomposition of $\Bet$ into closed and open subsets, we have the following exact triangle in the derived category $D^b(\Bet)$ of constructible sheaves on $\Bet$: 
	
	\begin{equation} \label{basicadjunttri} I_!I^! \Q \to \Q \to J_*J^* \Q \xrightarrow{[1]} \end{equation}
	
	(In this paper, standard operations on sheaves, such as $I_!, I^!, J_*$ and $J^*$, are always understood to be derived.)
	
	Because $I$ is the complex codimension one closed inclusion of one smooth variety in another, 
	$$I_! I^! \Q = \Q_{\bS_\Bet}[-2](-1).$$

	\begin{proposition} \label{prop:cohomzbet}
		$\H^\bullet(\Bet\homsep \Q) \cong \Q \oplus \Q[-1](-1) \oplus \Q[-2](-2)$. 
	\end{proposition}
	\begin{proof}
		Taking sections of the triangle \eqref{basicadjunttri} returns the excision sequence in cohomology. Its terms are as follows: 
		$$\H^\bullet(\Bet; I_! I^! \Q) = \H^\bullet(\Bet\homsep \Q_{\bS_{\Bet}}[-2](-1)) = \H^\bullet(\bS_{\Bet}\homsep \Q)[-2](-1) = \Q[-2](-1)$$
		$$\H^\bullet(\Bet; J_*J^* \Q) = \H^\bullet(\Gm \times \Gm\homsep \Q) = (\Q \oplus \Q[-1](-1))^{\otimes 2} = 
		\Q \oplus \Q^{\oplus 2} [-1](-1) \oplus \Q [-2](-2)$$
		The only potentially nonvanishing map in the long exact sequence is 
		$\H^1(\Bet; J_*J^* \Q) \to \H^2(\Bet; I_! I^! \Q)$.  In fact this map must be an isomorphism, since
		$\Bet$ has nonzero Betti numbers $b^0 = b^1 = b^2 = 1$.
	\end{proof}

	\begin{hypothesis}
		$Y$ is a smooth irreducible $(\Gm, \Gm)$-variety, $\mu_Y: Y \to \Gm$ is nonconstant, and $\zeta \in \Gm$ is such that 
		$\Gm$ acts freely on $\mu_{Y \bullet \Bet}^{-1}(\zeta) \subset Y \bullet \Bet$.
	\end{hypothesis}
	
	Taking $\star_{\Gm} \colonequals \star_{\Gm, \Gm, \zeta}$ with Equation \eqref{bettiinclproj} induces morphisms:

	$$Y \star_{\starGG} \GApoint \xleftarrow{\pi_Y} Y \star_{\starGG} \bS_{\Bet} \xrightarrow{I_Y} Y \star_{\starGG} \Bet \xleftarrow{J_Y}  Y \star_{\starGG} [\Gm \times \Gm]$$ 
	
	\begin{lemma}
		$I_Y$ (resp $J_Y$) is the inclusion of a smooth divisor (resp. its complement).
	\end{lemma}
	\begin{proof}
		By Lemma \ref{lem:injsur}, $J_Y$ is the complement of $I_Y$.  Let us see that $I_Y$ is a smooth divisor. 
		Consider the divisor $Y \times \bS_{\Bet} \subset Y \times \Bet$. Since the function $\mu_{Y}$ is nonconstant on $Y$, 
		the function $\mu_{Y \bullet \Bet}$ is nonconstant on $Y \times \bS_{\Bet}$ and $Y \times \Bet$.  
		Thus $\mu_{Y \bullet \Bet}^{-1}(\zeta) \cap (Y \times \bS_\Bet) \subset \mu_{Y \bullet \Bet}^{-1}(\zeta)$ is a divisor.  
		Passing to the quotient by the (free) $\G$ action, $Y \star_{\starGG} \bS_\Bet$ is a divisor in $Y \star_{\starGG} \Bet$. 
	\end{proof}
	
	We thus obtain, as in eq. \eqref{basicadjunttri}, the triangle 
	
	\begin{equation} \label{generaladjunctriang} \Q_{Y \star \bS_{\Bet}}[-2](-1) \to \Q \to (J_Y)_*J_Y^* \Q \xrightarrow{[1]} \end{equation}
	
	\begin{lemma} \label{lem:Aone}
		The map $\pi_Y$ has fiber $\A^1$ and a section induced from the inclusion $\GApoint \to \bS_{\Bet}$. 
	\end{lemma}
	\begin{proof}
		We have $\mu_\Bet(\bS_{\Bet}) = 1$. Hence $\YBet \star_{\starGG} \bS_{\Bet} = \big( \mu_Y^{-1}(\zeta) \times \bS_{\Bet} \big) / \Gm$. By assumption, $\Gm$ acts freely on $\mu_Y^{-1}(\zeta)$. Hence the quotient is a bundle over $\mu_Y^{-1}(\zeta) / \Gm$ with fiber isomorphic to $\bS_{\Bet} \cong \A^1$. \end{proof}
	
	In particular, either push forward along $\pi_Y$ or pullback along the section induces: 
	\begin{equation} \label{scarisoms} \H^{\bullet}(\YBet \star_{\starGG} \bS_\Bet \homsep \Q) \cong \H^{\bullet}(\YBet \star_{\starGG} \GApoint\homsep \Q) \cong \H^{\bullet}(\YBet \sslash_{\zeta} \G\homsep \Q). \end{equation}

	We have $\YBet \star_{\starGG} [\Gm \times \Gm] = \YBet$  by Corollary \ref{cor:bulletunit} and 
	and $\YBet \star_{\starGG} \GApoint = \YBet \sslash_{\zeta} \Gm$ by Lemma \ref{starwithapoint}.

	Taking cohomology of the triangle \eqref{generaladjunctriang} and combining with the above isomorphisms, we obtain the diagram:  
	
	\hspace{-12mm}
	\begin{tikzcd} \label{affinepart}
		& \H^{\bullet-2}(\YBet \sslash_{\zeta} \Gm\homsep \Q)(-1) \arrow[dr, dashed]  & &  \H^{\bullet}( \YBet\homsep \Q) \arrow[dr, dashed] & \\
		\cdots \arrow[ur, dashed] \rar & \H^{\bullet-2}(\YBet \star_{\starGG} \bS_\Bet \homsep \Q)(-1)  \arrow[u, "(\bf{0} \to \bS)^*"] \arrow[r, "\text{}"] 
		& \H^{\bullet}(\YBet \star_{\starGG} \Bet\homsep \Q)  \arrow[ru, dashed]  \arrow[r, "\text{}"]  
		&  \H^{\bullet}( \YBet \star_{\starGG} [\Gm \times \Gm]\homsep \Q)   \arrow[u, "="] \rar & \cdots \\
	\end{tikzcd}
	
	We call the dashed long exact sequence the $\Bet$ deletion-contraction sequence of $Y$.  The terminology is motivated by the 
	following special case:

	\begin{theorem} \label{thm:weightdcs} 
		Let $\Gamma$ be a graph, $e$ a non-loop edge, and $\eta$ chosen such that $\Bet(\Gamma)$ is smooth.  
		Then there is a long exact sequence 
		\begin{equation} \label{eq:graphbettiLES} 
			\to \H^{\bullet-2}(\Bet(\Gamma \setminus e)\homsep \Q)(-1) \xrightarrow{a^{\Bet}} \H^{\bullet}(\Bet(\Gamma)\homsep \Q) \xrightarrow{b^{\Bet}} \H^{\bullet}(\Bet(\Gamma / e)\homsep \Q)
			\xrightarrow{c^{\Bet} [+1]} 
		\end{equation} 
		Moreover, the maps strictly preserve the weight filtration on each space (taking into account the Tate twist on the left-hand term). 
	\end{theorem} 
	\begin{remark} When we wish to highlight which edge of $\Gamma$ is in play, we may write $a^{\Bet}_e$ and $b^{\Bet}_e$.
	\end{remark}
	\begin{proof}
		Apply the $\Bet$ deletion-contraction sequence to $Y = \Bet(\Gamma / e)$ 
		and $\zeta = \eta_{t(e)}$ as specified in Section \ref{sec:gammadelcon}.  Lemmas \ref{lem:uncontract} and \ref{lem:deletionfromcontraction} 
		give the desired identifications of the convolutions in the sequence with the stated spaces.  The desired statement
		regarding weight filtrations follows from the fact that the long exact sequence of a pair respects
		mixed Hodge structures (see \cite[Prop. 8.3.9]{D3}, or more explicitly \cite[Prop. 5.4.6, 5.5.4]{PS}).  
	\end{proof}
	
	When $e$ is a bridge, the space $\Bet(\Gamma \setminus e)$ is empty by Lemma \ref{lem:disconnectedgraphs}, and the map $b^{\Bet}$ is an isomorphism. The following lemma shows that in this case we moreover have an isomorphism of spaces.
	\begin{lemma} \label{lem:contractingbridges}
		Let $e$ be a bridge of $\Gamma$. Then we have a canonical isomorphism $\Bet(\Gamma) = \Bet(\Gamma / e)$.
	\end{lemma}
	\begin{proof}
		By Lemma \ref{lem:uncontract}, we can write $\Bet(\Gamma) = \Bet(\Gamma / e) \star_{\C^*} \Bet$. On the other hand, $\Bet$ is the disjoint union of $\bS_\Bet$ and $[\C^* \times \C^*]$ as explained above.
		
		Thus $ \Bet(\Gamma / e) \star_{\C^*} \Bet$ is the disjoint union of $\Bet(\Gamma / e) \star_{\C^*} \bS_\Bet$ and $\Bet(\Gamma / e) \star_{\C^*} [\C^* \times \C^*]$ by Lemma \ref{lem:injsur}. The former is the empty set by Corollary \ref{cor:deletingbridgeskills}. The latter equals $\Bet(\Gamma / e)$ by Corollary \ref{cor:bulletunit}.
	\end{proof}

	\subsection{Class in the Grothendieck ring of varieties} \label{sec: motivic interlude}
	
	The results above also allow us to calculate the class of $\Bet(\Gamma)$ in the Grothendieck group of varieties, and prove 
	Theorem \ref{thm:motivicdc}.  We compute: 
	$$|\Bet| = |\Bet \setminus \mathbf{S}_{\Bet}| + |\mathbf{S}_{\Bet}| = | \C^* \times \C^* | + |\C| = 
	(\L - 1)^2 + \L =  \L^2 - \L + 1$$
	
	Now Lemma \ref{lem:contractingbridges} asserts 
	that $\Bet(\Gamma) = \Bet(\Gamma / e)$ if $e$ is a bridge. Moreover, 
	we know
	that $\Bet(\Gamma) = \Bet(\Gamma \setminus e) \times \Bet(\boing)$ if $e$ is a loop, by Lemma \ref{lem:loopsandotherproducts}.
	
	The key remaining point is: 
	
	\begin{corollary} \label{cor:motivicdc} 
		In the Grothendieck ring of varieties, $|Y \star_{\starGG} \Bet| = |Y| + \L \cdot |Y \sslash_{\zeta} \Gm|$.  \end{corollary}
	\begin{proof}
		By Lemma \ref{lem:Aone}, $Y \star \bS_{\Bet}$ is an $\A^1$-bundle with section over $Y \sslash_{\zeta} \Gm$. In other words, it is a 
		line bundle, and thus Zariski locally trivial over $Y \sslash_{\zeta} \Gm$. Its class in the Grothendieck ring therefore factors 
		as $\L \cdot |Y \sslash_{\zeta} \Gm|$, and the result follows.
	\end{proof}
	
	\begin{proof}[Proof of Theorem \ref{thm:motivicdc}]
		Indeed, by taking $Y = \Bet(\Gamma / e)$ in Corollary \ref{cor:motivicdc}, and substituting
		via Lemmas \ref{lem:uncontract} and \ref{lem:deletionfromcontraction}, we find
		$|\Bet(\Gamma)| = |\Bet(\Gamma / e)| +  \L |\Bet(\Gamma \setminus e)|$.  
		This gives a recursive formula for 
		$|\Bet(\Gamma)|$, with initial term $|\Bet(\bullet)| = 1$.
		The universal solution to such recursions is a sum over spanning trees $\Gamma' \in Span(\Gamma)$ described, for instance, in \cite[Chap. 10, Thm. 2]{Bol}. 
		In our case it gives the desired formula:
		\begin{equation}  |\Bet(\Gamma)| = \sum_{\Gamma' \in Span(\Gamma)} (\L - 1)^{2 b_1(\Gamma')} \L^{b_1(\Gamma) - b_1(\Gamma')}\end{equation}
	\end{proof} 
	
	We will give a second proof of Theorem \ref{thm:motivicdc}
	in Remark \ref{second motivic calculation} below, by directly exhibiting a stratification
	by spaces which account for the terms in the above sum.

	\subsection{Charts and strata} \label{betti strata}
	
	We apply the construction of Definition \ref{def: open cover} with $(D, Z) = (S_\Bet, \Bet)$ to obtain open sets $\openset_{\Gamma'}$ for $\Gamma' \subset \Gamma$.  
	Note the identification $\Bet \setminus \bS_\Bet \cong [\C^\times \times \C^{\times}]$. 
	We write: 
	$$D_e := \Bet(\Gamma) \setminus \openset_{\Gamma \setminus e} \qquad \qquad e \in \ed$$ 
	Recall we obtain $\Bet(\Gamma)$ from symplectic reduction of $\Bet^{\ed} \subset (\C^2)^{\ed}$.  
	Explicitly, $D_e$ is the symplectic reduction of the 
	locus $\{x_e = 0\} \subset \Bet(\Gamma)$.  
	(The coordinate $x_e$ transforms under a nontrivial weight of the torus we quotiented by to form $\Bet(\Gamma)$, but
	descends to a section of a line bundle, whose zero locus is still meaningful.) 
	From the definition we see
	$$\openset_{\Gamma'} = Z(\Gamma, \eta) \setminus \bigcup_{e \in E(\Gamma')} D_e$$

	\begin{lemma} \label{lem:opencover}
		The opens $\openset_{\Gamma'}$ indexed by spanning trees $\Gamma' \subset \Gamma$ define an open cover of $Z(\Gamma, \eta)$.  These open sets carry
		isomorphisms $\openset_{\Gamma'} \cong \Bet^{E(\Gamma / \Gamma')}$.
	\end{lemma}
	\begin{proof}
		The opens $\openset_{\Gamma'}$ are exactly those obtained from Lemma \ref{lem:pretreecover}, starting from the closed set $\bS_\Bet \subset \Bet$. Lemma \ref{lem:treecover} and the identification $\Bet \setminus \bS_\Bet \cong [\C^\times \times \C^{\times}]$ determine isomorphisms $\openset_{\Gamma'} \cong \Bet^{E(\Gamma / \Gamma')}$. 
	\end{proof}
	\begin{lemma} \label{lem: snc} 
		The divisor $D(\Gamma) := \bigcup D_e \subset \Bet(\Gamma)$ has simple normal crossings. 
	\end{lemma} 
	\begin{proof}
		The intersection of $D(\Gamma)$ with the opens $\openset_{\Gamma'}$ of Lemma \ref{lem:opencover} is identified with the normal crossings divisor $\prod_{e \in \Gamma \setminus \Gamma'} D_e$ under the isomorphism  $\openset_{\Gamma'} \cong \Bet^{E(\Gamma / \Gamma')}$.
	\end{proof} 
	
	For $J \subset \ed$, we write $\stratum_\subedges := \bigcap_{e \in J} D_e$.  By Lemma \ref{lem: snc} that $\stratum_\subedges$ is a real codimension $2 |J|$ submanifold.   We have the relation:

	\begin{lemma} \label{lem:vectorbundleprep}
		There is a map $\pi_J  \colon \stratum_\subedges \to \Bet(\Gamma \setminus J, \eta \setminus \subedges)$ expressing $\stratum_\subedges$ as a rank $|J|$ vector bundle over $\Bet(\Gamma \setminus J, \eta \setminus \subedges)$ and $\eta \setminus \subedges$ is defined as in Remark \ref{rem:delconJ}.
	\end{lemma} 
	\begin{proof}
		
		Consider the basic space $\Bet$. We have a diagram
		$$
		\begin{tikzcd}
			\Bet & \arrow[l] \{ x = 0 \} \arrow[d, "\pi"] \\
			\ & \bf{n} = \{ x = y = 0 \}
		\end{tikzcd}
		$$
		The map $\pi$ makes $\{x = 0\}$ a $\C^*$-equivariant rank one vector bundle over $\bf{n}$, trivialized by the function $y$. Note also that the moment map $(1 + xy)$ has constant value $1$ on this locus. 
		
		$\stratum_\subedges$ is by construction a quotient of $\bigg( \prod_{e \in J} \{ x_e = 0 \} \times \prod_{e \notin J} \Bet \bigg) \bigcap \mu^{-1}(\eta)$ by $\redC^0(\Gamma\homsep \C^*)$. The maps $\{ x_e = 0 \} \to \mathbf{n}$ for $e \in J$ combine to give a map
		$$ \bigg( \prod_{e \in J} \{ x_e = 0 \} \times \prod_{e \notin J} \Bet \bigg) \bigcap \mu^{-1}(\eta) \to \bigg( \prod_{e \in J}  \mathbf{n} \times \prod_{e \notin J} \Bet \bigg) \bigcap \mu^{-1}(\eta).$$
		
		This is a $C^1(\Gamma\homsep \C^*)$-equivariant vector bundle (if we ignore the equivariant structure, it is a trivial vector bundle) over the target with fiber $\C^J$. Taking the quotient by $\redC^0(\Gamma\homsep\C^*)$ defines a $\H^1(\Gamma \setminus J\homsep \C^*)$-equivariant vector bundle $\pi  \colon \stratum_\subedges \to \Bet(\Gamma \setminus J)$.
	\end{proof}
	
	Note that as $\stratum_\subedges$ is a real codimension $2 |J|$ submanifold, 
	there is a Gysin map 
	\begin{equation} \label{eq: DJgysin}
		\H^{\bullet - 2|J|}(\stratum_\subedges, \Q) \to \H^{\bullet}(\Bet(\Gamma), \Q)
	\end{equation} 
	
	The following lemma establishes the required commutativity we need to later define 
	a deletion filtration on the cohomology of the $\Bet(\Gamma)$. 
	
	\begin{lemma} \label{lem:gysincommutes}
		For $\Gamma' \subset \Gamma$, the map 
		$$\H^{\bullet - 2|\Gamma \setminus \Gamma'|}(\Bet(\Gamma'), \Q) \cong \H^{\bullet - 2|\Gamma \setminus \Gamma'|}(D_{\Gamma \setminus \Gamma'}, \Q) \to \H^{\bullet}(\Bet(\Gamma), \Q)$$
		is equal to the composition (in any order) of $a_e^{\Bet}$ for $e \in \Gamma \setminus \Gamma'$. 
	\end{lemma}
	\begin{proof}
		This holds by definition when $|\Gamma \setminus \Gamma'| = 1$.  In general, it 
		follows from the fact that for any ordering $\{e_1, ..., e_n\}$ of $\Gamma \setminus \Gamma'$, each inclusion in the corresponding 
		flag of subspaces $D_{\Gamma \setminus \Gamma'} \subset ... \subset D_{e_{n-1}, e_n} \subset D_{e_n} \subset D_{\emptyset} = \Bet(\Gamma)$ 
		is the inclusion of a real codimension 2 submanifold. 
	\end{proof}

	We now turn to the construction of a stratification. 
	The space $\Bet$ has a decomposition $\Bet = (\C^* \times \C^*) \sqcup \A^1$
	(see e.g. Proposition \ref{prop:betproperties} (6)).  Here we construct similarly a stratification
	of $\Bet(\Gamma)$ by vector bundles over algebraic tori.  
	
	For $\subedges \subset \Gamma$, we recall the already defined $\stratum_{\subedges}$ 
	and introduce two related spaces:  
	
	\begin{equation} \label{strata} 
		\stratum_{\subedges} := \bigcap_{e \in \subedges} D_e \qquad \qquad \closedpart_\subedges \colonequals \bigcup_{e \notin \subedges} \stratum_{\subedges \cup e} \qquad \qquad \openpart_\subedges := \stratum_{\subedges} \setminus \closedpart_\subedges
	\end{equation}

	\begin{example} \label{ex:boing}
		Let $\Gamma = \boing$ and let $e$ be the only edge. Then $S_\emptyset = \Bet(\boing), K_\emptyset = \bS_\Bet = \{ x = 0 \}$ and $Q_\emptyset = \Bet(\boing) \setminus \bS_\Bet = \{ x \neq 0 \}$. On the other hand $S_e = \bS_{\Bet} = Q_e, K_e = \emptyset$. 
	\end{example}  
	\begin{lemma} \label{lem: empty}
		$\stratum_{\subedges} = \emptyset$ if $\Gamma \setminus \subedges$ is disconnected.
	\end{lemma}
	\begin{proof}
		Combine Lemma \ref{lem:vectorbundleprep}, Lemma \ref{lem:disconnectedgraphs} and Remark \ref{rem:delconJ}.
	\end{proof}

	\begin{proposition} \label{prop:vectorbundleovertorus}
		Suppose $\Gamma \setminus \subedges$ is connected.  
		The restriction of $\pi_\subedges$ to  $\openpart_\subedges$ defines a $\H^1(\Gamma \homsep 	\C^*)$-equivariant vector bundle over $\H^1(\Gamma \setminus \subedges\homsep \C^*) \times \H_1(\Gamma \setminus \subedges\homsep \C^*)_{\eta \setminus J} $. 
	\end{proposition}
	\begin{proof}
		From Lemma \ref{ex:cohomologicalquotient}, we have: 
		$\Bet(\Gamma, \eta) \setminus D(\Gamma) \cong [\C^* \times \C^*](\Gamma, \eta) \cong \H^1(\Gamma, \C^*) \times \H_1(\Gamma, \C^*)_\eta$,
		where the middle term is the graph space associated to $[\C^* \times \C^*]$, in the notation of 
		Definition \ref{definitionofgraphspace}. 
		The stated result now follows from Lemma \ref{lem:vectorbundleprep}, after 
		substituting $\Gamma \setminus \subedges$ for $\Gamma$ in the previous line.
	\end{proof}

	\begin{proposition} \label{stratification}
		$\Bet(\Gamma, \eta) = \coprod \openpart_{\subedges}$. 
	\end{proposition} 
	\begin{proof}
		By construction, $Q_J$ is the locus of points in $Z(\Gamma)$ contained in $D_e, e \in J$ and not in $D_e, e \notin J$.  
	\end{proof}

	\begin{remark} \label{second motivic calculation} Combining
		Lemma \ref{lem: empty}, Proposition \ref{prop:vectorbundleovertorus}, 
		and Proposition \ref{stratification}, we get another (and more explicit)
		proof of Theorem \ref{thm:motivicdc} \eqref{eq:mbetmotive}.  
	\end{remark}

	\begin{remarque}
		Going back at least to the work of Deodhar \cite{Deo},  stratifications by $(\C^*)^a \times \C^b$ have been found frequently in representation theoretic contexts.  At least since
		\cite[Prop. 6.31]{STZ}, we have known that such stratifications often admit 
		modular interpretations: often the spaces
		are moduli of objects in the Fukaya category of a symplectic 4-manifold; and the strata each parameterize objects coming from
		a given immersed Lagrangian.  
		
		The present case is presumably another example.  From \cite{BK} we learn
		that $\Bet(\Gamma)$ is a moduli space
		of microlocal sheaves on a singular real surface $L = \bigcup L_i$, where the $L_i$ are the smooth irreducible components.  In this context it is most natural to view $L$ 
		as the Lagrangian skeleton of the symplectic plumbing $W$ of the $T^*L_i$.  
		By e.g. \cite[Cor. 6.3]{GPS3} we may trade microlocal sheaves on $L$ for the wrapped Fukaya category of $W$. 
		Now a  spanning
		subgraph $\Gamma' \subset \Gamma$ determines an immersed Lagrangian:  smooth the singularities of the 
		skeleton corresponding to the edges $\Gamma'$, and leave the nodes in $\Gamma \setminus \Gamma'$.  A rank one local
		system on this Lagrangian, together with some extra data at the nodes, determines an object in the Fukaya category.  The space
		of such choices is $(\L - 1)^{2 b_1(\Gamma')} \L^{|\Gamma| - |\Gamma'|}$.  It is also possible to 
		give a similar description directly in terms
		of the microlocal picture of \cite{BK}. 
	\end{remarque}

	\subsection{ $\CKS(\Gamma)$ via differential forms on $\Bet(\Gamma)$}

	The purpose of this subsection is to show that the complex $\CKS(\Gamma)$
	computes the cohomology of $\Bet(\Gamma)$.  
	This  will be done using differential forms and the residue sequence (see Section \ref{app:residuetriang}) for a certain stratification of $\Bet(\Gamma)$. 
	As the space  $\Bet(\Gamma)$ is affine, as are all strata we encounter, we 
	will everywhere take termwise global sections in complexes of sheaves of differential forms, 
	and discuss the resulting complexes of sections
	rather than complexes of sheaves.   We preserve the notations of Section \ref{betti strata}, 
	especially for the spaces \eqref{strata}.

	As $Q_J$ has the homotopy type of a torus (Prop. \ref{prop:vectorbundleovertorus}), 
	there is a quasi-isomorphism
	$\Lambda^* \H^1(\openpart_{\subedges}, \C) = \H^*(\openpart_{\subedges}, \C)
	\xrightarrow{\sim} (\Omega^* \openpart_{\subedges}, d_{dR})$; here the LHS 
	has the zero differential.  Let us recall
	how to exhibit such  explicitly.  
	Let $V$ be a complex vector space with lattice $V_\Z$ and dual $V^*$. Then any $w \in V^*$ defines a 1-form $dw \in \Omega^1(V)$, which descends to a closed one form (which we denote by the same symbol $dw$) on the torus $V/V_\Z$. This defines a linear map 
	$ \H^1(V/V_\Z, \C) \cong V^* \to \Omega^1(V/V_\Z)$; taking exterior powers, we obtain
	the quasi-isomorphism $\Lambda^* V^* \xrightarrow{\sim} (\Omega^*(V/V_\Z), d_{dR})$.

	\begin{definition} \label{def:CKSmap}
		Consider $V = \H_1(\Gamma \setminus J\homsep \C) \oplus \H^1(\Gamma \setminus J\homsep \C)$;
		it is is self dual, with lattice $V_\Z \colonequals \H_1(\Gamma \setminus J, \Z) \oplus \H^1(\Gamma \setminus J, \Z)$ and quotient $V/V_\Z = \H_1(\Gamma \setminus J\homsep \C^*) \times \H^1(\Gamma \setminus J\homsep \C^*)$.  
		Thus we obtain a map  
		\begin{equation} \label{eq:CKSmap} \mathbb{H}(\Gamma \setminus J, \C) \colonequals \H_1(\Gamma \setminus J\homsep \C) \oplus \H^1(\Gamma \setminus J\homsep \C) \to \Omega^1(\H_1(\Gamma \setminus J\homsep \C^*) \times \H^1(\Gamma \setminus J\homsep \C^*)). \end{equation} 
		
		Composing with the pullback along the vector bundle $\pi_J  \colon \openpart_J \to \H_1(\Gamma \setminus J\homsep \C^*) \times \H^1(\Gamma \setminus J\homsep \C^*)$, we obtain a map
		$$ \CKSmap_J  \colon \mathbb{H}(\Gamma \setminus J, \C)  \to \Omega^1_{\openpart_\subedges}.$$
		We use the same notation for the induced quasi-isomorphism
		$\CKSmap_J : \Lambda^* \mathbb{H}(\Gamma \setminus J, \C)  \xrightarrow{\sim} (\Omega^1_{\openpart_\subedges}, d_{dR})$.
	\end{definition}

	\begin{example}
		We continue Example \ref{ex:boing}. We have $\H(\boing, \C) = \C \gamma \oplus \C e$. Then $\CKSmap(\gamma) = \frac{dx}{x}$ and $\CKSmap(e) = \frac{d(xy+1)}{xy+1}$. 
	\end{example}

	\begin{remark}
		Using the stratification of $\Bet(\Gamma)$ by the $Q_J$ (Prop. \ref{stratification}) 
		and the above identification of 
		the cohomologies of the $Q_J$, it follows formally from excision sequences
		that $\H^*(\Bet(\Gamma), \C)$ can be computed by a complex with underlying
		graded vector space $\bigoplus_J \Lambda^* \mathbb{H}(\Gamma \setminus J, \C)$,
		whose differential respects the filtration by $|J|$.  Identifying the differential with that
		of $\CKS(\Gamma)$, however, necessarily involves understanding how the closure
		of one strata meets others.  It is this which we accomplish using log forms and residues, below.
	\end{remark}

	We recall the notation $D^{J} = \bigcup_{e \in J} D_e$ from Appendix \ref{app:residuetriang}. 
	\begin{lemma}
		$\closedpart_{\generalsubedges} = D^{\generalsubedges^c} \cap \generalstratum_\generalsubedges$.
	\end{lemma}
	\begin{proof}
		A point lies in $S_J$ if it is contained in $D_e$ for all $e \in J$. A point lies in $K_J$ if it is contained in $D_e$ for all $e \in J$ and at least one $e \notin J$.  
	\end{proof}

	By Lemma \ref{lem: snc}, $\closedpart_\subedges$ is a simple normal crossings divisor in $\stratum_{\subedges}$. 
	The key point that allows us to work with the log de Rham complex is: 
	
	\begin{proposition} \label{cor:CKSlogpoles}
		The image of $\CKSmap_J: \mathbb{H}(\Gamma \setminus J, \C)  \to \Omega^1_{\openpart_\subedges}$ lies in $\Omega^1_{\stratum_\subedges} \langle \closedpart_J \rangle$. 
	\end{proposition}
	\begin{proof}
		The proposition asserts that the forms in the image of $\CKSmap_J$, a priori defined on 
		$\openpart_\subedges$, in fact extend to meromorphic forms on $\stratum_\subedges$ 
		with logarithmic 
		poles along divisors of the form $\openpart_{\subedges \cup e}$. 
		
		We check this using certain open charts which contain both $\openpart_\subedges$ and $\openpart_{\subedges \cup e}$. For each spanning tree $\Gamma' \subset \Gamma$, Lemma \ref{lem:opencover} gives an open $\openset_{\Gamma'}$ and an isomorphism $ \openset_{\Gamma'} \cong \Bet^{E(\Gamma \setminus \Gamma')}.$ The intersection $\bigcap_{\Gamma'} \openset_{\Gamma'}$ is precisely $\Bet(\Gamma) \setminus D^{\Gamma} $, and the composition
		\[ \H^1(\Gamma, \C^*) \times \H_1(\Gamma, \C^*) \cong \Bet(\Gamma) \setminus D^{\Gamma} \cong (\Bet \setminus \bS_\Bet)^{E(\Gamma \setminus \Gamma')} \cong (\C^{\times} \times \C^{\times})^{E(\Gamma \setminus \Gamma')} \]
		is described by Lemma \ref{lem:coordinatesoncharts}. Now suppose $\Gamma' \subset \Gamma \setminus \subedges$. There is a corresponding contraction $\Gamma \setminus \subedges \to (\Gamma \setminus \subedges) / \Gamma' = \boing^{E(\Gamma \setminus \Gamma' \cup \subedges)}$, inducing
		\begin{equation} \label{eq:toruscoordinatesagain} 
			\H_1(\Gamma \setminus \subedges\homsep \C^*) \times \H^1(\Gamma \setminus \subedges\homsep \C^*) \cong (\C^{\times} \times \C^{\times})^{E(\Gamma \setminus \Gamma' \cup \subedges)}.
		\end{equation}
		We have an inclusion of opens
		\begin{equation} \label{eq:openincludes} \openpart_{\subedges} \subset \stratum_\subedges \setminus \bigcup_{e \in \Gamma'} \stratum_{\subedges \cup e}.  \end{equation}
		The right-hand side equals 
		\begin{equation} \label{eq:trivialisethestrata}
			\openset_{\Gamma'} \cap \stratum_{\subedges} \cong \Bet^{\Gamma \setminus (\Gamma' \cup \subedges)} \times \bS_\Bet^{\subedges}.
		\end{equation}
		Under the isomorphism \eqref{eq:trivialisethestrata}, a dense open subset of the divisor $\stratum_{\subedges \cup e} \subset \stratum_{\subedges}$ is identified with $\{ x_ e = 0 \} \subset   \Bet^{\Gamma \setminus (\Gamma' \cup \subedges)} \times \bS_\Bet^{\subedges}$. Moreover, the composition of \eqref{eq:openincludes} and \eqref{eq:trivialisethestrata} fits into a commutative diagram 
		\begin{equation} \label{eq:isomoftorifromspanning} 
			\hspace{-12mm}
			\begin{tikzcd} 
				& \openpart_{\subedges} \arrow[r, "\sim"] \arrow[d, "\pi_\subedges"]  & (\Bet \setminus \bS_\Bet)^{\Gamma \setminus (\Gamma' \cup \subedges)} \times  \bS_\Bet^{\subedges} \arrow[r] \arrow[d]  & \Bet^{\Gamma \setminus (\Gamma' \cup \subedges)} \times  \bS_\Bet^{\subedges} \arrow[d] \\
				& \H_1(\Gamma \setminus \subedges\homsep \C^*) \times \H^1(\Gamma \setminus \subedges\homsep \C^*)   \arrow[r, "\sim"] & (\C^{\times} \times \C^{\times})^{\Gamma \setminus (\Gamma' \cup \subedges)}  \arrow[r] & \Bet^{\Gamma \setminus (\Gamma' \cup \subedges)} \\
			\end{tikzcd}
		\end{equation}
	
	where the bottom left map is \eqref{eq:toruscoordinatesagain} and the middle vertical and bottom right maps are induced by $\Bet \setminus \bS_\Bet \cong \C^\times \times \C^\times$.
	
	Pulling back by the bottom left isomorphism of Diagram \eqref{eq:isomoftorifromspanning}, the map \eqref{eq:CKSmap} is identified with the map
	\[ \mathbb{H}(\Gamma \setminus \subedges, \C) \to (\Omega^1(\C^{\times} \times \C^{\times}))^{\Gamma \setminus (\Gamma' \cup \subedges)} \]
	given by 
	\begin{equation} \label{eq:logarithmicpoleformula} (\gamma, \theta) \to \frac{1}{2\pi i} \left(\langle e, \gamma \rangle \frac{d x_e}{x_e }  + \theta_e \frac{d(x_e y_e + 1)}{x_ey_e + 1} \right). \end{equation}
	Here $e$ ranges over $\Gamma \setminus (\Gamma' \cup J)$ and $\theta_e$ is defined by $\theta = \sum_{e \in \Gamma \setminus (\Gamma' \cup J)} \theta_e e$, as in Lemma \ref{lem:coordinates}. Via the lower right-hand map in Diagram \eqref{eq:isomoftorifromspanning}, we can view the map \eqref{eq:logarithmicpoleformula} as a meromorphic form on $\Bet^{\Gamma \setminus (\Gamma' \cup J)}$.
	
	By Equation \eqref{eq:logarithmicpoleformula}, the image of $\CKSmap_\subedges$ has logarithmic singularities along $\stratum_{\subedges \cup e} = \{ x_e = 0 \}$ for any $e \in \Gamma \setminus (\Gamma' \cup J)$. Since we can pick $\Gamma'$ to avoid any $e \notin \subedges$ for which the divisor $\stratum_{\subedges \cup e}$ is nonempty, the image has logarithmic singularities along $\closedpart_\subedges \subset \stratum_\subedges$.
\end{proof} 

\begin{corollary} \label{log forms on torus} 
	The induced map 
	$\CKSmap_J  \colon \bigwedge^{\bullet}\mathbb{H}(\Gamma \setminus J, \C)  \to (\Omega^{\bullet}_{\stratum_{\subedges}}\langle \closedpart_J \rangle, d_{dR})$ is a quasi-isomorphism.  
\end{corollary} 
\begin{proof}
	This map, composed with the quasi-isomorphism from the log de Rham complex to the 
	ordinary de Rham complex on $\openpart_J$, is a quasi-isomorphism. 
\end{proof}

\begin{proposition} \label{prop:CKSmapstoforms}
	The following diagram commutes.
	\begin{equation}
		\begin{tikzcd}
			\bigwedge^{\bullet} \mathbb{H}(\Gamma \setminus J, \C) \arrow[d, "d_{e}"] \arrow[r, "\CKSmap_J"] & \Omega^{\bullet}(\stratum_\subedges \langle \closedpart_J \rangle ) \arrow[d, "\operatorname{res}_{J \to J \cup e}"] \\
			\bigwedge^{\bullet-1} \mathbb{H}(\Gamma \setminus J \cup e, \C) \arrow[r, "\CKSmap_{J \cup e}"] & \Omega^{\bullet - 1} (\stratum_{\subedges \cup e} \langle \closedpart_{J \cup e} \rangle)  
		\end{tikzcd}
	\end{equation}
\end{proposition}
\begin{proof}
	We will prove the case $J = \emptyset$; the other cases are identical. As in the proof of Lemma \ref{lem:imageofde}, we can write $\mathbb{H}(\Gamma, \C) = \mathbb{F} \oplus \mathbb{K}$ where $\mathbb{K} = \H_1(\Gamma \setminus e\homsep \C) \oplus \H^1(\Gamma \homsep \C)$ and $\mathbb{F}$ is any complementary rank-one subspace. We also fix a spanning tree $\Gamma' \subset \Gamma \setminus (\subedges \cup e)$. 
	
	Let us first assume $\bullet = 1$. By Equation \eqref{eq:logarithmicpoleformula}, we see that $\operatorname{res}_{J \to J \cup e} \circ \CKSmap_J$ factors through the projection to $\mathbb{F}$, and is given by the pairing $\langle e, \gamma \rangle$. This proves the result when $\bullet=1$. 
	
	For $\bullet > 1$, we can write $\bigwedge^{\bullet} \mathbb{H}(\Gamma, \C ) = \mathbb{F} \wedge \bigwedge^{\bullet - 1} \mathbb{K}  \oplus \bigwedge^{\bullet} \mathbb{K}$. Then $\operatorname{res}_{J \to J \cup e} \circ \CKSmap_J$ kills the second summand, and acts on the first by $\operatorname{res}_{J \to J \cup e} \circ \CKSmap(\sigma \wedge \tau) = \langle e, \sigma \rangle \CKSmap(\tau)_{\stratum_{J \cup e}}$ where $\CKSmap(\tau)_{\stratum_{J \cup e}}$ is the restriction of $\CKSmap(\tau)$ to $\stratum_{J \cup e}$. To compute this restriction, we use Equation \eqref{eq:logarithmicpoleformula}, which immediately implies $\CKSmap_J(\tau)_{\stratum_{J \cup e}} = \CKSmap_{J \cup e}(\tau)$. 
	
	Comparing with the formula for $d_e$ yields the result.
\end{proof}

We now use the residue  exact triangles from Appendix \ref{app:residuetriang}.  

\begin{theorem} \label{thm:quasiisom} 
	Define $\CKSmap$ by taking the direct sum of the maps $\CKSmap_J$ for all $J$. Then $\CKSmap$ induces an inclusion of complexes 
	$\CKS^{\bullet}( \Gamma, \C ) \to {}^D \Omega^{\bullet}_{\Bet(\Gamma)}$ (see Def. \ref{omega x d}), which
	is in fact a quasi-isomorphism. 
	In particular, this induces a canonical isomorphism $\H^{\bullet}(\CKS^{\bullet}( \Gamma, \C)) \xrightarrow{\sim} \H^{\bullet}(\Bet(\Gamma)\homsep \C)$.
\end{theorem}
\begin{proof}
	First, let us observe it is an inclusion of bigraded (by degree of wedge and size of $J$) vector spaces.  
	
	The de Rham differential vanishes on the image of $\CKSmap$, since by construction it is composed of wedge products of closed forms. Proposition \ref{prop:CKSmapstoforms} shows that $d_{\operatorname{res}}$ restricts to $d_{\CKS}$.
	
	The fact that the map is a quasi-isomorphism can be seen as follows. Filter both complexes by the size of $J$, and consider the associated map of spectral sequences. It induces an isomorphism in cohomology on the first page by Corollary \ref{log forms on torus}, and thus on all subsequent pages.  
	
	For the assertion regarding cohomology, note that 
	since $\Bet(\Gamma)$ is affine, 
	Proposition \ref{prop:cousincomputescohomology} yields a quasi-isomorphism $\Omega^{\bullet}_{\Bet(\Gamma)} \to {}^D\Omega^{\bullet}_{\Bet(\Gamma)}$.  This combines with the quasi-isomorphism $\CKSmap \colon \CKS^{\bullet}( \Gamma, \C ) \to {}^D \Omega^{\bullet}_{\Bet(\Gamma)}$
	to give the stated result. 
\end{proof}

\begin{remarque}
	In fact, the above argument shows that $\CKS^{\bullet}( \Gamma, \C ) \to {}^D \Omega^{\bullet}_{\Bet(\Gamma)}$ is
	an isomorphism in the filtered derived category, where both sides are filtered by the size of $J$. 
\end{remarque}

\begin{definition} \label{def:bettideletionfiltration}
	The {\em Betti deletion filtration} is the increasing filtration obtained from Definition \ref{def:delfiltration}, where the covariant functor $A$ is defined on objects by 
	$\Gamma \mapsto \H^{\bullet}(\Bet(\Gamma), \Q)$ and on morphisms by taking the inclusion $\Gamma' \to \Gamma$ to the composition of deletion maps 
	$a_e^{\Bet}$ (in any order) for $e \in \Gamma \setminus \Gamma'$.  
	(Independence of ordering follows from Lemma \ref{lem:gysincommutes}.) 
\end{definition}

\begin{proposition} \label{cksdelequalsbettidel}
	The map	$\CKSmap$ identifies the $\CKS$ deletion-contraction sequence with the Betti deletion contraction sequence, and the $\CKS$-deletion filtration with the Betti deletion filtration.  
\end{proposition} 
\begin{proof} 
	The quasi-isomorphism $\CKSmap$ of Theorem \ref{thm:quasiisom}, and its analogues for $\Gamma \setminus e$ and $\Gamma / e$, 
	defines a map of short exact sequence from Sequence \eqref{CKSshortexact} to the exact sequence  
	$$0 \to {}^{\closedpart_e}  \Omega^{\bullet - 2}_{D_e} \to {}^D \Omega^{\bullet}_{X} \to  {}^D \Omega_{X \langle D_e \rangle} \to 0$$
	of Proposition \ref{LESdeletion}.  
	(The commutativity of the $\CKSmap$ with the morphisms of in sequences holds because both exact sequences are the exact sequences
	associated to a cone, and the cone is over morphisms which we have already seen to be compatible in Proposition \ref{prop:CKSmapstoforms}.)
	Applying the comparison from Corollary \ref{cor:lesofapairdesc} recovers the deletion-contraction sequence. 
\end{proof}

\begin{proposition} \label{delconstrictlycompdelfilt}
	The Betti deletion-contraction sequence is strictly compatible with the Betti deletion filtration.
\end{proposition}
\begin{proof}
	Follows from Proposition \ref{cksdelequalsbettidel} and the corresponding fact for 
	the $\CKS$ deletion-contraction sequence (Cor. \ref{cor:cksstrictness}). 
\end{proof}

\begin{remark}
	The dependence of the remainder of this article on the present subsection
	factors through the statement of 
	Proposition 
	\ref{delconstrictlycompdelfilt}, which does 
	not involve $\CKS(\Gamma)$.  One could imagine this statement has a proof
	which does not require the comparison with $\CKS(\Gamma)$, but we do not 
	know one. 
\end{remark}

\subsection{Deletion versus weight filtrations}

\begin{theorem} \label{deletionisweight} 
	The weight filtration is given by doubling the Betti deletion filtration: 
	$$W_{2k} \H^{\bullet}(\Bet(\Gamma),\Q) = W_{2k+1} \H^{\bullet}(\Bet(\Gamma),\Q) = \mathrm{D}_k \H^{\bullet}(\Bet(\Gamma), \Q).$$
\end{theorem}

\begin{proof}
	First let us check the result when every edge of $\Gamma$ is a bridge or a loop. 
	By Definition \ref{def:delfiltration}, in this case the first nonvanishing
	step of the deletion filtration is $\delfilt_{i} \H^i(\Bet(\Gamma), \Q) = \H^i(\Bet(\Gamma), \Q)$.  
	We compute the weights.  Recall from Lemma \ref{lem:contractingbridges} that contracting bridges does not 
	change $\Bet(\Gamma)$; correspondingly we may as well assume $\Gamma$ is a vertex
	with $n$ loops.  For this space
	$$\H^\bullet(\Bet(\Gamma), \Q) = \H^\bullet(\Bet(\boing)^n, \Q) = \H^\bullet(\Bet, \Q)^{\otimes n}$$
	It follows from Proposition \ref{prop:cohomzbet} that the degree $i$ cohomology of the RHS 
	has weight $2i$, as desired.  
	
	Next, let us show that for any $\Gamma$, we have 
	\begin{equation}
		\label{eq: del in weight} 
		\delfilt_k \H^{\bullet}(\Bet(\Gamma), \Q) \subset W_{2k}\H^{\bullet}(\Bet(\Gamma), \Q)
	\end{equation}
	We proceed by induction on the number $N$ of edges that are neither bridges nor loops.  We have already treated
	the case $N=0$.  Suppose $N>0$. Note that $\delfilt_k \H^{\bullet}(\Bet(\Gamma), \Q)$ is spanned by the images under $a_e^{\Bet}$ for various $e \in \ed$ of $\delfilt_{k-1}\H^{\bullet-2}(\Bet(\Gamma \setminus e), \Q)$. Consider the sequence 
	$$\to \H^{i - 2}(\Bet(\Gamma \setminus e)\homsep \Q) \otimes \Q(-1) \xrightarrow{a_e^{\Bet} } \H^{i}(\Bet(\Gamma)\homsep \Q) \xrightarrow{b^{\Bet}_e} \H^{i}(\Bet(\Gamma / e)\homsep \Q) \xrightarrow{c_e^{\Bet}} $$
	By induction, we have that 
	$\delfilt_{k-1} \H^{i - 2}(\Bet(\Gamma \setminus e)\homsep \Q) \subset W_{2k - 2}\H^{i - 2}(\Bet(\Gamma \setminus e)\homsep \Q)$. 
	Taking into account the Tate twist, and noting that $a_e^{\Bet}$ preserves the weight filtration, 
	it follows that the image of $\delfilt_{k-1} \H^{i - 2}(\Bet(\Gamma \setminus e)\homsep \Q)$ lies 
	in $W_{2k}\H^{i}(\Bet(\Gamma)\homsep \Q)$, as desired.	 This completes the proof of 
	\eqref{eq: del in weight}.  
	
	Finally we are interested in upgrading the inclusion of \eqref{eq: del in weight} to an equality.  
	By \eqref{eq: del in weight}, the identity on $\H^{\bullet-2}(\Bet(\Gamma \setminus e), \Q)$
	induces a map of associated graded spaces; it will suffice to show that this map is an isomorphism.  
	Again we proceed by induction on $N$, having already established the case $N=0$. 
	
	Pick an edge $e \in \Gamma$ which is neither a bridge nor a loop. 
	By Theorem \ref{thm:weightdcs}, the deletion-contraction sequence for $e$ is strictly compatible with the weight-filtration.
	We write $\operatorname{gr}^{W}(DCS_e)$ for the sequence obtained from the deletion-contraction sequence by taking 
	the associated graded spaces for the weight filtration.  Strict compatibility implies that  $\operatorname{gr}^{W}(DCS_e)$
	is exact (see Lemma \ref{lem:strictmapsexactassocg}).   Similarly we write $\operatorname{gr}^{2\delfilt}(DCS_e)$ for the sequence 
	obtained from the deletion-contraction sequence by doubling indices and taking the associated graded spaces for the deletion filtration.
	
	Now, by strict compatibility of the deletion-contraction sequence with the Betti deletion filtration (Proposition \ref{delconstrictlycompdelfilt}), we may conclude the sequence $\operatorname{gr}^{2\delfilt}(DCS_e)$  is also exact.  
	
	Consider the map $\pi  \colon \operatorname{gr}^{2\delfilt}(DCS_e) \to \operatorname{gr}^{W}(DCS_e)$ 
	arising from the inclusion \eqref{eq: del in weight}.  By induction, $\pi$ is an isomorphism for the terms 
	associated to $\Gamma \setminus e$ and $\Gamma / e$.  Thus by the five lemma, $\pi$ is an isomorphism
	for the $\Gamma$ terms, as well.  This completes the proof.  
\end{proof}

\section{$\Dol$}

\subsection{Complex analytic structure} \label{subsec:basicdolbdef}
The space $\Dol$ will be a neighborhood of the nodal rational curve with dual graph $\boing$ 
inside a family of genus one curves.  We consider the universal cover of the universal deformation of $\P^1/\{0 = \infty\}$, 
as described in e.g. \cite[Sec. VII]{DR}, \cite[p. 135]{Mum-curves}.  

Let $\D_1 \subset \C$ be the interior of the unit disk, and $\D_1^*$ the punctured disk.  Take $q$ the coordinate on $\D_1$. 
One can form over $\D_1^*$ the family of genus one curves with fiber $\C^*/q^\Z$; it is by definition a quotient
$(\D_1^* \times \C^*)/\Z$.  

The monodromy of this family is such that it is natural to fill the special fiber by 
the rational curve with dual graph $\boing$, and one wants to extend the quotient description accordingly.  The picture 
is that one takes $\D_1 \times \P^1$, iteratively blows up the points at the intersection of the fiber over zero
and the strict transforms of the sections $\D_1 \times 0$ and $\D_1 \times \infty$, and then finally deletes
these sections.  The result has central fiber an infinite chain of $\P^1$.  

It is now possible to extend the $\Z$ action, as can be verified most easily in the following coordinate description. We consider the slightly larger space $\C \times \P^1$. It is a toric variety under the natural action of $\C^* \times \C^*$, and the iterated blow-ups (at torus-fixed points) inherit this toric structure. Each blow-up admits a compatible set of toric charts $\C^2_n \cong \C^2$, with coordinates $x_n, y_n$, glued by identifying 
$$\C^2_n \setminus \{x_n = 0\} \leftrightarrow \C^2_{n+1} \setminus \{y_{n+1} = 0\}$$
by the relations 
$$x_n y_n = x_{n+1} y_{n+1} $$
$$x_n = y_{n+1}^{-1}$$
If we glue all such charts $\C^2_n$ for $n \in \Z$, we obtain a complex manifold $\frak{W}$.

\begin{lemma} \label{lem: W structures}
	$\frak{W}$ carries the following structures: 
	\begin{enumerate}
		\item A holomorphic symplectic form $\Omega = dx_n \wedge dy_n$.
		\item A holomorphic function $q: \frak{W}  \to  \C$ given by 
		$q(x_n, y_n) = x_n y_n$.
		\item A $\C^*$ action  $(x_n, y_n) \to (\tau x_n, \tau^{-1} y_n)$, preserving $\Omega$ and the fibers of $q$.
	\end{enumerate}
\end{lemma}
\begin{proof}
	One checks that the formulas given descend along the above specified gluing. 
\end{proof}

\begin{proposition}
	The $\Z$ action on $\frak{W}$ defined by $k + (x_n, y_n) = (x_{n+k}, y_{n+k})$ is free
	and discontinuous over $q^{-1}(\D)$.  
\end{proposition} 
\begin{proof}
	We check
	the freeness and discontinuousness separately for $z \in \tilde{q}^{-1}(\D_1^*) \cong \D^* \times \C^*$, and $z \in \tilde{q}^{-1}(0)$. 
	In the former case, $n \in \Z$ acts by multiplication by $q^{n}$ on the $\C^*$-factor, which is free and discontinuous if 
	$|q| \neq 1$. In the latter case, $n \in \Z$ acts by translating the infinite chain $\tilde{q}^{-1}(0)$ by $n$ steps.
\end{proof}

\begin{definition}
	We write $\widetilde{\Dol} = {q}^{-1}(\D_1)$ and $\Dol =     {q}^{-1}(\D_1)/\Z$.   We write $\mathbf{n} \in \Dol$ for the common image
	of the points $(x_n, y_n) = (0,0)$.  
\end{definition}

The structures from Lemma \ref{lem: W structures} restrict to $\widetilde{\Dol}$ and descend to $\Dol$; we keep the 
same notation for the resulting structures. Note the $\C^*$ action on $\Dol$ factors through the quotient 
$\C^* / q^\Z$ away from $q^{-1}(0)$.  

\begin{proposition}
	We have the following: 
	\begin{enumerate}
		\item The map $q: \Dol \to \D_1$ is proper and holomorphic, with unique critical point $\mathbf{n}$.
		\item The fiber $q^{-1}(0)$ is a nodal rational curve, with node $\mathbf{n}$. 
		\item $q$ admits a section with image disjoint from $n$ and given in coordinates by $x_{2n} = y_{2n+1}=1$. 
		\item The $\U_1 \subset \C^*$ action on $\Dol$ is free away from its unique fixed point $\mathbf{n}$.
		\item There exist coordinates around $\mathbf{n}$ in which the $\U_1$ action is $\tau \cdot (x, y) = (\tau x, \tau^{-1} y)$.  
	\end{enumerate}
\end{proposition} 
\begin{proof}
	These can be checked in the coordinates given in Lemma \ref{lem: W structures}. 
\end{proof}

Our construction of $\Dol$ defines a manifold with a $\uu$-action, a $\uu$-invariant integrable complex structure $I_{\Dol}$, and a $\uu$-stable $(I_{\Dol}$-holomorphic) elliptic fibration to the open disk $\D_1$ with special fiber a nodal elliptic curve. 

The $\uu$-action can be recovered from the holomorphic map $\Dol \to \D_1$: 

\begin{lemma} \label{lem:uniqueaction} 
	Let $X$ be a smooth complex analytic surface and $q: X \to \D_1$ an elliptic fibration over a disk, with a single singular fiber at the origin
	of Kodaira type $I_1$ (reduced irreducible rational nodal elliptic curve).   
	
	Then there is a canonical action of $\U_1$ on $X$.  Said action is free away from the singularity of the central fiber, and preserves the fibers of $q$.  The quotient  
	$X / \U_1 \to \D_1$ is topologically a circle bundle.   
\end{lemma}
\begin{proof}
	The smooth locus of the fibers is identified with $\mathrm{Pic}^1(X/B)$, hence carries a $\mathrm{Pic}^0(X/B)$ action.  The 
	universal cover $\tilde{X}$ is topologically a $\C^\times$-bundle over the punctured disk, 
	with an infinite chain of rational curves over the origin; the $\mathrm{Pic}^0(X/B)$ action lifts to the action of a $\C^\times$ acting fiberwise. 
	We restrict this to $\U_1 \subset \C^\times$.   (Said restriction is canonical, as $\U_1$ is characterized as the maximal compact in $\C^\times$.) 
	The final assertion regarding $X / \U_1 \to \D_1$ is a local calculation at the nodal curve.  
\end{proof}

\subsection{Hyperk\"ahler structure}

In this section we recall how the Gibbons-Hawking ansatz can be used to construct hyperk\"ahler metrics on such elliptic fibrations \cite{AKL, Go, OV,  GrWi, Gr, GMN}. We learned these results from two letters of Michael Thaddeus \cite{Th} to Hausel and Proudfoot.

Fix the following data :
\begin{enumerate}
	\item A discrete subset $S \subset \R^3$.
	\item A positive harmonic function $V : \R^3 \setminus S \to \R$.
	\item A smooth $\U_1$-bundle $\pi : X_0 \to \R^3 \setminus S$.  
	\item A connection one-form $\theta \in \Omega^1(X_0)$ with curvature $d \theta = 2 \pi i \pi^*( \star dV)$, where $\star$ denotes the Hodge star operator with respect to the standard metric on $\R^3$. 
\end{enumerate}

The connection one-form is uniquely determined by $V$ up to adding a closed $\U_1$-invariant form,
which, as $\R^3 \setminus S$ is simply connected, must take the form $\theta' = \theta + \pi^*df$.

To this data, Gibbons and Hawking \cite{GH} associate a metric $g$ on $X_0$ defined by
\begin{equation} \label{gibbons hawking metric} 
g := V^{-1} \theta \cdot \theta + V \pi^* ds^2 
\end{equation}
where $ds^2$ is the Euclidean metric on $\R^3$.  

\begin{theorem} \cite{GH}  
The metric $g$ of \eqref{gibbons hawking metric} is hyperk\"ahler.  
\end{theorem} 

More precisely, in \cite{GH} the metric was shown to be Ricci flat, which in four dimensions is equivalent to being hyperk\"ahler. 

The two-sphere of compatible complex structures and K\"ahler forms can be described explicitly \cite{AKL, GrWi}. Choose an orthonormal frame $e_1, e_2, e_3$ of $\R^3$, and write $(u_1, u_2, u_3)$ for the associated coordinate system. Lift these to horizontal vector fields $\hat{e}_1, \hat{e}_2, \hat{e}_3$ on $X_0$. Let $\hat{e}_0$ generate the $\U_1$-action on $X_0$. Then 
\[ V^{1/2} \hat{e}_0, V^{-1/2} \hat{e}_1, V^{-1/2} \hat{e}_2, V^{-1/2} \hat{e}_3 \] 
are an orthonormal frame for $g$. A compatible compatible complex structure sends $V^{1/2} \hat{e}_0$ to a unit vector $V^{-1/2} (a_1 \hat{e}_1 + a_2 \hat{e}_2 + a_3 \hat{e}_3)$ in the orthogonal complement. We can index the complex structure by the unit vector $s = a_1 e_1 + a_2 e_2 + a_3 e_3 \in \R^3$. For example, the complex structure $I_{e_1}$ is given by
\begin{equation} \label{eq:complexstructures} \begin{bmatrix} 0 & -1 & 0 & 0 \\ 1 & 0 & 0 & 0 \\  0 & 0 & 0 & -1 \\ 0 & 0 & 1 & 0 \end{bmatrix}. \end{equation}

The k\"ahler forms associated to $I_{e_1}, I_{e_2}, I_{e_3}$ are
\begin{align*}
\omega_1 = du_1 \wedge \theta / 2 \pi i +  V du_2 \wedge du_3, \\
\omega_2 = du_2 \wedge \theta / 2 \pi i +  V du_3 \wedge du_1, \\
\omega_3 = du_3 \wedge \theta / 2 \pi i +  V du_1 \wedge du_2. 
\end{align*}

\begin{lemma}
The action of $\U_1$ on $X_0$ is hyperhamiltonian,with hyperk\"ahler moment map given by the projection $X_0 \to \R^3$. 
\end{lemma} 
The choice of $s$ determines a decomposition $\R^3 = (\R s)^{\perp} \oplus \R s = \C \times \R$, with coordinates $(z_s, u_s)$. The component $z_s : (X_0, I_s) \to \C$ of the hyperk\"ahler moment map is holomorphic.

Now suppose that we have a smooth four-manifold $X$ with a $\U_1$-action, and an open embedding $X/\U_1 \subset \R^3$. Let $S \subset \R^3$ denote the image of the $\U_1$-fixed locus, and fix data as above.
\begin{hypothesis}  \label{gh singularities} 
Near each $s \in S$,  $V \pm \frac{1}{4\pi |u - s|}$ extends smoothly over $s$, where $\pm 1$ is the local Chern class
of the $\U_1$-bundle defined by taking the preimage in $X$ of a small sphere around $s$.  
\end{hypothesis}
\begin{remark} 
By the mean value property for harmonic functions, if $V \pm \frac{1}{4\pi |u - s|}$ is  bounded, 
it is harmonic, and in particular smooth. 
\end{remark}

\begin{proposition} \cite[Section 2]{AKL} 
Under hypothesis \ref{gh singularities},  the metric $g$ extends smoothly to $X$. 
The extended metric $g$ is hyperk\"ahler, and 
the hyperk\"ahler moment map and the sphere of complex structures can be extended smoothly to $X$. 
\end{proposition}
\begin{proof} 
The existence of smooth extension is checked explicitly in coordinates.
Such a smooth extension is automatically hyperk\"ahler, since the Ricci tensor is a continuous function of the metric. 
The complex structures are then defined by parallel transport from any point away from the singularity. 
Likewise, the circle action will preserve the extended metric, since the Lie derivative is continuous. 
\end{proof}

When $S \subset \R^3$ is a finite set of points, we can produce a hyperk\"ahler metric on $X$ starting from the everywhere positive harmonic function 
\begin{equation} \label{eq:sumofpotentials} V_0 = \sum_{s \in S} \frac{1}{4 \pi |u - s|}. \end{equation}
If the number of points is infinite, but the space between these points grows sufficiently fast, the corresponding sum \eqref{eq:sumofpotentials} will converge to a harmonic function with the right singularities, and one again obtains a hyperk\"ahler metric. 

Now consider the case where our four-manifold $\tilde{X}$ carries a free action of $\Z$ covering a translation on $\R^3$, and the critical locus $S$ is a single $\Z$-orbit. In this case \eqref{eq:sumofpotentials} will diverge everywhere. By adding a suitable constant to each term, however, Ooguri and Vafa 
\cite{OV} obtained a series which converges everywhere on $\R^3$ and is positive on a neighborhood of $S$ of the form $\D_r \times \R$, where $\D_r$ a disk of some radius $r > 0$. One thereby obtains a $\Z$-invariant hyperk\"ahler metric on $X$, which descends to a hyperk\"ahler metric on $X = \tilde{X} / \Z$. 

We now return to the setting of Lemma \ref{lem:uniqueaction}. $X$ already carries a complex structure, and we want to define a hyperk\"ahler metric compatible with this structure. Gross and Wilson \cite{GrWi} show how to accomplish this by adding to $V_0$ a suitable harmonic correction term, defined in terms of the periods of the elliptic fibration. The resulting function $V$ will again be positive on a neighborhood of $S$ of the form $\D_r \times \R$.  In fact, they define a family of such metrics depending on $\epsilon > 0$, which give the elliptic fibers volume $\epsilon$. The radius $r$ of positivity of $V$ will depend on both $\epsilon$ and the periods of the fibration.

Let us now give the precise statement. Given $0 < r < 1$, let $\D_r \subset \D$ be the disk of radius $r$. Let $X_r$ be the preimage of $\D_r$.

\begin{proposition} \label{prop: gross-wilson} \label{pinchdeg}
\cite[Sec. 3]{GrWi}, \cite{OV}: 
For each $\epsilon > 0$ there exists a radius $0 < r < 1$, and a (not complete) hyperk\"ahler structure on $X_r$,  extending the 
given complex structure $I$ on $X_r$ (in which $q: X_r \to \D_r$ is holomorphic).   
This metric gives the elliptic fibers volume $\epsilon$, and $\U_1$ acts by isometries and admits
a hyper-hamiltonian moment map of the form 
\begin{equation}
	\label{eq: pi D} 
	\pi_X = q \times \mu_X^{\U_1}  \colon X_r \to \uutimesD
\end{equation}

Moreover, let $\mathbf{n} \in X$ be the node of $q^{-1}(0)$.  Then: 
\begin{enumerate}
	\item The only critical point of $\pi_X$ is $\mathfrak{n}$, and $\pi_X(\mathfrak{n}) = 0 \times 1$. 
	\item  \label{principal bundle} Away from $\mathbf{n}$, the map $\pi_X$ is a principal $\U_1$-bundle, 
	with Chern class $-1 \in \Z =  \H^2( \uutimesD \setminus \hyperbasepoint\homsep \Z)$.  
	\end{enumerate}
\end{proposition}
\begin{proof}
We recall the general outline of the argument, for details see \cite[Sec. 3]{GrWi}. 

Consider the universal cover $\tilde{X} \to X$; it has deck group $\Z$.  The task is then to construct
a $\Z$ and $\U_1$ invariant hyperk\"ahler metric on $\tilde{X}_r$, for which the $\U_1$ action is moreover
hyperhamiltonian with $\Z$-equivariant moment map $q \times \tilde \mu_X^\R: X \to \C \times \R$. 
We do this using the Gibbons-Hawking ansatz described above.

In the situation at hand, we fix an identification $\tilde{X} / \U_1 \cong \D \times \R$, and write $u_2 + i u_3$ (resp. $u_1$) for the coordinates on $\D$ (resp $\R$). We take $S$ to be the $\Z$-invariant subset  $\{ (0, n) | n \in \Z \} \subset \C \times \R$. Let $A \in \H_1(X_b)$ be the vanishing cycle in a nearby fiber, and extend it to a basis $A,B$ such that monodromy around zero acts by $A \to A, B \to A + B$. We may normalize the periods of the elliptic fibers to be $1$ along $A$. Then the period along $B$ will be 
\begin{equation} \label{eq:periodsofellipticfibers} \frac{1}{2\pi i}\log(z) + ih(z) \end{equation}
for some holomorphic function $h(z)$.

Fix $\epsilon > 0$, and consider the following function
\begin{equation} \label{eq:Vzero} V_0 = \frac{1}{4\pi} \sum_{- \infty}^{\infty} \left( \frac{1}{\sqrt{(u + \epsilon n)^2 + z \bar{z}}} - a_n\right), \end{equation}
where
\[ a_n = \frac{1}{\epsilon n}, n \neq 0 \ \ \ a_0 = 2 ( -\gamma + \log(2\epsilon) )/\epsilon. \]
Here, $\gamma$ is Euler's constant. The choice of $a_n$ is made so that the sum \eqref{eq:Vzero} converges uniformly over compact subsets of $|z| < 1$. Since the individual terms are harmonic, Harnack's convergence theorem implies that the limit is also harmonic. It is also manifestly $\Z$-invariant. Each term of the sum has a $\frac{1}{4\pi |u - s|}$-singularity at exactly one $s \in S$, and this property is inherited by the sum.

Gross and Wilson show that there exists a function $0 < r(\epsilon) < 1$ such that on $\R \times \D_{r(\epsilon)}$, 
one has 

\[ V := V_0 + \mathrm{Re}(h(z)) > 0. \]  

We fix $r = r(\epsilon)$ and apply the Gibbons-Hawking construction to the fibration $\tilde{X}_r \to \R \times \D_r$ and the function $V$ and a suitable choice of connection one-form $\theta$. The result is a $\Z$-invariant hyperk\"ahler metric $g$ on $\tilde{X}_r$, and a compatible sphere $(aI + bJ + cK)$ of complex structures. Let $I$ be the complex structure for which the map $X \to \D$ is holomorphic, and therefore defines an elliptic fibration. One can directly compute the periods of an elliptic fiber as an integral of the function $V$, and verify that these match \eqref{eq:periodsofellipticfibers}. It follows that the elliptic fibration with complex structure $I$ is fiberwise isomorphic to the original elliptic fibration.

By $\Z$-invariance, the metric descends to $X_r$. A direct calculation shows that the volume of each fiber 
equals $\epsilon$.
\end{proof}

We will often abbreviate $\mu^{\U_1}_X$ to $\mu_X$. 
The $\U_1$ action and map $\mu_X$ give $X$ the structure of a $(\U_1, \U_1 \times \C)$-manifold.  
We will however often be interested
in the restricted structure of a $(\uu, \uu)$-manifold given by just using $\mu_X$. 
Note from Proposition \ref{prop: gross-wilson} \eqref{principal bundle}, it follows that 
for $\epsilon \ne 0$, there is an isomorphism of  $(\uu, \uu)$ manifolds $q^{-1}(\epsilon) \cong [\uu \times \uu]$.

\begin{lemma}
Proposition \ref{prop: gross-wilson} applies to the fibration $\Dol \to \D_1$, defining a hyperk\"ahler metric on $\Dol_r$. We can take $r$ arbitrarily close to $1$ by making $\epsilon$ sufficiently small.
\end{lemma} 
\begin{proof}
For the elliptic fibration $\Dol \to \D$ with coordinate $z=q$ on $\D$, the correction term $\Re(h(z))$ vanishes. The calculation of $r(\epsilon)$ in \cite[Sec. 3]{GrWi} shows that $r$ can be taken near $1$ for $\epsilon$ sufficiently small.
\end{proof}

\begin{remark}
For Gross and Wilson, the important case is $\epsilon \ll 1$, whereupon these metrics approximate a global metric on a K3 surface. On the other hand, we are content to fix some $\epsilon > 0$. For the purposes of this paper, the specific choice of $r > 0$ is irrelevant, and we will suppress it by writing $\D = \D_r$ in what follows. 
\end{remark}

\begin{remark} 
We may rescale the metric on $\Dol$ so that the volume of any fiber of $\dmomD$ equals one.  
Because the fibers are one complex dimensional, this corresponding K\"ahler form $\omega_I$ is integral.  
This will be helpful for later arguments regarding projectivity.  
\end{remark}

\begin{remark}
The holomorphic functions $x_0, q$ define an isomorphism $\tilde{\Dol} \setminus q^{-1}(0) \cong \C^* \times \D_1$. In fact there is a family of such isomorphisms, obtained by
\begin{equation} \label{eq:reparametrize} (x_0,q) \to (f(q)x_0, q) \end{equation}
for any map $f : \D_1 \to \C^*$. The function $\mu^{\R}_{\Dol} : \tilde{\Dol} \to \R$ is $\U_1$-invariant and therefore depends only on $|x_0|, q$. 

We do not know an explicit formula for $\mu^{\R}_{\Dol}$.  However, we can give the following characterisation. Set $t = \log |x_0|$. Then $\mu^{\R}_{\Dol}$ solves the non-linear ODE
\begin{equation} \label{eq:ODEformoment} \frac{du}{dt} = \frac{1}{V(u, q)}. \end{equation}
To see this, recall the vector field $\hat{e}_0$ generating the $\U_1$ action. Then $I \hat{e}_0$ generates the action of $\R^* \subset \C^*$. We have 
\[ \mathcal{L}_{I \hat{e}_0} (\mu_{\Dol}) = \iota_{I \hat{e}_0}(d \mu_{\Dol}) = g(I \hat{e}_0, I \hat{e}_0) = g(\hat{e}_0, \hat{e}_0) = V^{-1}. \] 
Equation \eqref{eq:ODEformoment} has a unique solution up to shifts $t \to t + c(q)$ of the argument. These shifts correspond exactly to the modifications \eqref{eq:reparametrize}.
\end{remark}

\begin{lemma} \label{lem: holsympissame}
The form $\Omega = dx_n \wedge dy_n = d \log(x_n) \wedge dq$ from Lemma \ref{lem: W structures} 
agrees with the holomorphic symplectic form $\Omega^{GW} := \omega_J + \sqrt{-1} \omega_K$
associated to the
hyperk\"ahler metric of Proposition \ref{prop: gross-wilson}
\end{lemma}
\begin{proof}  
Fix a contractible open $U \subset \D \setminus 0$, and let $\widetilde{X}_U \to U$ be the restriction of $\widetilde{X} \to \D$. Gross and Wilson \cite[Construction 2.6]{GrWi} show that 
on any such open, $\Omega^{GW} = d \log(u) \wedge dq$, where $u : \widetilde{X}_U \to \C$ is a certain holomorphic function scaled with weight one by $\C^{\times}$ and defined up to multiplication by an invertible function of $q$.\footnote{To compare with \cite{GrWi}: our $q$ is their $y$, and our $u$ is their $\exp(2 \pi i x)$.} 
Since $d \log(u) - d \log(x_n) = d \log(u/x_n) = d f(q)$ for some function $f : U \to \C$, we have
$ d \log(u) \wedge dq = d \log(x_n) \wedge dq$. \end{proof}

\begin{lemma} \label{cor:dolretract} 
There is a $\U_1$-equivariant deformation retraction of $\Dol$ onto $\dmomD^{-1}(\D_{\epsilon})$, 
where $\D_{\epsilon}$ is the open disk of radius $0 < \epsilon < 1$.   There is also 
a $\U_1$-equivariant deformation retraction of $\Dol$ to $q^{-1}(0)$. 
\end{lemma}
\begin{proof}
Pick any $\uu$-connection on the bundle described in Proposition \ref{pinchdeg}  (we do not require it to be flat). 
Then the linear retraction $\uutimesD \to \D^{\epsilon} \times \uu$ or $\uutimesD \to 0 \times \uu$ induces the desired retraction of $\Dol$ via parallel transport. One must treat the case of the line through $0 \times 1$ separately, as the fiber over $0 \times 1$ collapses to a point. The preimage of this line is homeomorphic to a disk, and parallel transport extends to a retraction to the origin.
\end{proof}

\begin{figure}[h]
\large
\centering{
	\resizebox{70mm}{!}{\includegraphics{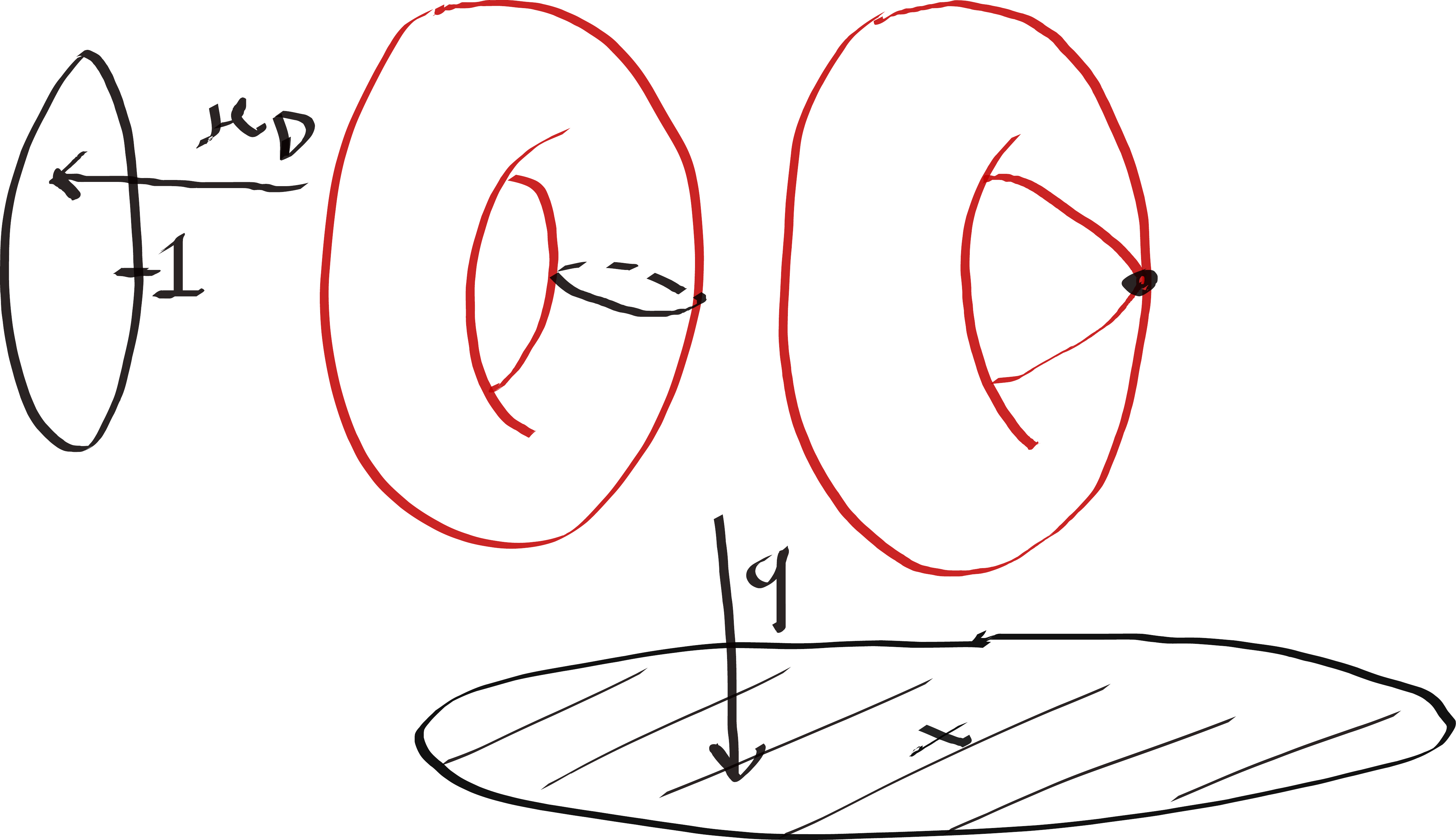}}
	\caption{A schematic picture of $\Dol$ and its various moment maps and their targets. Two fibers of $q$ are shown, and the intersection of each fiber with $\mu_\Dol^{-1}(1)$ is indicated in black.}
}
\end{figure}

\subsection{Vanishing cycles for $q$} \label{sec: qvan}
Given any space $X$ and map $q \colon X \to \C$, denote the inclusion of the zero fibre by  $i \colon q^{-1}(0) \to X$, and 
consider the inclusions

$$q^{-1} ( \{\operatorname{Re}(z)\leq 0\} ) \xrightarrow{\mathcal{I}} X \xleftarrow{\mathcal{J}} q^{-1}( \{\operatorname{Re}(z)> 0\} )$$

The corresponding excision triangle $\cI_!\cI^!\Q \to \Q \to \cJ_*\cJ^*\Q \xrightarrow{[1]}$ restricts to the 
nearby-vanishing triangle (in some accounts this is the definition of nearby and vanishing cycles).  That is,   
\[  \Psi_q \Q = i^*\cJ_*\cJ^*\Q. \]
\[  \Phi_q \Q = i^*\cI_!\cI^!\Q. \]

We now return to the case at hand.  We define 

$$\bS_\Dol \colonequals \dmom^{-1}(\R^{\leq 0} \times 1)$$
and write for the inclusions $\bS_\Dol \xrightarrow{I} \Dol \xleftarrow{J} \Dol \setminus \bS_\Dol$. 

\begin{lemma} 
$\bS_\Dol$ is a codimension 2 submanifold of $\Dol$, diffeomorphic to an open disk. 
\end{lemma}
\begin{proof}
Since $\dmom$ is a circle fibration away from $0 \times 1$, the preimage $\dmom^{-1}(\R^{< 0} \times 1)$ is evidently
a cylinder.  It suffices to investigate the geometry near $0 \times 1$.  This point is the image of the fixed point $\mathbf{n}$, where
in local coordinates the circle action is $\tau \cdot (x, y) = (\tau x, \tau^{-1} y)$.  It follows that the fibration is equivariantly 
diffeomorphic to the standard Hopf fibration $\R^4 \to \R^3$, where it can be checked in coordinates that the preimage of 
any smooth ray leaving the origin is a disk.  
\end{proof}

\begin{remark} \label{rem:SDolandSBet}
In Lemma \ref{lem:tateNAHT} below, we 
use a more elaborate version of this argument to construct an embedding $\Dol \subset \Bet$, with
respect to which $\bS_\Dol = \bS_\Bet \cap \Dol$. Under this embedding, we shall see that $I,J$ are intertwined with the same-named maps from Equation \eqref{bettiinclproj}. 
\end{remark}

\begin{lemma} \label{lem:gutteddolb}
Let $\epsilon > 0$ be sufficiently small. We have a retraction of $(\U_1, \U_1)$-spaces $\Dol \setminus \bS_\Dol \to q^{-1}(\epsilon)$.
\end{lemma}
\begin{proof}
$\Dol \setminus \bS_{\Dol}$ is a trivial $\U_1$-bundle over $\uutimesD \setminus [0,1) \times 1$. A retraction of $\D \setminus [0,1)$ to $\epsilon$ can be lifted to an $(\U_1, \U_1)$-retraction of the total space of the bundle.
\end{proof}

Observe that we have a closed inclusion

$$\bS_\Dol = \dmom^{-1}(\R^{\leq 0} \times 1) \subset  \dmom^{-1}(\{ \operatorname{Re}(z)\leq 0 \} \times \U_1) = q^{-1} ( \{\operatorname{Re}(z)\leq 0\} )$$

Thus the excision triangle for $\bS_\Dol$ maps to the excision triangle for $q^{-1} ( \{\operatorname{Re}(z)\leq 0\} )$.  

\begin{equation}
\begin{tikzcd} \label{basicdiagram}
	I_! I^! \Q   \arrow[r] \arrow[d] & \Q \arrow[r] \arrow[d] &   J_* J^*\Q  \arrow[r, "{[1]}"] \arrow[d]  & \;  \\ 
	\cI_!\cI^! \Q   \arrow[r] &  \Q \arrow[r] &   \cJ_*\cJ^*\Q  \arrow[r, "{[1]}"]  & \;  \\ 
\end{tikzcd}
\end{equation} 
\begin{proposition} \label{basicvanishingpairisom} 
The restriction of the above diagram to the nodal rational curve at $q^{-1}(0)$ is
an isomorphism of triangles.   
\end{proposition}
\begin{proof}
The only point of intersection between $\bS_\Dol$ and the central fiber is the node: 
$\bS_\Dol \cap q^{-1}(0) = \mathbf{n}$. Away from the node, the map $q$ is smooth. Thus along $q^{-1}(0) \setminus \mathbf{n}$, the diagram  restricts to 
\begin{center} 
	\begin{tikzcd} \label{basicdiagramrest}
		0   \arrow[r] \arrow[d] & \Q \arrow[r] \arrow[d] &  \Q  \arrow[r, "{[1]}"] \arrow[d]  & \;  \\ 
		0  \arrow[r] &  \Q \arrow[r] &   \Q  \arrow[r, "{[1]}"]  & \;  \\ 
	\end{tikzcd}
\end{center}
with all maps given by the identity or the zero map. On the other hand, the Milnor fiber of $q$ at $\mathbf{n}$ has cohomology supported in degrees $0$ and $1$, and the degree one homology is generated by any orbit of $\U_1$. Similarly, the degree one homology of $V \setminus V \cap \bS_\Dol$ (for $V$ a small ball containing $\mathbf{n}$) is generated by any orbit of $\U_1$, since $\bS_\Dol$ is a smooth $\U_1$-stable codimension-two submanifold. Thus the restriction map from $V \setminus V \cap \bS_\Dol$ to the Milnor fiber induces an isomorphism on cohomology. It follows that the right-hand vertical map is an isomorphism in a neighborhood of $\mathbf{n}$. Since it is also an isomorphism away from $\mathbf{n}$, it is a global isomorphism. Since the central vertical map of the diagram is simply the identity, the left-hand map must also be an isomorphism.
\end{proof}

\begin{corollary} \label{cor:vanishing}
$\Psi_q \Q = \Q_{\mathbf{n}}[-2]$.
\end{corollary}
\begin{proof}
We have
$$\Psi_q \Q = i^*\cI_!\cI^! \Q = i^* I_! I^! \Q = i^* \Q_{\bS_\Dol} [-2] = \Q_{\mathbf{n}}[-2].$$
Here, the first equality is essentially the definition of  of vanishing cycles.  The second equality holds because
of the isomorphism of exact triangles.  The third equality holds because $\bS_\Dol$ is a real codimension 2 submanifold.
The final equality holds because $\bS_\Dol \cap q^{-1}(0) = \mathbf{n}$. 
\end{proof}

\begin{remarque}
The locus $\bS_\Dol$ is a Lefschetz thimble for the vanishing cycle. 
\end{remarque}

\begin{remark}
In the nearby-vanishing exact triangle of sheaves on $q^{-1}(0)$,
$$\Psi_q \Q \to \Q|_{q^{-1}(0)} \to \Phi_q \Q \xrightarrow{[1]}$$
we substitute $\Psi_q \Q = \Q_{\mathbf{n}}[-2]$, and pass to global cohomology: 
$$ \Q[-2] \to \Q \oplus \Q[-1] \oplus \Q[-2] \to \Q \oplus \Q^{\oplus 2}[-1] \oplus \Q[-2] \xrightarrow{[1]} $$
Note the similarity to the sequence appearing in the calculation of $\H^\bullet(\Bet\homsep \Q)$ in Proposition \ref{prop:cohomzbet}, save 
that the weight grading no longer appears.  This similarity will ultimately develop into the
comparison result of Theorem \ref{Hodgeintertwinesdelcon}. 
\end{remark}

\section{$\Dol(\Gamma)$}   \label{sec:dol}  

\subsection{Construction}
\label{subsec:graphdolbspace} 

Per Proposition \ref{prop: gross-wilson}, the space $\Dol$ carries 
a $(\uu, \uutimesC)$-structure. 

\begin{definition}
Given a graph $\Gamma$ and  $\eta \in C_0(\Gamma, \uutimesC)$, we set
$$\Dol(\Gamma) \colonequals \Dol^{(\uu, \uutimesC)} (\Gamma, \eta). $$
\end{definition}

Writing $\eta = p \times \nu$, the space $\Dol(\Gamma, \eta)$ is  
$\mu_{\Gamma}^{-1}(p \times \nu) / \redC^0(\Gamma, \uu)$.  The space $\mu_{\Gamma}^{-1}(p \times \nu) \subset \Dol^{\ed}$ is
the subset satisfying 
\begin{align}
\sum_{\text{edges exiting } v} q_e - \sum_{\text{ edges entering } v} q_e &  = p_e \\
\prod_{\text{edges exiting } v} \mu^{\uu}_e \prod_{\text{ edges entering } v} \big( \mu^{\uu}_e \big)^{-1} &  = \nu_v. 
\end{align}

\begin{proposition} \label{prop:dolbHK}  For generic $\eta$, the space 
$\Dol(\Gamma, \eta)$ is a (non-complete) hyperk\"ahler manifold. 
\end{proposition}
\begin{proof}
It follows from Proposition \ref{pinchdeg}  that 
$\Dol$ satisfies Hypothesis \ref{hyp:bZsmoothness} as a $(\U_1, \uutimesC)$-space. 
Thus since we chose generic $\eta$, Proposition \ref{prop:smoothcond} implies $\Dol(\Gamma, \eta)$ is smooth. 

As $\mu_{\Dol}^{\U_1} \times \dmomD$ is a multiplicative moment map for a hyperk\"ahler action of $\U_1$ 
by Proposition \ref{prop: gross-wilson}, $\Dol(\Gamma, \eta)$ is the hyperk\"ahler reduction of a hyperk\"ahler manifold. 
\end{proof} 
\begin{proposition} \label{prop:dolbintegdef}
$\Dol(\Gamma, \eta)$ is equipped with a complex analytic action of $\H^1(\Gamma\homsep\C^*)$ and a proper holomorphic $\H^1(\Gamma\homsep\C^*)$-invariant map $q_{\res} \colon \Dol(\Gamma, \eta) \to \C^{\ed}$ whose image is the intersection of the unit polydisk with $\H_1(\Gamma\homsep\C)_p$.  
\end{proposition}
\begin{proof}
Proposition \ref{prop: h1h1 structure} yields a proper map $\Dol(\Gamma, \eta) \to \H_1(\Gamma, \uutimesC)_{\eta}$. Composing with the projection $\H_1(\Gamma, \uutimesC)_{\eta} \to \H_1(\Gamma, \C)_p$ preserves properness. Concretely, the projection is induced by restricting $q^{\ed}$ to the zero fiber of the moment map and descending to the quotient. The result is holomorphic and $\H^1(\Gamma, \C^*)$-invariant since $q$ is holomorphic and $\C^*$-invariant.
\end{proof}

Note it is possible to arrange $\eta = (p, v)$ generic while requiring $p = 0$, as a special case of Lemma \ref{genericvaluesexist2}.   We will henceforth restrict attention to this case, 
whereupon $\H_1(\Gamma\homsep\C)_p = \H_1(\Gamma \homsep \C)$. 

\begin{remark}
Symplectic reduction of the holomorphic symplectic form $\Omega^{\oplus \ed}$ on $\Dol^\ed$ gives a 
holomorphic symplectic form $\Omega_\Gamma$ on $\Dol(\Gamma)$.  
The map $q_{\res}: \Dol(\Gamma) \to \H_1(\Gamma\homsep\C)$ determines an integrable
system, as can be seen e.g. by symplectic reduction from the corresponding fact about $q: \Dol \to \D$.  
We consider this structure the counterpart of Hitchin's integrable system on the moduli of Higgs bundles \cite{Hit2}.  
\end{remark} 

\begin{proposition} \label{prop: action is symplectic} 
The action of $\H^1(\Gamma\homsep \C^*)$ preserves $\Omega_{\Gamma}$.   
\end{proposition} 
\begin{proof}
This follows from the fact that the $\C^*$ action preserves the holomorphic symplectic form on $\Dol$.
\end{proof}

\begin{remarque}
As with $\Bet(\Gamma)$, we can show that  the dependence on the chosen orientation is quite mild. More precisely, if $\Gamma, \Gamma'$ differ only by the choice of orientation, then there is a canonical isomorphism of smooth manifolds
\[ \Dol(\Gamma) \to \Dol(\Gamma)' \]
By Proposition \ref{prop:indeporient}, it is enough to find an isomorphism $\Dol \to \Dol$ of smooth manifolds, intertwining the $\U_1$-action and the $\uutimesC$-moment map with their inverses. One can construct such a map by arguments similar to the proof of Lemma \ref{lem:tateNAHT}. We will not need this result elsewhere in the paper. 
\end{remarque}

\subsection{Vanishing cycles and convolution} \label{subsec:ageneralsequence}
\label{subsec:retractionstocore}
We wish to show that the calculations of Section \ref{sec: qvan} ``commute with convolution'' with an auxilliary  
$(\uu, \uu)$-manifold $X$, under certain hypotheses on $X$ to be described later.  

Recall that by definition, 
$$X \star_{\uu, \uu, \zeta} \Dol := (\mu_X  + \mu_{\Dol})^{-1}(\zeta) / \uu$$
As always we implicitly require (and must check for any particular $X$ of interest) 
that $\zeta$ is a regular value of the moment map, and that the $\uu$ action
on the fiber over $\zeta$ is free. 
For the remainder of this subsection we abbreviate $\star := \star_{\U_1, \U_1, \zeta}$ product, for some fixed $\zeta \in \U_1$. 

Recall the map $\dmomD \colon \Dol \to \C$. Composing with projection onto the second factor, 
we get a map $\tilde{q}_{2} \colon X \times \Dol \to \C$, which descends to 
\[ q_2  \colon X \star \Dol \to \C. \]
\begin{remark}
In applications, often $X= \Dol(\Gamma)$, hence has its own $q_{\res}$ map.  The subscript in the notation
$q_2$ reminds that the map is built only from the $q$ of the second factor $\Dol$, and has nothing to do with
this $q_{\res}$. 
\end{remark}

It is immediate from the definitions that: 
\begin{eqnarray}
X \star q^{-1}(t) & = &  q_2^{-1}(t) \\
\label{starsing} 
X \star \mathbf{n} & = &  q_2^{-1}(0)  \cap (X \star \bS_\Dol) 
\end{eqnarray}
In addition, note that for $\epsilon \ne 0$, we have  $\dmomD^{-1}(\epsilon) \cong [\U_1 \times \U_1]$ as a $(\U_1, \U_1)$-space, thus 
\begin{equation}
\label{q2nonzeropreimage}  q_2^{-1}(\epsilon) =  X \star q^{-1}(t)  \cong X \qquad \qquad \epsilon \ne 0
\end{equation}

\begin{lemma}  \label{lem: retractions} 
The following inclusions are deformation retracts: 
\begin{eqnarray}
	i:  q_2^{-1}(0) & \hookrightarrow & X \star \Dol \\
	i_{\mathbf{n}} :  X \star \mathbf{n} & \hookrightarrow & X \star \bS_\Dol \\
	i_\epsilon  \colon q_2^{-1}(\epsilon) & \hookrightarrow &  X \star (\Dol \setminus \bS_\Dol) \qquad \qquad \epsilon \in \D \setminus 0
\end{eqnarray} 
\end{lemma} 
\begin{proof}
In each case, the claim follows from the existence of a retraction on the basic space: 
\begin{enumerate}
	\item The retraction $\Dol \to \dmomD^{-1}(0)$ constructed in Corollary \ref{cor:dolretract} is a retraction of $(\U_1, \U_1)$-spaces, and thus induces a retraction $X \star \Dol \to X \star \dmomD^{-1}(0) = q_2^{-1}(0)$. 
	\item Recall that $\bS_\Dol$ is the preimage of $0 \times \R^{\leq 0}$ under $\dmom$. The retraction of $\R^{\leq 0}$ to $0$ induces a retraction of $(\U_1, \U_1)$-spaces $\bS_\Dol \to \mathbf{n}$, which in turn induces a retraction $X \star \bS_\Dol \to X \star \mathbf{n}$.
	\item We have a retraction of $(\U_1,\U_1)$-spaces $\Dol \setminus \bS_{\Dol} \to q^{-1}(\epsilon)$ (Lemma \ref{lem:gutteddolb}), which then induces a retraction $X \star \Dol \setminus \bS_\Dol \to X \star q^{-1}(\epsilon)$.
\end{enumerate}
\end{proof}

\begin{proposition} \label{dol conv main diagram} 
Consider the maps
\[X \star \bS_\Dol   \xrightarrow{I} X \star \Dol \xleftarrow{J} X \star (\Dol \setminus \bS_\Dol)  \] 
There is the following commutative diagram, with each row an exact triangle: 
\begin{equation}
	\label{global dol key diagram}
	\begin{tikzcd} 
		& & \H^\bullet(X, \Q) & \; \\
		\H^\bullet(X \star \mathbf{S}_\Dol,  I^! \Q )   \arrow[r] \arrow{d}[swap]{i_{\mathbf{n}}^*} & \H^\bullet(X \star \Dol, \Q) \arrow[r] \arrow{d}{\sim}[swap]{i^*} 
		& \H^\bullet(X \star (\Dol \setminus \mathbf{S}_\Dol),  \Q)  \arrow[r, "{[1]}"] \arrow[d] \arrow{u}{i_\epsilon^*}[swap]{\sim} & \;  \\ 
		\H^\bullet(X \star \mathbf{n},  i_{\mathbf{n}}^* I^! \Q)   \arrow[r] \arrow[d] & \H^\bullet(q_2^{-1}(0), \Q)  \arrow[r] \arrow[d, equals] & 
		\H^\bullet(q_2^{-1}(0),  i^* J_* J^*\Q)  \arrow[r, "{[1]}"] \arrow[d]  & \;  \\ 
		\H^\bullet(q_2^{-1}(0), \Phi_{q_2} \Q)   \arrow[r] & \H^\bullet(q_2^{-1}(0), \Q) \arrow[r] &  \H^\bullet(q_2^{-1}(0),  \Psi_{q_2} \Q)  \arrow[r, "{[1]}"]  & \;  \\ 
	\end{tikzcd}
\end{equation}
Here, $\epsilon$ is any nonzero element of $\D$. 
\end{proposition}
\begin{proof} 
We will use the additional triangle 
$$q_2^{-1} ( \{\operatorname{Re}(z)\leq 0\} ) \xrightarrow{\mathcal{I}} X \star \Dol \xleftarrow{\mathcal{J}} q_2^{-1}( \{\operatorname{Re}(z)> 0\} ).$$
We have the comparison between excision triangles: 
\begin{equation}
	\label{keydiagram}
	\begin{tikzcd} 
		I_! I^! \Q   \arrow[r] \arrow[d] & \Q \arrow[r] \arrow[d, equals] &   J_* J^*\Q  \arrow[r, "{[1]}"] \arrow[d]  & \;  \\ 
		\cI_!\cI^! \Q   \arrow[r] &  \Q \arrow[r] &   \cJ_*\cJ^*\Q  \arrow[r, "{[1]}"]  & \;  \\ 
	\end{tikzcd}
\end{equation}
and its restriction along $i  \colon  q_2^{-1}(0) \subset X \star \Dol$: 
\begin{equation}
	\label{morekeydiagram}
	\begin{tikzcd} 
		i^* I_! I^! \Q   \arrow[r] \arrow[d] & \Q \arrow[r] \arrow[d, equals] &   i^* J_* J^*\Q  \arrow[r, "{[1]}"] \arrow[d]  & \;  \\ 
		\Phi_{q_2} \Q   \arrow[r] &  \Q \arrow[r] &  \Psi_{q_2} \Q  \arrow[r, "{[1]}"]  & \;  \\ 
	\end{tikzcd}
\end{equation}
which we may push forward again by $i_*$ to find
\begin{equation}
	\label{evenmorekeydiagram}
	\begin{tikzcd} 
		I_! I^! \Q   \arrow[r] \arrow[d] & \Q \arrow[r] \arrow[d] &   J_* J^*\Q  \arrow[r, "{[1]}"] \arrow[d]  & \;  \\ 
		i_*i^* I_! I^! \Q   \arrow[r] \arrow[d] & i_*\Q \arrow[r] \arrow[d, equals] &   i_* i^* J_* J^*\Q  \arrow[r, "{[1]}"] \arrow[d]  & \;  \\ 
		i_* \Phi_{q_2} \Q   \arrow[r] & i_* \Q \arrow[r] &  i_* \Psi_{q_2} \Q  \arrow[r, "{[1]}"]  & \;  \\ 
	\end{tikzcd}
\end{equation}
The main 3x3 square of Diagram \eqref{global dol key diagram} is found by taking global sections of Diagram \eqref{evenmorekeydiagram}, 
and using Equation \eqref{starsing}.   The indicated maps are isomorphisms by Lemma \ref{lem: retractions}. 
\end{proof}

\begin{lemma} \label{singlocus} 
The singular locus of $q_2$ is contained in $X \star \mathbf{n}$. 
\end{lemma}
\begin{proof} 
Let $z \in \Dol \setminus \mathbf{n}$. Choose a small open neighborhood of polydisk $V_1 \times V_2$ containing $\dmom(z) \in \uutimesD$, not containing the point $(0,1)$, on which the $\U_1$ bundle $\mathscr{P}_{\Dol}$ can be trivialized. By Corollary \ref{cor:bullettriv}, we have a diffeomorphism
\[ \dmom^{-1}(V_1 \times V_2) \star X \cong V_1 \times \mu_X^{-1}(V_2) \]
and under this isomorphism $q_2$ becomes projection onto $V_1$. This concludes the proof.
\end{proof}

\begin{hypothesis} \label{hyp} 
There is an open contractible neighborhood $\baseopen \subset \U_1$ of $\zeta \in \U_1$, such that: 
\begin{enumerate}
	\item \label{Xassumpone}
	The map $\mu_X  \colon X \to \U_1$ is locally constant near $\zeta$, in the sense that for some interval 
	$\zeta \in \baseopen \subset \U_1$, the space $\mu_X^{-1}(\baseopen) \subset X$ is isomorphic as a $(\U_1, \U_1)$-space to $[\baseopen \times \mu_X^{-1}(\zeta)]$. 
	\item \label{Xassumptwo}
	The action of $\U_1$ on $\mu_X^{-1}(\zeta)$ is free. 
\end{enumerate}
\end{hypothesis}

\begin{proposition} \label{vanishpairisom}
When Hypothesis \ref{hyp} holds, $X \star \bS_\Dol$ is a submanifold of $X \star \Dol$ with $\R$-codimension $2$, we have
$\Phi_{q_2} \Q = \Q_{X \star \mathbf{n}}[-2]$, and all vertical arrows in Diagram \eqref{global dol key diagram} are isomorphisms. 
\end{proposition} 
\begin{proof} 
The key points are to show
\begin{enumerate}[label=(\alph*)]
	\item \label{dolbdivisor} $X \star \bS_\Dol$ is a real-codimension two submanifold of $X \star \Dol$
	\item \label{wantedisom} The left vertical map $i^* I_! I^! \Q \to \Phi_{q_2} \Q$ of \eqref{morekeydiagram} is an isomorphism. 
\end{enumerate}
Indeed \ref{dolbdivisor} implies $I^! \Q = \Q[-2]$; thus because  $i_{\mathbf{n}}$ is a retraction, it would follow
that $i_{\mathbf{n}}^*$ is an isomorphism in
Diagram \eqref{global dol key diagram}, hence that all vertical arrows between the top rows are isomorphisms.    
Then \ref{wantedisom} implies that Diagram \ref{morekeydiagram} is an isomorphism of exact triangles.  
From this, together with \ref{dolbdivisor}, we may deduce $\Phi_{q_2} \Q = \Q_{X \star \mathbf{n}}[-2]$ as in the proof of  Corollary \ref{cor:vanishing}. 
Also, by taking global sections, we find 
that the vertical maps between the bottom two rows of \eqref{global dol key diagram} are isomorphisms. 

We turn to establishing \ref{dolbdivisor} and \ref{wantedisom}.  Both are local statements.  First
we study them away from $X \star \mathbf{n}$, i.e. in $X \star (\Dol \setminus \mathbf{n})$.  
Since $\Dol \setminus \mathbf{n}$ is a principal $\uu$-bundle (compatibly with $\bS_\Dol$), it is 
easy to see that \ref{dolbdivisor} holds here.  As for \ref{wantedisom}, we will show
both terms vanish.  For $i^* I_! I^! \Q$, this is simply because  $i^* I_!$ of anything is supported 
on $q_2^{-1}(0) \cap \bS_\Dol$, which by \eqref{starsing} is $X \star \mathbf{n}$.  And it follows immediately
from Lemma \ref{singlocus} that $ \Phi_{q_2} \Q$ is also supported on $X \star \mathbf{n}$.     

We now study a neighborhood of $X \star \mathbf{n}$.  Recall $\baseopen$ is some 
neighborhood of $\zeta \in \uu$.  
Let 
$$B \colonequals \dmomabrevU^{-1}(\zeta - \baseopen)$$ 
(Since $\uu$ is playing the role of moment image, in keeping with our general conventions
we use additive notion for its group law.)

As $B$ is an open neighborhood of $\mathbf{n} \in \Dol$, 
the space $q_2^{-1}(0) \cap (X \star B)$ is a neighborhood of $q_2^{-1}(0) \cap (X \star \mathbf{n})$.   
Observe
$$X \star B = (\mu_{X} + \dmomabrevU)^{-1}(\zeta) / \U_1 = \mu_X^{-1}(\baseopen) \star B.$$

By Hypothesis \ref{hyp} \eqref{Xassumpone}, we have an isomorphism of $(\U_1, \U_1)$-spaces
\begin{equation} \label{connection} \mu_X^{-1}(\baseopen) \cong [\mu_X^{-1}(\zeta) \times \baseopen] \end{equation}
where the moment map on the RHS is the projection $ \mu_X^{-1}(\zeta) \times \baseopen \to \baseopen$.
By Hypothesis \ref{hyp} \eqref{Xassumptwo}, the action of $\U_1$ on $\mu_X^{-1}(\zeta)$ is free, hence defines a principal 
$\U_1$-bundle $\rho  \colon \mu_X^{-1}(\zeta) \to \mu_X^{-1}(\zeta) / \U_1$.  

Working locally, we may assume the bundle  $\rho  \colon \mu_X^{-1}(\zeta) \to \mu_X^{-1}(\zeta) / \U_1$ is trivial.  That is, we have reduced
ourselves to establishing the stated result of the proposition with $X$ replaced by a $(\uu, \uu)$-space of the form
$Y \times [\uu \times \baseopen]$, where the $\uu$ action is translation on the $\uu$ factor, and the moment map is just
projection to $\baseopen \subset \uu$.   As we have (functorially) 
$(Y \times [\uu \times \baseopen]) \star K = Y \times K$ for any $(\uu, \uu)$-subspace $K \subset B$, 
\ref{dolbdivisor} is obvious and \ref{wantedisom} follows immediately
from Proposition \ref{basicvanishingpairisom}. 
\end{proof}

\begin{remark}
Note in particular that Proposition \ref{vanishpairisom} implies
$\H^\bullet(q_2^{-1}(0),  \Psi_{q_2} \Q) \cong \H^\bullet(q_2^{-1}(\epsilon), \Q) $ even though
we do not require $q_2$ proper. 
\end{remark}

\subsection{A larger family}  \label{sec:versalfamily} 
Here we introduce a larger family  $q_{\res} \colon \versD(\Gamma, \eta) \to \C^{\ed}$,
from which $q_{\res}: \Dol(\Gamma) \to \H_1(\Gamma, \C)$ may be recovered via base-change
along the inclusion $\H_1(\Gamma, \C) \subset \C^{\ed}$.  Later,
we prove results $\Dol(\Gamma)$ by first establishing their analogues for $\versD(\Gamma, \eta)$ and
then studying the restriction. 

Recall that $\Dol$ had the structure of a $(\U_1, \U_1 \times \C)$-manifold.  
We write $\versD$ for the space $\Dol$ viewed as a $(\U_1, \U_1)$-manifold, i.e. we forget the projection to $\D$.  

\begin{definition} \label{defversalspace0} For $\nu \in C_0(\Gamma, \uu)$, we write  $\versD(\Gamma, \nu) \colonequals \versD^{\U_1, \U_1}(\Gamma, \nu)$. \end{definition}

More explicitly, we have $\versD(\Gamma, \nu) = \mu_{\Gamma}^{-1}(\nu) / \redC^0(\Gamma\homsep \uu)$, where:
\begin{align}
\versD^{\ed} \supset \mu_{\Gamma}^{-1}(\nu)  := \bigg\{ \prod_{\text{edges exiting } v} \mu^{\uu}_e \prod_{\text{ edges entering } v} \big( \mu^{\uu}_e \big)^{-1} &  = \nu_v \bigg\} 
\end{align}

\begin{proposition} For generic $\nu$, the space 
$\versD(\Gamma, \nu)$ is a (non-complete) K\"ahler manifold, equipped with a complex analytic action of $\H^1(\Gamma\homsep\C^*)$ and a proper holomorphic $\H^1(\Gamma\homsep\C^*)$-invariant map $q_{\res} \colon \versD(\Gamma, \nu) \to \C^{\ed}$ whose image is the unit polydisk.\end{proposition} 
\begin{proof}
The proof is similar to that of Proposition \ref{prop:dolbintegdef}. 
$\Dol$ satisfies Hypothesis \ref{hyp:bZsmoothness} as a $(\U_1, \U_1)$-space,
and $\nu$ was chosen to be generic, so $\versD(\Gamma, \nu)$ is smooth by Proposition \ref{prop:smoothcond}. 
Proposition \ref{prop: gross-wilson}
shows $\mu_{\Dol}^{\U_1}$ is a multiplicative moment map for a K\"ahler action of $\U_1$, 
so $\versD(\Gamma, \nu)$ is K\"ahler. The complex analytic action of $\C^*$ on $\Dol$ descends to an action of 
$\H^1(\Gamma\homsep\C^*)$ on $\versD(\Gamma, \nu)$, preserving the fibers of $q_{\res}$. 
\end{proof}

Note that if $\nu \in C_0(\Gamma, \uu)$ is generic, then so is $(\nu \times 0) \in C_0(\Gamma, \uutimesC)$.
We have the following fiber product diagram: 
\begin{equation} \label{eq:qres} 
\begin{tikzcd}
	\Dol(\Gamma, \nu) \arrow[r] \arrow[d, "q_{\res}"] & \versD(\Gamma, \nu \times 0) \arrow[d, "q_{\res}"] \\
	\H_1(\Gamma\homsep\C) \arrow[r] & \C^{\ed}.
\end{tikzcd}
\end{equation} 
Note that neither of the vertical maps are surjective. 
We suppress the dependence on $\nu$ for much of the remainder of the article.

\begin{remarque}  \label{rem:comparemsv} 
Let $C$ be a nodal curve with rational components and dual graph $\Gamma$. Let $B$ be the base of a locally versal family of deformations of $C$; it has dimension (\# of nodes of $C$) = $|\ed|$. Given an auxiliary choice of stability parameter, there is a family of versal compactified Jacobians $\overline{\mathcal{J}} \to B$.    

The family $q_{\res} \colon \versD(\Gamma, \eta) \to \C^{\ed}$ is very similar to this family, although neither one is the base-change of the other. In particular, the notation $\versD$ is meant to suggest ``versal''.  Meanwhile $q_{\res}: \Dol(\Gamma) \to \H_1(\Gamma, \C)$ is similar to the relative compactified 
Jacobian of a subfamily of deformations of $C$  within a fixed ambient holomorphic symplectic surface $S$. 
\end{remarque}

\subsection{Some retractions} 

Many of the spaces we consider here are equipped with maps $\pi : X \to V$ to a vector space and a retraction to the central fiber $\pi^{-1}(0)$.  Here we record properties and examples of such structure.

We say a closed subset $\Lambda \subset V$ of a vector space is semiconical if it is stable with respect to scalar multiplication by $\{ r \in [0,1) \}$.

Let $r_t  \colon \Dol \to \Dol$, $t \in [0,1]$ be the retraction onto $q^{-1}(0)$ constructed in Corollary \ref{cor:dolretract}, and let $r^{\ed}  \colon \Dol^{\ed} \to q^{-1}(0)^{\ed}$ be the product retraction.

\begin{lemma} \label{lem:dolbretract}
Let $\Lambda  \subset \C^{\ed}$ be semiconical. Let $T \subset \U_1^{\ver}$ be a subtorus and $\zeta \in \operatorname{Lie}(T)^*$. Then $r^{\ed}$ descends to a retraction $(q^{\ed})^{-1}(\Lambda) \sslash_{\zeta} T$ onto $(q^{\ed})^{-1}(0) \sslash_{\zeta} T$. 
\end{lemma}

\begin{proof}
We first claim that $r^{\ed}$ restricts to a retraction of $(q^{\ed})^{-1}(\Lambda)$ onto $(q^{\ed})^{-1}(0)$. Indeed, recall that $\dmomD \circ r_t$ covers a linear retraction $\D \to 0$, and $\mu_{\Dol}^{\uu} \circ r_t = \mu_{\Dol}^{\uu}$. It follows that $r_t^{\ed}$ preserves $\Lambda$, establising the first claim. Since $\mu_{\Dol}^{\uu} \circ r_t = \mu_{\Dol}^{\uu}$, it follows that $r_t^{\ed}$ preserves the $T$-moment map. It is also $\uu^{\ver}$-invariant, hence it descends to a retraction on the quotient. 
\end{proof}

\begin{corollary} \label{basechangeing}
Let $\Lambda \subset \Lambda' \subset \C^{\ed}$ be semiconical. Fix a torus $T \subset \U_1^{\ver}$. The inclusion $(q^{\ed})^{-1}(\Lambda) \sslash T \to (q^{\ed})^{-1}(\Lambda') \sslash T$ induces an isomorphism in cohomology.
\end{corollary}

\begin{corollary} \label{centralandversalrestrictionisoms}
The inclusions $q_{\Gamma}^{-1}(0) \to \Dol(\Gamma)$ and $\Dol(\Gamma) \to \versD(\Gamma, \eta)$ 
induce isomorphisms in cohomology.
\end{corollary}
\begin{proof}
Both are special cases of Corollary \ref{basechangeing}.  Indeed, 
We have $\Dol(\Gamma) = (q^{\ed})^{-1}(\Lambda) / T$ where $\Lambda = \H_1(\Gamma ;\C)$ and $T = \U_1^{\ver}$. The same holds for $\versD(\Gamma, \eta)$ with $\Lambda = \C^{\ed}$ and $T = \U_1^{\ver}$. 
\end{proof}

\subsection{The deletion-contraction sequence} \label{sec:holmom}

We now consider $X = \Dol(\Gamma / e)$, with $(\U_1, \U_1)$ structure as in Definition \ref{def:ghostaction} and $\zeta = \eta_{t(e)}$. Let us now verify Hypothesis \ref{hyp} for this space.  

\begin{lemma} \label{lem:cond1}
Condition 
\eqref{Xassumptwo} of Hypothesis \ref{hyp} holds for generic choice of $\eta$.  
\end{lemma}
\begin{proof} 
Consider the embedding $\Dol(\Gamma /e, \eta / e) \subset \versD(\Gamma / e, \eta / e)$. By construction, both the action of $\H^1(\Gamma, \U_1)$ and the map $\mu_{X}$ extend to the 
larger space. 

By Lemma \ref{lem:deletionfromcontraction}, we have $\versD(\Gamma \setminus e, \eta \setminus e) = \versD(\Gamma / e, \eta / e) \star_{\eta_{t(e)}} \mathbf{0}$. The right-hand space is by definition the $\U_1$ quotient of $\mu_{X}^{-1}(\eta_{t(e)}) \subset \versD(\Gamma / e, \eta / e)$. By genericity of $\eta$, this is a free $\U_1$-quotient. Thus the same is true after restriction to the closed subset $\Dol(\Gamma / e ) \cap \mu_{X}^{-1}(\eta_{t(e)})$.  
\end{proof}

We turn to checking the local constancy asked in Condition \eqref{Xassumpone}.  Note that in our setting, the map $\mu_X: X \to \U_1$ is not proper.  

\begin{lemma} \label{lem:Dolbassumpone}
Let $\mu^{\U_1}_{\res} \colon \Dol(\Gamma) \to \H_1(\Gamma, \U_1)_\eta$ be the residual moment map. Let $\alpha \colon  \H_1(\Gamma, \U_1)_\eta \to \U_1$ be the restriction of a character of $C_1(\Gamma, \U_1)$. For all but a finite set of $\zeta \in \U_1$, Condition (\ref{Xassumpone}) of Hypothesis \ref{hyp} holds, i.e. 
there exists an open neighborhood $\baseopen$ of $\zeta$ and an isomorphism of $(\U_1, \U_1)$-spaces $(\alpha \circ \mu^{\U_1}_{\res} )^{-1}(\baseopen) \cong [(\alpha \circ \mu^{\U_1}_{\res})^{-1}(\zeta) \times \baseopen]$.
\end{lemma}

\begin{proof}
In general, for a map $E \to B$, a (nonlinear) connection is an assignment, for each path in the base $B$ with endpoints
$x, y$, of a diffeomorphism $E_x \cong E_y$, compatible with composition of paths.  Given a stratification of $B$, 
by a stratified connection on $E \to B$, we mean the data of a connection on each stratum.  In the presence of a group action, we 
can discuss equivariant connections (those which commute with the group action). 

Recall from Proposition \ref{prop: gross-wilson} that $\Dol$ maps to $\uutimesC$ and is a principal $\U_1$-bundle over $\D \times \uu \setminus 0 \times 1$.
Fix a $\U_1$-equivariant connection $\nabla^*$ for this bundle (i.e. a principal bundle connection in the usual sense).  
We stratify $\C \times \uu$ by $0 \times 1$ and its complement; this also defines a stratification of the subspace $\D \times \uu$. Then $\Dol \to \D \times \uu$ carries a stratified connection $\nabla$, where $\nabla(\gamma)$ is defined by parallel transport 
using $\nabla^*$ over the open stratum. Since the closed stratum $0 \times 1$ is a point, it requires no extra data. 

We now turn to $\Dol(\Gamma)$, and set $n = |E(\Gamma)|$. Consider the residual moment map $\mu_{\res}^{\C} \times \mu^{\uu}_{\res} \colon \Dol(\Gamma) \to \H_1(\Gamma, \uutimesC)_\eta$. Its image $B := \H_1(\Gamma, \uutimesC)_\eta \cap (\D \times \uu)^n$ inherits a stratification from $(\D \times \uu)^n$. Taking products defines a stratified connection $\nabla^n$ on $\Dol^n \to (\D \times \uu)^n$, which descends to a stratified connection $\nabla_{\Gamma}$ on $\Dol(\Gamma) \to B$.

Let $\nu \in \U_1$. The spaces $\H_1(\Gamma, \C) \times \alpha^{-1}(\nu)$ inherit a stratification from the product stratification on $(\C \times \uu)^n$. This family of stratified spaces can be compactified to a proper family $\P(\H_1(\Gamma, \C) \oplus \C) \times \alpha^{-1}(\nu)$ of stratified spaces over $\uu$. The result is a stratified submersion away from a finite set of points in $\uu$. 

Over the complement of these points, Thom's first isotopy lemma tells us that the family must be locally constant. In other words, for $\zeta \in \uu$ avoiding a finite set of bad points, there exists an open neighborhood $\zeta \in \baseopen$ and a stratification-preserving map $f: \H_1(\Gamma, \C) \times \alpha^{-1}(\baseopen) \to \H_1(\Gamma, \C) \times \alpha^{-1}(\zeta) \times \baseopen$, covering the projection to $\H_1(\Gamma, \C) \times \baseopen$ and restricting to the identity on $\H_1(\Gamma, \C) \times \alpha^{-1}(\zeta)$.

Parallel transport for the stratified connection $\nabla_\Gamma$ lifts $f$ to the desired isomorphism of $(\U_1, \U_1)$-spaces.
\end{proof}

We may apply Proposition \ref{vanishpairisom} to conclude the vertical arrows in 
Diagram \eqref{global dol key diagram} are isomorphisms.  
Then we may extract from Diagram \eqref{global dol key diagram} the sequence: 

\begin{equation}
\begin{tikzcd}
	\rar & \H^{\bullet-2}(\Dol(\Gamma / e) \star \mathbf{n}\homsep \Q)   \arrow[r, "\text{}"] & \H^{\bullet}(\Dol(\Gamma / e) \star \Dol\homsep \Q)  \arrow[r, "\text{}"]  & \H^{\bullet}(\Dol(\Gamma / e) \star (\Dol \setminus \bS_{\Dol})\homsep \Q) \rar & \  \\
\end{tikzcd}
\end{equation}
Recall that the convolution $\star$ products above are $\star_{\U_1}$.  We want to replace these by their submanifolds
given by the corresponding the $\UUCstar$ products.   
Assume $e$ is a nonloop, nonbridge edge of $\Gamma$.  
Using Lemmas \ref{lem:uncontract}, \ref{lem:deletionfromcontraction}, we have: 

\begin{equation}  \label{eq:restrictionsfordolb}
\begin{tikzcd}
	\H^{\bullet-2}(\Dol(\Gamma / e) \star_{\uu} \mathbf{n}\homsep \Q)  \arrow[r, "\kappa_{\mathbf{n}}^* "]
	& 
	\H^{\bullet-2}(\Dol(\Gamma / e) \UUCstar \mathbf{n}\homsep \Q) = 
	\H^{\bullet-2}(\Dol(\Gamma \setminus e) \homsep \Q)  \\ 
	\H^{\bullet}(\Dol(\Gamma / e)  \star_{\uu} \Dol\homsep \Q)  \arrow[r, "\kappa^* "]  
	&
	\H^{\bullet-2}(\Dol(\Gamma / e) \UUCstar  \Dol \homsep \Q) = 
	\H^{\bullet}(\Dol(\Gamma)\homsep \Q)  \\ 
	\H^{\bullet}(\Dol(\Gamma / e)  \star_{\uu} (\Dol \setminus \bS_{\Dol})\homsep \Q) \arrow[r, "i_\epsilon^*"]
	&
	\H^{\bullet}(\Dol(\Gamma / e) \UUCstar (\Dol \setminus \bS_{\Dol})\homsep \Q)  =
	\H^{\bullet} (\Dol(\Gamma / e) \homsep \Q) 
\end{tikzcd}
\end{equation}

\begin{theoremdef} \label{td:doldcs}
Each of the restriction maps is an isomorphism.  We may therefore define the lower row in the following diagram
by requiring that the diagram commute.  

\begin{equation} \label{pinchingparttwo}
	\begin{tikzcd}
		\rar & \H^{\bullet-2}(\Dol(\Gamma / e) \star \mathbf{n}\homsep \Q) \arrow[d, "\kappa_{\mathbf{n}}^*"]  \arrow[r, "\text{}"]  
		& \H^{\bullet}(\Dol(\Gamma / e) \star \Dol\homsep \Q)  \arrow[d, "\kappa^*"]  \arrow[r, "\text{}"]  
		& \H^{\bullet}(\Dol(\Gamma / e) \star (\Dol \setminus \bS_{\Dol}) \homsep \Q) \arrow[d, "i_\epsilon^*"] \rar & \  \\
		\arrow[r, dashed, "c^{\Dol}_e"] &  \H^{\bullet-2}(\Dol(\Gamma \setminus e)\homsep \Q) \arrow[r, dashed, "a^{\Dol}_e"]
		& \H^{\bullet}(\Dol(\Gamma)\homsep \Q) \arrow[r, dashed, "b^{\Dol}_e"] 
		& \H^{\bullet}(\Dol(\Gamma / e)\homsep \Q) \arrow[r, dashed, "c^{\Dol}_e"] & \ \\
	\end{tikzcd}
\end{equation}

We term this lower row the Dolbeault deletion-contraction sequence ($\Dol$-DCS). 
\end{theoremdef}
\begin{proof}
To see that the map $\kappa_{\mathbf{n}}^*$ is an isomorphism, note that we are in the setting of Corollary \ref{basechangeing}, with $\Lambda = \H_1(\Gamma / e\homsep \C)$, $\Lambda' = \H_1(\Gamma \setminus e\homsep \C)$ and $T = \U_1^{V(\Gamma)}$.

Similarly, the fact that $\kappa^*$ is an isomorphism follows from Corollary \ref{basechangeing} with $\Lambda' = \H_1(\Gamma / e\homsep \C))$, $\Lambda = \H_1(\Gamma ;\C)$ and $T = \U_1^{V(\Gamma)}$.

Finally, $i_\epsilon^*$ is an isomorphism by Lemma \ref{lem: retractions}; we have also used the identification \eqref{q2nonzeropreimage}. 
\end{proof}

\begin{lemma} \label{lem:smalldiagcommutessodolbcommutes}
The maps $a^{\Dol}_e$ for different edges $e$ commute when their composition is defined. 
\end{lemma}
\begin{proof}
Using the isomorphism  $\kappa_{\mathbf{n}}^* : \H^{\bullet}(\Dol(\Gamma / e) \UUCstar \bS_{\Dol}, \Q) \cong \H^{\bullet}(\Dol(\Gamma \setminus e), \Q)$, the map $a^{\Dol}_e$ may be identified with the Gysin map for the codimension two embedding of manifolds $\Dol(\Gamma / e) \UUCstar \bS_{\Dol} \subset \Dol(\Gamma)$. The proof then proceeds along the same lines as the corresponding argument for $a^{\Bet}_e$ in Lemma \ref{lem:gysincommutes}. 

We will give a second, independent proof of this Lemma \ref{lem:smalldiagcommutessodolbcommutes} 
in Remark \ref{rem:smalldiagcommutessodolbcommutes2}, by deducing the commutation of $a_e^{\Dol}$ from that of $a_e^{\Bet}$. 
\end{proof}
The commutativity allows us to make the following special case of Definition \ref{def:delfiltration}.
\begin{definition} \label{def:dolbeaultdeletionfiltration}
The {\em Dolbeault deletion filtration} is the increasing filtration $\delfilt_r \H^{n}(\Dol(\Gamma),\Q)$ obtained from Definition  \ref{def:delfiltration}, where the functor $A$ takes $\Gamma$ to $\H^{\bullet}(\Dol(\Gamma), \Q)$ and takes $\Gamma' \to \Gamma$ to the composition, in any order, of $a_e^{\Dol}$ for $e \in \Gamma \setminus \Gamma'$.
\end{definition}

\subsection{Another sequence} 
The top row of \eqref{pinchingparttwo} is identified with the long exact sequence
of a pair upon replacing $\mathbf{n}$ with $\bS_{\Dol}$ in the upper left corner 
($\mathbf{n} \to \bS_{\Dol}$ is the inclusion of a point in a line, hence induces isomorphisms
in cohomology).  Said long exact sequence has a natural map to its $\UUCstar$ version: 

\begin{equation} \label{pinchingpartthree}
\hspace{-10mm}
\begin{tikzcd}
	\H^{\bullet-2}(\Dol(\Gamma / e) \star \bS_\Dol \homsep \Q) \arrow[d, "\kappa_{\bS}^*"]  \arrow[r, "\text{}"]  
	& \H^{\bullet}(\Dol(\Gamma / e) \star \Dol\homsep \Q)  \arrow[d, "\kappa^*"]  \arrow[r, "\text{}"]  
	& \H^{\bullet}(\Dol(\Gamma / e) \star (\Dol \setminus \bS_{\Dol})\homsep \Q) \arrow[d, "\kappa_\circ^*"] \rar & \  \\
	\H^{\bullet-2}(\Dol(\Gamma / e) \UUCstar \bS_\Dol \homsep \Q) \arrow[r, "A^{\Dol}_e"]
	& \H^{\bullet}(\Dol(\Gamma)\homsep \Q) \arrow[r, "B^{\Dol}_e"] 
	& \H^{\bullet}(\Dol(\Gamma / e) \UUCstar (\Dol \setminus \bS_\Dol) \homsep \Q) \arrow[r, "C^{\Dol}_e"] & \ \\
\end{tikzcd}
\end{equation}

Composing Diagrams \eqref{pinchingparttwo} and \eqref{pinchingpartthree} gives:  

\begin{equation} \label{pairstodolb}
\begin{tikzcd}
	\H^{\bullet-2}(\Dol(\Gamma / e) \UUCstar \bS_\Dol \homsep \Q) \arrow[r, "A^{\Dol}_e"] \arrow[d, equals] 
	& \H^{\bullet}(\Dol(\Gamma)\homsep \Q) \arrow[r, "B^{\Dol}_e"] 
	& \H^{\bullet}(\Dol(\Gamma / e) \UUCstar (\Dol \setminus \bS_\Dol) \homsep \Q) \arrow[r, "C^{\Dol}_e"]  \arrow[d, " i_\epsilon^* (\kappa_\circ^*)^{-1}"] & \ \\
	\H^{\bullet-2}(\Dol(\Gamma \setminus e),\Q) \arrow[r, "a^{\Dol}_e"] 
	& \H^{\bullet}(\Dol(\Gamma),\Q) \arrow[u, equals] \arrow[r, "b^{\Dol}_e"] 
	&   \H^{\bullet}(\Dol(\Gamma / e), \Q) \arrow[r, "c^\Dol_e"]  & \ \\
\end{tikzcd}
\end{equation}
The left vertical	$=$ is induced by the inclusion $\bf{n} \to \bS_{\Dol}$.

For later use we record the following structure: 

\begin{equation} \label{diag:compinclusionrest}
\begin{tikzcd}
	\Dol(\Gamma / e) \UUCstar (\Dol \setminus \bS_\Dol) \arrow[d] \arrow[r, "\kappa_\circ "] & \arrow[d, "q_{\operatorname{res}} \times q"]  \Dol(\Gamma / e) \star (\Dol \setminus \bS_\Dol) \\
	\H_1(\Gamma\homsep \C) \arrow[r] & \H_1(\Gamma / e\homsep \C) \times \C
\end{tikzcd}
\end{equation}

\addtocontents{toc}{\protect\newpage}

\section{Fibers and monodromy} \label{sec:system}

\subsection{Structure of the generic fiber}

We write $\D^{\ed}_{\reg}$ for the complement of the coordinate hyperplanes. 
Our first step is to give a natural presentation of the fundamental group of a fiber over 
$\D^{\ed}_{\reg}$. Let $b \in \D^{\ed}_{\reg}$. We write $\versD(\Gamma)_b \colonequals q_{\res}^{-1}(b)$.
\begin{lemma} \label{lem:cyclesequence}
There is a natural short exact sequence of groups
\[ \H^1(\Gamma\homsep\Z) \to \pi_1(\versD(\Gamma)_b) \to \H_1(\Gamma\homsep\Z). \]
\end{lemma}
\begin{proof}
The basic space $\Dol$ is defined as a $\Z$-quotient. Let $b \in \D^*$, and let $\C^* \to \C^*/b^{\Z} \cong q^{-1}(b)$ be the restriction of this quotient to the fiber. It induces an inclusion of fundamental groups, defining a short exact sequence 
$$\pi_1(\C^*) \to \pi_1(q^{-1}(b)) \to \Z,$$ where the image is identified via the inclusion to $\Dol$ with $\pi_1(\Dol)$. 

Now let $b \in \D^{\ed}_{\reg}$. The point $b$ determines a product of elliptic curves $\bE_b \colonequals \prod_{e \in \ed} q_e^{-1}(b_e) $ in $\Dol^{\ed}$, and $\versD(\Gamma)_b$ is the K\"ahler reduction of $\bE_b$ by $\overline{ \U_1^{V(\Gamma)}}$. More precisely, there is a moment map $\mu_{\Gamma}^{\U_1} \colon \bE_b \to \U_1^{V(\Gamma)}$ for the action of $\overline{ \U_1^{V(\Gamma)}}$ on $\bE_b$, and $\versD(\Gamma)_b$ is the quotient of the fiber $\bE_b(\eta)$ over $\eta$. Taking cartesian products of the basic sequence of fundamental groups, we obtain a sequence which we may write as
\[ C^1(\Gamma\homsep\Z) \to \pi_1(\bE_b) \to C_1(\Gamma\homsep \Z). \]
The inclusion $\bE_b(\eta) \to \bE_b$ gives the embedded short exact sequence
\[ C^1(\Gamma\homsep\Z) \to \pi_1(\bE_b(\eta)) \to \H_1(\Gamma\homsep \Z). \]
The quotient $\bE_b(\eta)/ \U_1^{V(\Gamma)}$ defines the quotient short exact sequence
\[ \H^1(\Gamma\homsep\Z) \to \pi_1(\versD(\Gamma)_b) \to \H_1(\Gamma\homsep\Z). \]
\end{proof}

We now give a description of $\versD(\Gamma)_b$ as a group quotient. 

Recall that $b_e$ for $e \in \ed$ be the coordinates of $b$ in $\C^{\ed}$. To alleviate notation, we will number the edges of $\Gamma$ $e_1$ through $e_n$ and write $b_{i}$ for $b_{e_i}$. Given $\beta \in \H_1(\Gamma\homsep\Z)$, consider $b^{\beta} \colonequals (b_{1}^{\beta_{1}}, ...., b_{n}^{\beta_{n}}) \in \C^n$ where $\beta_i$ are the coordinates of the image of $\beta$ under the pullback $\H_1(\Gamma\homsep\Z) \to C_1(\Gamma \homsep \Z)$. Since by assumption, all of the $b_i$ are nonzero, $b^{\beta}$ defines an element of $C^1(\Gamma\homsep \C^*)$, and we write $b^{\beta}$ for its image in $\H^1(\Gamma\homsep\C^*)$. This defines a map 	$\tau_b\colon \H_1(\Gamma\homsep\Z) \to \H^1(\Gamma\homsep \C^*)$. We write $b^{\H_1(\Gamma\homsep\Z)}$ for the image of $\tau_b$.
\begin{proposition} \label{prop:latticequotient}
$b^{\H_1(\Gamma\homsep\Z)}$ is a discrete lattice in $\H^1(\Gamma\homsep\C^*)$. The fiber $q_{\res}^{-1}(b)$ is naturally isomorphic to the quotient $\H^1(\Gamma\homsep\C^*) / b^{\H_1(\Gamma\homsep\Z)}$.
\end{proposition}
\begin{proof}

Consider the cover $(\C^*)^{\ed} \to \bE_b$, obtained by taking the Cartesian product of the maps $\C^* \to \C^*/b_e^{\Z} \cong q^{-1}(b_e)$. The torus-valued moment map $\mu_{\Gamma}^{\U_1}$ lifts to a real-valued moment map $\mu_{\Gamma}^{\R} \colon (\C^*)^{\ed} \to \R^{V(\Gamma)}$. Pick any lift $\widetilde{\eta}$ of $\eta$; the quotient $(\mu_{\Gamma}^{\R})^{-1}(\widetilde{\eta}) / \overline{ \U_1^{V(\Gamma)}}$ is the Galois cover of $\versD(\Gamma)_b$ corresponding to the subgroup $\H^1(\Gamma\homsep\Z) \subset \pi_1(\versD(\Gamma)_b)$. By the Kempf-Ness theorem, we can identify it with $(\C^*)^{\ed}/(\C^*)^{V(\Gamma)} = \H^1(\Gamma\homsep\C^*)$. We can compute the action of an element $\gamma \in \pi_1(\versD(\Gamma)_b)/\H^1(\Gamma\homsep\Z) = \H_1(\Gamma\homsep\Z)$ on the cover by choosing a lift to $\pi_1(\bE_b)$; we find it is given by multiplication by $\tau_b(\gamma)$. This proves the second claim. 

Discreteness of the image of $\tau_b$ can be deduced from the fact that the quotient is a manifold. Here we give a direct proof.

The torus $C^1(\Gamma\homsep\C^*)$ splits into a real and a compact factor : $C^1(\Gamma\homsep\C^*) = C^1(\Gamma\homsep \U_1) \times C^1(\Gamma\homsep\R^{>0})$. Likewise, we have $\H^1(\Gamma\homsep\C^*) = \H^1(\Gamma\homsep \U_1) \times \H^1(\Gamma\homsep\R^{>0})$. The exponential map defines isomorphisms $C^1(\Gamma\homsep \R) \cong  C^1(\Gamma\homsep\R^{>0})$ and $\H^1(\Gamma\homsep\R) \cong \H^1(\Gamma\homsep\R^{>0})$. Postcomposing $\tau_b$ with the projection $\H^1(\Gamma\homsep\C^*) \to \H^1(\Gamma\homsep\R^{>0}) \cong \H^1(\Gamma\homsep\R) $ defines a map $\H_1(\Gamma\homsep\R)\to \H^1(\Gamma\homsep\R)$. Tensoring the left-hand side with $\R$, we obtain a map of vector spaces
\[ \overline{\tau_b} \colon \H^1(\Gamma\homsep\R) \to \H^1(\Gamma\homsep\R). \]
It is enough to show that this map is an isomorphism. Let $c_i = \log|b_i| < 0$, and let $[e] \in \H^1(\Gamma\homsep\Z)$ be the element of cohomology corresponding to the oriented edge $e$. Then $\overline{\tau_b}(\beta) = \sum_{i = 1}^n c_i \beta(v_i) [e]$.

Define an inner product on $C_1(\Gamma\homsep\R)$ by $\langle \mathbf{x}, \mathbf{x}' \rangle_{c} \colonequals \sum_{i=1}^n -c_i \mathbf{x}_e \mathbf{x}'_e$. Since it is manifestly positive definite, so is its pullback along the injection $\H_1(\Gamma\homsep\R) \to C_1(\Gamma\homsep\R)$. $\overline{\tau_b}$ is the map $\H_1(\Gamma\homsep\R) \to (\H_1(\Gamma\homsep\R))^{\vee} = \H^1(\Gamma\homsep\R)$ given by $\beta \to \langle \beta, - \rangle_c$. It follows that it is an isomorphism, as was to be shown.
\end{proof}

\begin{corollary}
The restriction of $q_{\res} \colon \versD(\Gamma) \to \D^{\ed}$ has a section. 
\end{corollary}
\begin{proof}
We can define such a section by taking the image of the unit section of the trivial fibration $(\C^*)^{\ed}$ under the quotient map $(\C^*)^{\ed} \to \H^1(\Gamma\homsep\C^*) \to \versD(\Gamma)_b$ from Proposition \ref{prop:latticequotient}. 
\end{proof}

\begin{proposition} \label{prop:abelianvariety}
For $b \in \D^{\ed}$ in the complement of the coordinate hyperplanes, $q_{\res}^{-1}(b)$ is an abelian variety. \end{proposition}
\begin{proof}
It suffices to show that $q_{\res}^{-1}(b)$ is a compact complex group admitting a projective embedding. We have shown that $\versD(\Gamma)_b$ is a compact complex group. We will now show that $\versD(\Gamma)_b$ carries a K\"ahler form $\omega_b$ with integral pairings $\omega_b(\beta)$ for $\beta \in \H_2(\versD(\Gamma)_b \homsep \Z)$. The Kodaira embedding theorem then tells us that $\versD(\Gamma)_b$ admits a projective embedding. 
Let $\widetilde{\omega}_b$ be the K\"ahler form on $\bE_b$; recall that we have chosen it to be integral. We can represent any curve class $\beta \in \H_2(\versD(\Gamma)_b,\Z)$ as the image under the quotient map of a curve $\widetilde{\beta}$ in $\bE(\eta)$. The K\"ahler form on $\versD(\Gamma)_b$ is obtained by reduction of that on $\bE_b$, and thus $\omega_b(\beta) = \widetilde{\omega}_b(\widetilde{\beta}) \in \Z$. 

\end{proof}

\subsection{Monodromy}
For a torus $A$, we have a natural isomorphism $\H^{\bullet}(A\homsep\C) = \bigwedge^{\bullet} \H^1(A;\C)$. Thus $R^{\bullet} q_{\res *} \Q_{\Dol(\Gamma)}$ is a graded local system on a hyperplane in $\D^{\ed}_{\reg}$, with fiber at $b$ given by $\bigwedge^{\bullet} \H^1(\versD(\Gamma)_b,\Z)$. 

The monodromy of this local system is determined by the monodromy in degree one. This is described as follows. Consider the short exact sequence in cohomology
\begin{equation} \label{eq:SESHone} \H^1(\Gamma, \Z) \to \H^1(\versD(\Gamma)_b,\Z) \to \H_1(\Gamma, \Z) \end{equation}
dual to that of Lemma \ref{lem:cyclesequence}. 
\begin{proposition} \label{prop:monodromycomp}
Fix an edge $e$ of $\Gamma$, and consider the corresponding hyperplane in $C_1(\Gamma\homsep\C)$. The logarithm of the monodromy of $R^{1}q_{\res *} \Q_{\Dol(\Gamma)}$ around this hyperplane is given by the composition $\H^1(\versD(\Gamma)_b,\Z) \to \H_1(\Gamma ,\Z) \xrightarrow{ \langle e, - \rangle [e]} \H^1(\Gamma ,\Z) \to \H^1(\versD(\Gamma)_b,\Z)$.
\end{proposition}
\begin{proof}
Fix a basepoint $b$ near the hyperplane $b_e = 0$.  Fix a basis $e_1 = e, e_2, .., e_g$ of $\H^1(\Gamma\homsep \Z)$ and a basis $\gamma_1, ..., \gamma_g$ of $\H_1(\Gamma\homsep\Z)$ such that $\langle \gamma_i, e \rangle =0$ for $i \neq 1$. Thus $\langle e, \gamma \rangle$ picks out the coefficient of $\gamma_1$ in $\gamma = \sum_{i} c_i \gamma_i$.

Recall the equality $\versD(\Gamma)_b = \H^1(\Gamma \homsep \C^*) / b^{\H_1(\Gamma\homsep \Z)}$.  Choose branches of the logarithms $\log(b_{e_i})$, and define $\tilde{\gamma_i} \in \H_1(\versD(\Gamma)_b)$ as the cycle 
\begin{equation} \label{eq:explicitcycle} r \in [0,1] \to \{ \exp \big( r \log(b_{e_i}) \langle \gamma_i, e \rangle \big) \}_{e \in \ed}. \end{equation}
Let $\tilde{e_i} \in \H_1(\H^1(\Gamma\homsep\C^*),\Z)$ be the tautological cycles; we abusively use the same notation for their projections to $\H_1(\versD(\Gamma)_b \homsep \Z)$. Then $\tilde{e_i}$ and $\tilde{\gamma_i}$ form a basis of $\H_1(\versD(\Gamma)_b \homsep \Z)$. By following the explicit cycle \ref{eq:explicitcycle} as $b_0 \to \exp(2 \pi i \theta) b_0$, one can verify the proposition. 

Note, however, that the general form of the answer follows without any further calculations. Consider a small loop around the hyperplane $b_e = 0$, starting and ending at $b$. By construction, all the basis elements but $\gamma_1$ are globally defined along this loop; it follows that the log monodromy factors through $\H_1(\versD(\Gamma)_b,\Z) \to \H_1(\Gamma\homsep\Z) \xrightarrow{\langle e, - \rangle} \Z$. Applying the same reasoning to the Poincar\'e dual basis in $\H_{2g-1}(\versD(\Gamma)_b\homsep \Z) = \bigwedge^{2g-1} \H_1(\versD(\Gamma)_b \homsep \Z)$, we see that the image of log monodromy must lie in the span of $e$. 
\end{proof}

\subsection{Structure of the special fiber}
By construction, $q_{\res}^{-1}(0)$ is the symplectic reduction of $q^{-1}(0)^{\ed}$ by $\redC^0(\Gamma\homsep \U_1)$. It is easier to understand this reduction by first passing to the universal cover $\widetilde{q^{-1}(0)}^{\ed}$ of $q^{-1}(0)^{\ed}$. The universal cover of $q^{-1}(0)$ is an infinite chain of rational curves $\mathbb{P}^1_n$ which we index by the integers $n \in \Z$. Thus $\widetilde{q^{-1}(0)}^{\ed}$ is an infinite grid of irreducible components $\prod_{e \in \ed} \mathbb{P}^1_{n_e}$. 

To understand the reduction of  $\widetilde{q^{-1}(0)}^{\ed}$, we will use Delzant's dictionary between polytopes and toric varieties \cite{Delz}, according to which a toric variety is classified by its image under the moment map. The moment map
\[ \mu_\Dol^{E(\Gamma)} : \widetilde{q^{-1}(0)}^{\ed} \to C_1(\Gamma, \R) = \R^{E(\Gamma)} \] maps the component indexed by $\mathbf{n} = \{ n_e \}$ to the cube $\square_{\mathbf{n}} := \prod_{e \in E(\Gamma)} [n_e, n_e+1]$. These cubes are the chambers of the coordinate periodic hyperplane arrangement on $C_1(\Gamma, \R)$.

Identify $\U_1$ with $\R/\Z$, and suppose that $\eta \in \redC_0(\Gamma\homsep\Q/\Z)$. Let $\widetilde{\eta}$ be a lift to $\redC_0(\Gamma \homsep \Q)$ with all components in the range $0 \leq \widetilde{\eta}_v \leq 1$. The affine subspace $d_{\Gamma}^{-1}(\widetilde{\eta}) \subset C_1(\Gamma, \R)$ intersects the coordinate periodic arrangement in a $\H_1(\Gamma, \Z)$-periodic arrangement $A^{\operatorname{per}}$. It is given by the hyperplanes $\langle \gamma, e \rangle = n + \tilde{\eta}_e$ for $e \in E(\Gamma), n \in \Z$. 	By our genericity assumptions on $\eta$, $A^{\operatorname{per}}$ is a simple unimodular arrangement, i.e. any $k$ hyperplanes intersects in codimension $k$, and the integral normal vectors at such an intersection span the lattice $\H_1(\Gamma\homsep\Z)$. 

The chambers of $A^{\operatorname{per}}$ are given by $\Delta_{\mathbf{n}} := \square_{\mathbf{n}} \cap d_{\Gamma}^{-1}(\widetilde{\eta})$, for those $\mathbf{n}$ such that the right-hand side is nonempty. Each such chamber $\Delta_{\mathbf{n}}$ corresponds to a component $\mathfrak{X}_{\mathbf{n}} \colonequals \prod_{e \in \ed} \mathbb{P}^1_{n_e} \sslash_{\widetilde{\eta}} \redC^0(\Gamma\homsep \U_1)$ of the reduction. The reduction of a toric variety by a torus action is toric, with moment map obtained by restriction from the moment map of the prequotient, and in particular the moment image of $\mathfrak{X}_{\mathbf{n}}$ is $\Delta_{\mathbf{n}}$. The components $\mathfrak{X}_{\mathbf{n}}$ and $\mathfrak{X}_{\mathbf{m}}$ intersect along the sub-toric variety determined by the mutual face $\Delta_{\mathbf{n}} \cap \Delta_{\mathbf{m}}$ of their polytopes.	

We thus obtain a description of $\widetilde{q_{\res}^{-1}(0)}$ as a union of smooth toric varieties glued along toric subvarieties, whose moment map defines an infinite periodic hyperplane arrangement. The fiber $q_{\res}^{-1}(0)$ itself is obtained by quotienting this picture by $\H_1(\Gamma\homsep\Z)$.

\subsection{Class of the central fiber in the Grothendieck group of varieties} 

The map $q: \Dol \to \D$ has fiber $q^{-1}(0)$ a nodal rational curve with dual graph $\boing$.  Thus in
the Grothendieck group of varieties, 
$|q^{-1}(0)| = |\C^*| + |\mathrm{point}| = \L$.  

The preceeding description of $q_{\res}^{-1}(0)$ as a union of toric varieties gives
one way to compute $|q_{\res}^{-1}(0)|$.  Here we give a different argument: 

\begin{proposition} \label{prop:fixedpointdescription}
The $\H^1(\Gamma\homsep\C^*)$-fixed points of $\Dol(\Gamma)$ are indexed by the spanning trees $\Gamma' \subset\Gamma$. The fixed point $p_{\Gamma'}$ is the reduction by $\redC^0(\Gamma\homsep \U_1)$ of the subspace  $\prod_{e \notin \Gamma'} \mathbf{n} \times \prod_{e \in \Gamma'} \Dol \subset \prod_{e \in \Gamma} \Dol$.  
\end{proposition}
\begin{proof}
Let $p \in \Dol(\Gamma)$, and let $\tilde{p}$ be a lift to $\Dol^{\ed}$. $p$ is fixed by $\H^1(\Gamma\homsep\U_1)$ if and only if the action of $C^1(\Gamma\homsep\U_1)$ on $\tilde{p}$ preserves the $\redC^0(\Gamma\homsep\U_1)$-orbit of $\tilde{p}$. 

Let $\Gamma'$ be the unique subgraph of $\Gamma$ such that $\tilde{p} \in  \prod_{e \notin \Gamma'} \mathbf{n} \times \prod_{e \in \Gamma'}  (\Dol \setminus \mathbf{n}) \subset \prod_{e \in \Gamma} \Dol$. Since we know the $\redC^0(\Gamma\homsep \U_1)$ orbit of $\tilde{p}$ is free by assumption, this subgraph must contain all vertices of $\Gamma$. 

The action of $C^1(\Gamma\homsep \U_1)$ on this subspace factors through a free action of $C^1(\Gamma' \homsep \U_1)$, which descends to a free action of $\H^1(\Gamma'\homsep \U_1) = C^1(\Gamma'\homsep \U_1) / \redC^0(\Gamma' \homsep \U_1)$ on the quotient space. Hence $p$ is a fixed point if and only if $\H^1(\Gamma'\homsep \U_1)$ is trivial, i.e. $\Gamma'$ is a tree.

From the description of $p_{\Gamma'}$, it follows that it is also fixed by $\H^1(\Gamma\homsep\C^*)$. 
\end{proof} 
\begin{corollary}
All $\H^1(\Gamma\homsep\C^*)$-fixed points of $\Dol(\Gamma)$ are contained in the central fiber $q_{\res}^{-1}(0)$. 
\end{corollary}
\begin{proof}
By Proposition \ref{prop:fixedpointdescription}, $q_{\res}(p_\Gamma) \subset \mu^{\ed} \big( \prod_{e \notin \Gamma'} \mathbf{n} \times \prod_{e \in \Gamma'} \Dol \big)$.
In turn, the right-hand side is contained in the image of $C_1(\Gamma', \C) \to C_1(\Gamma, \C)$. Since $\Gamma' \subset \Gamma$ is a tree, the latter intersects $\H_1(\Gamma, \C)$ at $0$. 
\end{proof}

\begin{theorem} \label{thm:dolcfclass}
In the Grothendieck group of varieties, the class of the central fiber is:
$$|q_{\res}^{-1}(0)| = (\text{\# of spanning trees of }
\Gamma) \times \A^{h^1(\Gamma)}$$
\end{theorem}
\begin{proof}
Pick a cocharacter $\sigma \colon \C^* \to \H^1(\Gamma\homsep \C^*)$ whose image is not contained in the kernel of any restriction map $\H^1(\Gamma\homsep \C^*) \to \H^1(\Gamma \setminus e\homsep \C^*)$. 

The resulting complex analytic action of $\C^*$ on $\Dol(\Gamma)$ preserves the fibers of $q_{\res}$ and has isolated fixed points $p_{\Gamma'} \in q_{\res}^{-1}(0)$, naturally indexed by the spanning trees $\Gamma' \subset \Gamma$.

The attracting cell of a fixed point is a smooth variety with a contracting $\C^*$-action, hence 
isomorphic to $\C^n$ for some $n$. 
Per Proposition \ref{prop: action is symplectic}, the action of $\H^1(\Gamma\homsep\C^*)$ preserves the holomorphic 
symplectic form on $\Dol(\Gamma)$.  It follows that the attracting cell is Lagrangian, and thus of complex dimension 
$1/2 \dim \Dol(\Gamma) = h^1(\Gamma)$. 
Since $q_{\res}^{-1}(0)$ is proper, every $\C^*$-orbit has a limit point, and so the fiber is the disjoint union of these attracting cells.
\end{proof}
\begin{example}
The central fiber of $\Dol = \Dol(\boing)$ is the disjoint union of the node $\mathbf{n}$ and a copy of $\C^*$. Their union is the attracting cell of the node, with respect to either the usual action of $\C^*$ or its inverse.
\end{example}

\begin{remarque} \label{rem:dolafclass}
The class of the general fiber has a similar description, which we will not need in this paper. Namely, let $\Gamma' \subset \Gamma$ be the (possibly disconnected) subgraph consisting of the edges $e \in \ed$ for which $b_e \neq 0$. We can view $b$ as a generic point in $\D^{E(\Gamma')}$. This defines an abelian variety $\versD(\Gamma')_b$; it is a product of smaller abelian varieties determined by the connected components of $\Gamma'$; the general form of the factors is described explicitly in Proposition \ref{prop:latticequotient} below. Then we have 
\[ |q_{\res}^{-1}(b)| = |\versD(\Gamma')_b| \times (\text{ \# of spanning trees of } \Gamma / \Gamma') \times \A^{h^1(\Gamma / \Gamma')} \]
Here $\Gamma / \Gamma'$ is the contraction of $\Gamma$ by $\Gamma'$, where exceptionally we allow the contraction of subgraphs of genus $>0$.
\end{remarque}

\subsection{Projectivity}  Here we show that $q_{\operatorname{res}}$ is projective, at least near the central fiber.  The argument
is independent of Proposition \ref{prop:abelianvariety}.

Recall the definition of $\widetilde{\Dol}$ from Section \ref{subsec:basicdolbdef}. We define a line bundle $\cL$ on $\widetilde{\Dol}$ with transition function $x_n$ on the overlaps $\C^2_n \cap \C^2_{n+1}$. $\cL$ is naturally equivariant with respect to the $\C^*$ and $\Z$ actions. It is not, however, jointly equivariant. Instead, if $s_1^* \cL$ denotes the shift of $\cL$ by $1 \in \Z$, we have an equality of $\C^*$-equivariant bundles $s_1^*\cL = \chi \cL$ where $\chi$ is the fundamental character of $\C^*$.

The following proposition is direct from the definition of $\cL$:
\begin{proposition} \label{prop:Lample}
$\cL$ restricts to an ample bundle on any finite chain of rational curves in the fiber $\tilde{q}^{-1}(0) \subset \widetilde{\Dol}$. 
\end{proposition}   

\begin{theorem} \label{thm:projective} 
There is a line bundle $\cL_{\Gamma}$ on $\versD(\Gamma, \eta)$ and an open neigborhood of $0 \in \H_1(\Gamma\homsep\C)$ such that for any $b$ in this neighborhood, the restriction of $\cL_{\Gamma}$ to $q_{\res}^{-1}(b)$ defines a projective embedding. 
\end{theorem}
\begin{proof}

We will construct such a bundle starting from the bundle $\cL$ in Proposition \ref{prop:Lample}. By Proposition 1.4 of \cite{Nak}, if $f\colon X \to S$ is a proper map of complex manifolds, and a line bundle $L$ on $X$ is ample on a given fiber, then it is relatively ample over a neighborhood of the image. Thus after constructing $\cL_{\Gamma}$, it will be enough to check its ampleness on the central fiber.

Let $\phi \in C_1(\Gamma, \Z)$ satisfy $d_\Gamma(\phi) = N \widetilde{\eta}$ for some integer $N$. Consider the $C^1(\Gamma, \C^*)$-equivariant bundle $$\cL^N_\phi \colonequals \phi \otimes \bigg( \boxtimes_{e \in E(\Gamma)} \cL_e^{N} \bigg)$$ on $\widetilde{\Dol}^{\ed}$. It also carries an action of $C_1(\Gamma, \Z)$ which does not commute with the torus action. The image of $\redC^0(\Gamma, \C^*) \to C^1(\Gamma, \C^*) = (\C^*)^n$, however, commutes with the action of $\H_1(\Gamma\homsep\Z)$. Thus $\cL^{N}_\phi$ descends to a $\H_1(\Gamma\homsep\Z)$-equivariant bundle $\widetilde{\cL}_{\Gamma}$ on $\widetilde{\Dol}^n \sslash_{\eta} \redC^0(\Gamma\homsep \U_1)$. 

A component by component application of the dictionary between polytopes and toric varieties shows that $\widetilde{q_{\res}^{-1}(0)}$ is the union of GIT quotients $\prod_{e \in \ed} \mathbb{P}^1_{n_e} \sslash_{\cL^{N}_\phi} \redC^0(\Gamma, \C^*)$, glued along GIT quotients of subvarieties, such that $\widetilde{\cL}_{\Gamma}$ restricts to the GIT bundle $\mathcal{O}(1)$ on any component. It follows that $\widetilde{\cL}_{\Gamma}$ is ample on any finite union of components. 

We can now conclude that the descent $\cL_{\Gamma}$ of $\widetilde{\cL}_{\Gamma}$ to $q_{\res}^{-1}(0)$ is also ample, by the same argument as in \cite{Mum}, Theorem 3.10. 
\end{proof}

\section{Perverse Leray filtrations}  \label{sec:perverse}

Given a map $f: X \to B$ of algebraic varieties, the middle perverse $t$-structure on $B$ induces a filtration 
-- the perverse Leray filtration -- on the cohomology of $X$.  We recall some facts about this filtration in 
Appendix \ref{app:perverse}.   

\begin{convention}  \label{def:pervleraydolb}
When we speak without further qualification of `the' perverse Leray filtration on $\H^{\bullet}(\Dol(\Gamma),\Q)$,
we mean the one associated to the map $q_{\res}  \colon \Dol(\Gamma) \to \H_1(\Gamma\homsep \C)$ (which was defined in Proposition \ref{prop:dolbintegdef}). 
Likewise, by `the' perverse Leray filtration on $\H^{\bullet}(\versD(\Gamma),\Q)$, we mean the one associated
to $q_{\res}  \colon \versD(\Gamma) \to \C^{E(\Gamma)}$. 
\end{convention} 

There are two key takeaway results from this section.  The first is 
Theorem \ref{thm:abstisom} establishing an
isomorphism $\H^{\bullet}(\Dol(\Gamma)\homsep \Q) \cong \H^{\bullet}(\CKS(\Gamma) \homsep \Q)$
intertwining the perverse filtration with the $\CKS$ filtration.  The argument adapts the 
methods of \cite{MSV} (where the complex $\CKS(\Gamma)$ plays a similar role for
the Jacobian of a nodal curve with dual graph $\Gamma$).  The second is 
Proposition \ref{prop: deletion perverse compatible},  where we show (non-strict) compatibility
of the deletion map $a_\Dol^e$ with the perverse Leray filtration.

\subsection{Compatibility of  \eqref{global dol key diagram} with perverse filtrations}
Let us begin with the following general discussion.  Suppose $\cK$ is a complex of sheaves on a space $Y$ 
equipped with a map $f \colon Y \to \C$. Then the nearby-vanishing triangle defines a long exact sequence 
\begin{equation} \label{nearbyvanishingperv}
\to \mathbb{H}^{\bullet}(f^{-1}(0); \Phi_f \cK) \to \mathbb{H}^{\bullet}(f^{-1}(0)\homsep \cK|_{f^{-1}(0)}) \to \mathbb{H}^{\bullet}(f^{-1}(0) \homsep \Psi_f \cK) \xrightarrow{[1]} 
\end{equation}
If we view the first and last terms in the sequence as carrying the perverse filtration induced from the perverse $t$-structure on $f^{-1}(0)$, 
and the middle term as carrying the perverse filtration induced from the perverse $t$-structure on $Y$, then 
perverse $t$-exactness implies that the maps in the sequence respect filtrations.   
(This would not generally
be true if we used the perverse $t$-structure on $f^{-1}(0)$ to define the filtration on the middle term.)

We return the setting of Section \ref{subsec:ageneralsequence}.  
Let $V$ be a vector space.
Suppose that $X$ comes with a  $\U_1$-invariant map $q_{X}  \colon X \to V$, proper over its image.

We have the following related maps:

\begin{enumerate}
\item Compose the projection $X \times \Dol \to X$ with $q_{X}$ to obtain a map $\tilde{q}_1 \colon X \times \Dol \to V$. Since this map is $\U_1$-invariant, it descends to the map 
$$q_1 \colon X \star \Dol \to V.$$ 
\item Restrict $q_1$ to the closed submanifold $X \star \mathbf{n} \subset X \star \Dol$ to obtain
$$q_{1}^{\mathbf{n}} \colon X \star \mathbf{n} \to V.$$
\item Compose the projection $X \times \Dol \to \Dol$ with $q \colon \Dol \to \C^1$ to obtain a map $q_2 \colon X \times \Dol \to \C$. Since this map is $\U_1$-invariant, it descends to a 
map 
$$q_2 \colon X \star \Dol \to \C.$$
\item Since $q_{X} \times q \colon X \times \Dol \to V \times \C$ is invariant for the diagonal $\U_1$-action, it descends to a map 
$$q_{12} \colon X \star \Dol \to V \times \C.$$ 
\end{enumerate} 
The maps $q_{1}^{\mathbf{n}} $ and $q_{12}$ are proper over their images.  The maps $q_1, q_2$ are not proper. 

This allows us to define three filtrations:

\begin{enumerate}
\item $\H^{\bullet}(X\homsep \Q)$ carries the perverse Leray filtration associated to $q_{X}$. 

\item The map $q_{12} \colon X \star \Dol \to B \times \D^1$ 
is proper, and endows  $\H^{\bullet}(X \star \Dol\homsep \Q)$  with a perverse Leray filtration. We can transport this filtration to $\H^{\bullet}(q_2^{-1}(0)\homsep \Q)$ via the pullback by $i  \colon q_2^{-1}(0) \to X\star \Dol$ (the top middle vertical map in Diagram \eqref{global dol key diagram}).

\item 
The map $q_{1}^{\mathbf{n}}$ endows $\H^{\bullet}(X \star \mathbf{n}, \Q)$ with a perverse Leray filtration. 
\end{enumerate}

\begin{proposition} \label{generalsequencepreservesperv}
In the above situation, and assuming in addition that Hypothesis \ref{hyp} holds, 
the map $\H^{\bullet-2}(X \star \mathbf{n}, \Q)\{-1\} \to \H^{\bullet}(X \star \Dol \homsep \Q)$
is compatible with the above perverse Leray filtrations. 
\end{proposition}
\begin{proof}
The shift in perverse Leray filtration $\{-1\}$ arises from our Convention \ref{perverse covention} that although the perverse t-structure is symmetric
around $0$, the perverse Leray filtration for a map between spaces `starts in degree zero'.  The cohomological shift by 2 is inherited from the shift
appearing in the isomorphism $\Phi_{q_2} \Q = \Q_\mathbf{X \star n}[-2]$.  

The map $\H^{\bullet-2}(X \star \mathbf{n}, \Q)\{-1\} \to \H^{\bullet}(X \star \Dol \homsep \Q)$, or rather its
counterpart in the bottom row of Diagram \eqref{global dol key diagram}, was previously described
as being obtained from taking hypercohomology of the map $\Phi_{q_2} \Q \to \Q|_{q_2^{-1}(0)}$.  
We may instead first push forward
by $q_{12}$, or rather the restriction of $q_{12}$ to $q_2^{-1}(0)$, and then take hypercohomology of 
$(q_{12})_* \Phi_{q_2} \Q \to (q_{12})_*(\Q_{q_2^{-1}(0)})$. 
Since $q_{12}$ is proper, we can switch the order in which we take vanishing cycles and pushforwards. 
(Recall that the commutativity of vanishing cycles and pushforward along proper maps is an immediate consequence 
of proper base change \cite[1.3.6.1]{D-vanishing}.) That is, we should study 
\begin{equation} \label{tri:yetanothertriangle} \Phi_{r_2} (q_{12})_*\Q \to ((q_{12})_* \Q)_{r_2^{-1}(0)}  \end{equation}
where the projection $r_2 \colon B \times \C \to \C$ satisfies $q_2 = r_2 \circ q_{12}$. 

It remains only to note that the perverse filtrations respected by the hypercohomology of \eqref{tri:yetanothertriangle} are precisely those 
we have used to define the perverse Leray filtrations on the cohomology groups of interest. 
\end{proof}

\begin{remark}
We will not directly check that $\H^\bullet(X \star \Dol, \Q) \to \H^\bullet(X, \Q)$ respects  the perverse Leray filtrations
we have set up above.  Doing so would involve studying how such filtrations interact with the vertical maps
in the left column of Diagram \eqref{global dol key diagram}.  Because certain non-proper maps would be involved, one would
have to show by hand (e.g. by methods similar to those of Lemma \ref{lem:Dolbassumpone}) some local constancy.  
\end{remark}

\subsection{Compatibility of \eqref{eq:restrictionsfordolb} with perverse Leray filtrations} 
\label{sec:dolbeaultcompat}

In this subsection we are ultimately interested in showing that 
the maps $\kappa_{\mathbf{n}}^*$ and $\kappa^*$ from  \eqref{eq:restrictionsfordolb} preserve the perverse Leray filtrations defined by the vertical maps in Diagrams \eqref{diag:nodeinclusion} 
and \eqref{diag:compinclusion}.  To do this we use the transversality criterion of Corollary \ref{transverse perverse}.  First we give the relevant estimates on 
various images of derivatives.

\begin{lemma} \label{lem:dqdmu} 
At any point $z \in \Dol$, we have $\ker dq + \ker d \dmomabrevU  = T_z \Dol$. 
\end{lemma} 
\begin{proof}
Recall that the basic map $q \colon \Dol \to \D$ is a submersion away from a single point, namely 
$(q \times \dmomabrevU)^{-1}(0 \times 1)$.
Over $(q \times \dmomabrevU)^{-1}(0 \times 1)$, this holds since $\ker dq$ is the entire tangent space. 
At any other point $q \times \dmomabrevU$ is a submersion, and
$\dim (\ker dq + \ker d \dmomabrevU ) =  \dim \ker d q +  \dim  \ker d \dmomabrevU - \dim \ker d(q \times \dmomabrevU) = 
2 + 3 - 1 = 4$. 
\end{proof} 

Consider now $\Dol^{E(\Gamma)}$.

\begin{definition}
For $z \in \Dol^{E(\Gamma)}$, we write
$R(z) \subset E(\Gamma)$ for the subset of edges $e$ with the property that 
$(q_e \times \mu_{\Dol, e})(z) = 0 \times 1$. 
\end{definition} 

\begin{lemma} \label{sums} 
The subset
$dq(T_z  \Dol^{E(\Gamma)}) \subset T^*_{q(z)} \D^{E(\Gamma)}$ 
is $\C^{E(\Gamma \setminus R(z))} \subset   \C^{E(\Gamma)}$. 
\end{lemma} 
\begin{proof}
We can calculate the images of the differential by taking
direct sums.
\end{proof}

\begin{definition}
A subset $R \subset E(\Gamma)$ is said to be independent if its elements are linearly independent in $\H^1(\Gamma\homsep \C)$. Equivalently, $h_1(\Gamma \setminus R) = h_1(\Gamma\homsep \C) - |R|$.
\end{definition}

Recall the subset $\mu^{-1}_{\Gamma}(\eta) \subset \Dol^{E(\Gamma)}$, which, for generic $\eta$, is the space whose free 
$\overline{\U_1^{V(\Gamma)}}$ quotient is $\versD(\Gamma)$. 

\begin{lemma} \label{lem:simpleindependent}
For $\eta$ generic (see Definition \ref{defsimplemomentvalue}), $R(z)$ is independent for all $z \in \mu^{-1}_\Gamma(\eta)$.
\end{lemma} 
\begin{proof}
By definition, $\eta$ is generic if for all $z \in \mu^{-1}_{\Gamma}(\eta)$, the graph $\Gamma \setminus R(z)$ is connected. Thus $\chi(\Gamma \setminus R(z)) = 1 - h_1(\Gamma \setminus R(z))$. 
On the other hand, we have $$\chi(\Gamma \setminus R(z)) =  \chi(\Gamma) + |R(z)|  = 1 - h_1(\Gamma) + |R(z)|.$$ The lemma follows.
\end{proof}

This provides our transversality criterion: 

\begin{corollary} \label{transversality criterion}
If a submanifold $M \subset \C^{E(\Gamma)}$ is transverse to all $\C^{E(\Gamma \setminus R)}$ for independent subsets $R \subset \Gamma$, 
then $M$ is transverse to the map $q_{\operatorname{res}}: \versD(\Gamma) \to  \C^{E(\Gamma \setminus R)} $. 

In particular, this holds for $M = \H_1(\Gamma\homsep \C)$, or any submanifold containing it.  
\end{corollary} 
\begin{proof}
Regarding the general criterion, we 
may check instead for the map $\mu^{-1}_{\Gamma}(\eta) \to  \C^{E(\Gamma \setminus R)}$, for which the result
follows from Lemmas \ref{sums} and \ref{lem:simpleindependent}. 

Now for $M = \H_1(\Gamma\homsep \C)$, we calculate: 
$$\dim \H_1(\Gamma\homsep \C) + \dim \C^{E(\Gamma \setminus R)} - \dim \C^{E(\Gamma)}  = h_1(\Gamma\homsep \C) - |R|.$$ 
Thus it is enough to show that $ \dim ( H_1(\Gamma\homsep \C) \cap \C^{E(\Gamma \setminus R)}) \leq h_1(\Gamma\homsep \C) - |R|$. But this is immediate from the condition that $R$ be independent. 
\end{proof} 

\begin{theorem} \label{thm:versaltosymplectpreservesperverse}
The map $f^*: \H^*(\versD(\Gamma) \homsep \C) \to  \H^*(\Dol(\Gamma) \homsep \C)$
preserves the perverse Leray filtrations associated to the following Cartesian diagram:
\begin{equation} \label{dol vers gamma mod e} 
	\begin{tikzcd} 
		\Dol(\Gamma) \arrow[r, "f"] \arrow[d] & \versD(\Gamma) \arrow[d] \\
		\H_1(\Gamma \homsep \C) \arrow[r] & \C^{E(\Gamma )}
	\end{tikzcd}
\end{equation}
\end{theorem} 
\begin{proof}
Immediate from Corollary \ref{transversality criterion} and Corollary \ref{transverse perverse}. 
\end{proof} 

We now turn to the situation of interest.  
Recall the maps $\kappa$  and $\kappa_{\mathbf{n}}$  which are defined in \eqref{pinchingparttwo} and appear in
Equation  \eqref{pinchingparttwo}. 

\begin{theorem} \label{prop:pinchpervpres}
There is a Cartesian diagram 
\begin{equation} \label{diag:compinclusion}
	\begin{tikzcd}
		\Dol(\Gamma) \arrow[d] \arrow[r, "\kappa"] & \arrow[d, "q_{\operatorname{res}} \times q"]  \Dol(\Gamma / e) \star \Dol \\
		\H_1(\Gamma\homsep \C) \arrow[r] & \H_1(\Gamma / e\homsep \C) \times \C
	\end{tikzcd}
\end{equation}
Here the bottom row is the pushforward along $\Gamma \to \Gamma / e$ on the first factor, and the projection $\gamma \to \langle \gamma, e \rangle$ on the second factor. 

Moreover, the pullback on cohomology $\kappa^*$ preserves the perverse filtrations given by this diagram.  
\end{theorem}
\begin{proof}
We start by constructing the diagram.  By Lemma \ref{lem:uncontract}, we have a Cartesian diagram
\begin{equation} \label{diag:startingdiag}
	\begin{tikzcd}
		\Dol(\Gamma) \arrow[d] \arrow[r] & \arrow[d, "\mu_{\operatorname{res}} \times \mu"]  (\Dol(\Gamma / e) \times \Dol) / \U_1 \\
		\H_1(\Gamma\homsep \C \times \U_1) \arrow[r] & \H_1(\Gamma / e\homsep \C \times \U_1) \times \C \times \U_1.
	\end{tikzcd}
\end{equation}
We have $\H_1(\Gamma, \C \times \U_1) = \H_1(\Gamma, \C) \times \H_1(\Gamma, \U_1)$, and the bottom map is the product of the maps of the same description with $\C$ and $\U_1$ coefficients: 
\begin{equation} \label{eq:bottomarrow}	\H_1(\Gamma\homsep \C)  \to \H_1(\Gamma / e \homsep \C) \times \C \ \ \text{ and } \ \  \H_1(\Gamma\homsep \U_1) \to \H_1(\Gamma / e\homsep \U_1) \times \U_1.\end{equation}
Let $\im(\H_1(\Gamma\homsep \U_1)) \subset (\H_1(\Gamma / e\homsep \U_1) \times \U_1)$ be the image of the right-hand map. We can factor Diagram \eqref{diag:startingdiag} as a pair of Cartesian squares
\begin{equation} \label{diag:middlediag}
	\begin{tikzcd}
		\Dol(\Gamma) \arrow[d] \arrow[r, "\kappa"] & \Dol(\Gamma / e) \star \Dol \arrow[r] \arrow[d] & \arrow[d, "\mu_{\operatorname{res}} \times \mu"]  (\Dol(\Gamma / e) \times \Dol) / \U_1 \\
		\H_1(\Gamma\homsep \C) \times \H_1(\Gamma\homsep \U_1) \arrow[r] &(\H_1(\Gamma / e \homsep \C) \times \C) \times \im(\H_1(\Gamma\homsep \U_1)) \arrow[r] &  \H_1(\Gamma / e \homsep \C \times \U_1) \times \C \times \U_1) .
	\end{tikzcd}
\end{equation}

Now, Diagram \eqref{diag:compinclusion} is obtained by taking the left-hand square of Diagram \eqref{diag:middlediag} and projecting the bottom row of to its complex part.

We can enlarge Diagram \ref{diag:compinclusion} to 

\begin{equation} 
	\begin{tikzcd} \label{diag:enlarged}
		\Dol(\Gamma) \arrow[d] \arrow[r, "\kappa"] & 
		\Dol(\Gamma / e) \star \Dol \arrow[r] \arrow[d] & \versD(\Gamma / e) \star \Dol = \versD(\Gamma) \arrow[d] \\
		\H_1(\Gamma\homsep \C) \arrow[r] & \H_1(\Gamma / e\homsep \C) \times \C \arrow[r] & \C^{E(\Gamma / e)} \times \C = \C^{E(\Gamma)}
	\end{tikzcd}
\end{equation}
The left square is our original Diagram. The right-square is obtained by taking the $\star$-product of the Cartesian square \eqref{dol vers gamma mod e} (with $\Gamma$ replaced by $\Gamma / e$)
with $\Dol$. 

To show that the left square of Diagram \eqref{diag:enlarged} respects perverse filtrations, it suffices to show this for the right square and the total rectangle. Both follow by Corollary \ref{transversality criterion} and Corollary \ref{transverse perverse}. 
\end{proof} 

\begin{theorem} \label{deletepreservesperverse} 
There is a Cartesian diagram 
\begin{equation} \label{diag:nodeinclusion}
	\begin{tikzcd}
		\Dol(\Gamma \setminus e) \arrow[d] \arrow[r, "\kappa_{\mathbf{n}}"] & \arrow[d, "q_{\operatorname{res}}"]  \Dol(\Gamma / e) \star \mathbf{n} 
		\\
		\H_1(\Gamma \setminus e\homsep \C) \arrow[r] & \H_1(\Gamma / e\homsep \C)
	\end{tikzcd}
\end{equation}
Here the bottom arrow is the pushforward along the composition $\Gamma \setminus e \to \Gamma \to \Gamma / e$. 

Moreover, the pullback on cohomology $\kappa_{\mathbf{n}}^*$ preserves the perverse filtrations given by this diagram.  
\end{theorem}
\begin{proof}
We start by constructing the diagram.  By Lemma \ref{lem:deletionfromcontraction}, we have a Cartesian diagram
\begin{equation} \label{diag:startingdiag1}
	\begin{tikzcd}
		\Dol(\Gamma \setminus e) \arrow[d] \arrow[r] &
		\arrow[d, "\mu_{\operatorname{res}} \times \mu"]  (\Dol(\Gamma / e) \times \mathbf{n}) / \U_1 \\
		\H_1(\Gamma \setminus e \homsep \C \times \U_1) \arrow[r] &
		\H_1(\Gamma / e\homsep \C \times \U_1) \times \C \times \U_1.
	\end{tikzcd}
\end{equation}
As in the proof of Theorem \ref{prop:pinchpervpres}, this diagram factors as a pair of Cartesian squares
\begin{equation} \label{diag:middlediag2}
	\hspace{-1.5cm}
	\begin{tikzcd}
		\Dol(\Gamma \setminus e) \arrow[d] \arrow[r, "\kappa_{\mathbf{n}}"] & \Dol(\Gamma / e) \star \mathbf{n} \arrow[r] \arrow[d] & 
		\arrow[d, "\mu_{\operatorname{res}} \times \mu"]  (\Dol(\Gamma / e) \times \mathbf{n}) / \U_1 \\
		\H_1(\Gamma \setminus e \homsep \C) \times \H_1(\Gamma \setminus e \homsep \U_1) \arrow[r] &
		\H_1(\Gamma / e \homsep \C) \times \C \times \im(\H_1(\Gamma \setminus e \homsep \U_1) \arrow[r] &
		\H_1(\Gamma / e \homsep \C \times \U_1) \times \C \times \U_1).
	\end{tikzcd}
\end{equation}

Diagram \eqref{diag:nodeinclusion} is obtained by taking the left-hand square of Diagram \eqref{diag:middlediag2} and projecting the bottom row of to its complex part. 

We can enlarge Diagram \eqref{diag:nodeinclusion} as follows.
\begin{equation}
	\begin{tikzcd} \label{diag:enlarged2}
		\Dol(\Gamma \setminus e) \arrow[d] \arrow[r, "\kappa_{\mathbf{n}}"] & \Dol(\Gamma / e) \star \mathbf{n} \arrow[r] \arrow[d] & \versD(\Gamma / e) \star \mathbf{n} = \versD(\Gamma \setminus e) \arrow[d] \\
		\H_1(\Gamma \setminus e\homsep \C) \arrow[r] & \H_1(\Gamma / e\homsep \C) \arrow[r] & \C^{E(\Gamma / e)} = \C^{E(\Gamma \backslash e)}
	\end{tikzcd}
\end{equation}
The left hand square is our original \eqref{diag:nodeinclusion}. The right-hand square is obtained by taking $\star \mathbf{n}$ with \eqref{dol vers gamma mod e} (with $\Gamma$ replaced by $\Gamma / e$).

To show that the left square of Diagram \eqref{diag:enlarged2} respects perverse filtrations, it suffices to show this for the right square and the total rectangle.  
Both follow by Corollary \ref{transversality criterion} and Corollary \ref{transverse perverse}.
\end{proof}

\subsection{The perverse Leray filtration and the $\CKS$ filtration}

\begin{theorem} \label{thm:abstisom}
There  is an isomorphism $\H^{\bullet}(\Dol(\Gamma)\homsep \Q) \cong \H^{\bullet}(\CKS\homsep \Q)$ 
identifying the perverse filtration with the $\CKS$ filtration.
\end{theorem}
\begin{proof}
By Corollary \ref{centralandversalrestrictionisoms} and Theorem \ref{thm:versaltosymplectpreservesperverse}, it suffices to study 
the perverse filtration on the central fiber induced
by the map $q_{\res} : \versD(\Gamma, \eta) \to \C^{\ed}$.  

We have shown in Theorem \ref{thm:projective} that $q_{\res}$ is projective 
in a neighborhood of the central fiber; we henceforth restrict to this neighborhood.  Recalling that 
$\versD(\Gamma, \eta)$ is nonsingular, we may therefore
apply the decomposition theorem of \cite{BBD}  to conclude 
$q_{\res*} \Q$ is a direct sum of semisimple perverse sheaves.  

Let $q_{\res}^\circ$ be the restriction of $q_{\res}$ to the complement of
the coordinate hyperplanes.  One summand of $q_{\res *}\Q$ is therefore 
$\bigoplus_j IC (R^j q_{\res*}^\circ \Q)$.  Of these, the summands with $j \le k$ contribute
to the $k$'th step of the perverse Leray filtration. 

Comparing Proposition \ref{prop:monodromycomp} to \cite[Eq. 3.7]{MSV}, we see that the local systems
$R^j  q^\circ_{\res*} \Q$ considered here are isomorphic to those 
called $\bigwedge^j R^1 \pi_* \Q|_{B_{reg}}$ in \cite{MSV}.  (See Remark \ref{rem: microlocal}
for some further discussion about the similarities and differences between $q_{\res}$ and the relative compactified Jacobian
for the versal family of a nodal curve with dual graph $\Gamma$.) 	

In general, there is a formula \cite{CKS} for the stalks of the intermediate extension 
of a local system across a normal crossing divisor. 	
\cite{MSV} explicitly computed in the case of the local system at hand; the result \cite[Lem. 3.6]{MSV} 
was that 
$IC(R^j q^\circ_{\res*} \Q)_0$ is computed by the complex we have here called 
$\CKS^{\bullet}$.  

We have seen there is a summand
$\H^{\bullet}(\CKS(\Gamma)\homsep \Q) \subset \H^{\bullet}(\Dol(\Gamma)\homsep \Q)$, such that the perverse
filtration on the later restricts to the $\CKS$ filtration on the former.  It remains
to show this inclusion is an equality.  
By the argument of  \cite[Prop. 15]{MS}, 
it suffices to check the equality of
weight polynomials.  The calculation for $\CKS(\Gamma)$ is carried out in 
\cite[Cor. 3.8]{MSV}, and for the central fiber of $\Dol(\Gamma)$ in Theorem \ref{thm:dolcfclass} above. 
The results agree: each is $t^{2 h_1(\Gamma)}$ times the number of spanning trees of $\Gamma$.
\end{proof}

\begin{remark}
We will later show in Theorem \ref{thm:basicnaht} that $\Bet(\Gamma)$ retracts to $\Dol(\Gamma)$,
hence in particular has the same cohomology.  We also know that $\H^\bullet(\Bet(\Gamma),\Q) \cong \H^\bullet(\CKS(\Gamma),\Q)$ from Theorem \ref{thm:quasiisom}.  It follows that $ \H^\bullet(\CKS(\Gamma),\Q)$ and $\H^{\bullet}(\Dol(\Gamma)\homsep \Q)$ 
have the same total dimension, hence that the inclusion 
$\H^{\bullet}(\CKS(\Gamma)\homsep \Q) \subset \H^{\bullet}(\Dol(\Gamma)\homsep \Q)$ must be an isomorphism. 
This is an independent argument from the weight polynomial one given above. 
\end{remark}

We record the following special case. 
\begin{corollary} \label{cor:pervbasiccomp}
The perverse Leray filtration on $\H^{\bullet}(\Dol, \Q)$ with respect to the map $q: \Dol \to \D$ is the filtration by cohomological degree. In other words, $$P_0 \H^{\bullet}(\Dol,\Q) = \H^0(\Dol, \Q),$$ 
$$P_1 \H^{\bullet}(\Dol,\Q) = \H^{\leq 1}(\Dol, \Q),$$ 
$$P_2 \H^{\bullet}(\Dol,\Q) = \H^{\leq 2}(\Dol, \Q).$$ 
\end{corollary}
\begin{proof}
One can of course verify this by direct geometric arguments. Here we simply note that since the cohomology of $\Dol$ has rank at most one in any given degree, it is enough to determine the associated graded of the perverse Leray filtration. This in turn is computed by the case $\Gamma = \boing$ of Theorem \ref{thm:abstisom}. 
\end{proof}

\begin{remark}
Comparing the formula for weight polynomials in \cite[Cor. 3.8]{MSV} with the 
formula in Remark \ref{rem:dolafclass}, we see that in fact every summand of 
$(q_{\res} : \versD(\Gamma, \eta) \to \C^{\ed})_* \Q$
has full 
support; in particular,  
${}^p R^j q_{res*} \Q = IC(R^j q^{\circ}_{\res*} \Q)$.

By contrast, $(q: \Dol(\Gamma, \eta) \to H_1(\Gamma, \C))_* \Q$ can have summands supported in positive codimension. 
\end{remark}

\begin{remark} \label{rem:dolmeaning}
By combining Theorem \ref{thm:abstisom} with the results of \cite{MSV}, we
establish the isomorphism $H^*(\Dol(\Gamma_\Sigma)) \cong D(\Sigma)$ asserted in Remark \ref{rem: microlocal}: both sides are computed by
$\CKS^\bullet(\Gamma)$.
\end{remark}

\subsection{Compatibility of $a^{\Dol}_e$ with perverse filtration} 

Following convention \eqref{def:pervleraydolb}, each term of the Dolbeault deletion-contraction sequence (dashed sequence in Diagram \eqref{pinchingparttwo}) carries a perverse Leray filtration.
We will translate the filtration on the left-hand term by one step - this is analogous to the Tate twist occuring in the $\Bet$-DCS. We denote the resulting filtered vector space by $\H^{\bullet-2}(\YDol \star \bS_\Dol\homsep \Q)\{-1\}$, so that $P_k \H^{\bullet-2}(\YDol \star \bS_\Dol\homsep \Q)\{-1\}  = P_{k-1}\H^{\bullet-2}(\YDol \star \bS_\Dol\homsep \Q)$. That is, we consider: 
\begin{equation} \label{DolbeaultLES}
\begin{tikzcd}
	\arrow[r, dashed, "c^{\Dol}_e"] &  \H^{\bullet-2}(\Dol(\Gamma \setminus e)\homsep \Q) \{-1\} \arrow[r, dashed, "a^{\Dol}_e"]
	& \H^{\bullet}(\Dol(\Gamma)\homsep \Q) \arrow[r, dashed, "b^{\Dol}_e"] 
	& \H^{\bullet}(\Dol(\Gamma / e)\homsep \Q) \arrow[r, dashed, "c^{\Dol}_e"] & \ \\
\end{tikzcd}
\end{equation}

\begin{proposition} \label{prop: deletion perverse compatible}
The map $a^{\Dol}_e$ of \eqref{DolbeaultLES} is compatible with the perverse Leray filtration.  
\end{proposition} 
\begin{proof}
By combining Proposition \ref{generalsequencepreservesperv}, Theorem \ref{prop:pinchpervpres}, and Theorem \ref{deletepreservesperverse}, 
we learn that  $a^{\Dol}_e$ of Equation \eqref{DolbeaultLES} (not necessarily strictly)
preserves the perverse filtration. $\square$
\end{proof} 

\begin{remark}
We will eventually show strict compatibility of all maps in \eqref{DolbeaultLES} with perverse Leray filtrations, but 
only by {\em first} proving that P=W compatibly with an intertwining of the deletion-contraction sequences 
(which in turn will require having first proven Proposition \ref{prop: deletion perverse compatible}).  
Recall by contrast that strict compatibility of \eqref{eq:graphbettiLES} with weight filtrations
followed from from general considerations.
\end{remark}

\begin{corollary} \label{cor:delbounded}
The Dolbeault deletion filtration is bounded by the perverse Leray filtration.
\end{corollary}
\begin{proof}
Using Proposition \ref{prop: deletion perverse compatible}, we find
$\delfilt_k \H^{\bullet}(\Dol(\Gamma)\homsep \Q) \subset P_{k}\H^{\bullet}(\Dol(\Gamma)\homsep \Q)$,
by the same argument as we used in Theorem \ref{deletionisweight} to establish \eqref{eq: del in weight}.   
\end{proof}

\section{Comparisons} \label{sec:hodge}

\subsection{``Hodge'' correspondence} \label{subsec:basichodge} \label{subsec:hodgecorrespondence}

We will construct a homotopy equivalence $\Dol \to \Bet$, which we use to induce a homotopy equivalence  $\Dol(\Gamma) \to \Bet(\Gamma)$. 

\begin{lemma} \label{lem:tateNAHT}
Let $\momenthomomorphism \colon \uutimesC \cong \C^* \times \R$ be the group isomorphism $(z,  e^{2\pi i \theta}) \to ( e^{2\pi i \theta + Im(z)}, Re(z)   )$. 
There is a (non-unique!) $\uu$-equivariant $\mathcal{C}^{\infty}$ embedding $\basichodge \colon \Dol \to \Bet$ such that the following diagram commutes.
\[ \begin{CD}
	\Dol @> \basichodge >> \Bet \\
	@ V \mu^{\uu}_{\Dol} \times \dmomD  VV   @ VV \mu_{\Bet}^{\C^*} \times \mu_{\Bet}^{\R} V\\
	\uutimesD @> \momenthomomorphism >> \C^* \times \R \\
\end{CD} \]
\end{lemma}
\begin{proof}
We write $0_\Dol = 1 \times 0 \in \uutimesD$ and $0_\Bet = 1 \times 0 \in \C^* \times \R$.  
Evidently $\momenthomomorphism(0_\Dol) = 0_\Bet$.  

The maps $\mu_{\Bet}^{\C^*} \times \mu_{\Bet}^{\R}$ and $\mu^{\uu}_{\Dol} \times \dmomD $ define principal $\uu$-bundles  
$\mathcal{P}_{\Bet}, \mathcal{P}_{\Dol}$ away from the point $0_\Dol$ and $0_\Bet$ (Lemma \ref{lem:betprincipbundle} and Proposition \ref{pinchdeg}).  
The restrictions of these bundles to $\uutimesD \setminus 0_\Dol$ are classified by their Chern characters $c_1(\mathcal{P}_i) \in \H^2(\mathbb{D} \times \uu \setminus 0_\Dol \homsep \mathbb{Z}).$ 
The group $\H^2(\uutimesD \setminus 0_\Dol\homsep \Z)$ is spanned by a small sphere around $0_\Dol$. 
It follows that a $\U_1$-equivariant isomorphism over a small disk around $0_\Dol \to 0_\Bet$ can be extended to such 
an isomorphism over all of $\uutimesD \setminus 0_\Dol$. We must show that some such isomorphism extends over $0_\Dol$.

Both spaces $\Bet$ and $\Dol$ have a single $\uu$ fixed point, with image $0_\Bet$ and $0_\Dol$ respectively.  By the differentiable slice theorem, 
the fibration near a small neighborhood of $0_\Bet$ (or $0_\Dol$) is equivariantly diffeomorphic to that 
given by a linear circle action on the unit ball in $\R^4$.  We have seen that our actions have no nontrivial stabilizers
away from the fixed point; it follows from this that the circle acts by the identity character, its inverse, or a sum of these. 
In any case, in the coordinates of our descriptions of both of these spaces, the $\U_1$ action around the fixed 
point was explicitly given
in complex coordinates as  $(x, y) \mapsto (\tau x, \tau^{-1} y)$.  

The above argument suffices to establish that {\em some} map lifts to an equivariant embedding.  To see that this is true for $\momenthomomorphism$ specifically (or indeed, any homotopy equivalence carrying $0_\Dol \to 0_\Bet$), it suffices to note that
if $\R^4 \to \R^3$ is the Hopf fibration, then  
any diffeomorphism $\R^3 \to \R^3$ fixing the origin lifts to an equivariant diffeomorphism of Hopf fibrations.  This again follows
from local linearity at the fixed point. 
\end{proof}

Let $\Bet^{<}$ be the image of $\basichodge$; alternatively, $\Bet^{<}$ is the preimage of $\momenthomomorphism(\uutimesD)$.
\begin{lemma} \label{lem:zbethomretract}
$\Bet^{<}$ is a diffeomorphic homotopy retract of $\Bet.$ 
\end{lemma}
\begin{proof}
Let $r\colon \C \times I \to \C$ be a linear retraction of $\C$ onto $\D$. As in the proof of Corollary \ref{cor:dolretract}, this induces a diffeomorphism and homotopy retract $\Bet \to \Bet^{<}$.
\end{proof}

\begin{lemma}
$\bS_\Dol = \basichodge^{-1} (\bS_\Bet)$
\end{lemma}
\begin{proof}
Each is the moment preimage of a half line, and we have chosen $\momenthomomorphism$ to identify these half-lines. 
\end{proof}

We turn to the case of $\Dol(\Gamma)$ and $\Bet(\Gamma)$.  We restore now the 
moment map parameter $\eta$ in our notation, since it plays a priori different roles for $\Bet$ and $\Dol$.

We have defined $\Bet(\Gamma)$ as a complex algebraic (GIT) quotient; by the Kempf-Ness theorem we can instead understand it as a symplectic
reduction, as the following proposition shows. 
\begin{proposition} \label{prop:Bettikempfness} There is a diffeomorphism
\[ \Bet^{\uu, \R \times \C^*}(\Gamma, \eta) \to \Bet^{\C^*, \C^*}(\Gamma, \eta) . \]
\end{proposition}
\begin{proof}
In the construction of $\Bet(\Gamma, \eta)$, 
we took a $(\C^*)^{\ver}$ quotient of $\mu_{\Gamma}^{-1}(\eta)$.  In fact we had the structure of a $\uu^{\ver} \subset (\C^*)^{\ver}$ action 
on the complex manifold $\mu_{\Gamma}^{-1}(\eta) \subset \Bet^{\ed}$ with $\R^{\ver}$-valued moment map, induced from the 
$\uu \subset \C^*$ action on $\Bet$ with moment map $|x|^2 - |y|^2$. 
By the Kempf-Ness theorem \cite{KN}, we can replace the $(\C^*)^{\ver}$ quotient of $\mu_{\Gamma}^{-1}(\eta)$ by the symplectic reduction by 
$\uu^{\ver}$. (This is a particularly simple application of the Kempf-Ness theorem, since the $(\C^*)^{V(\Gamma)}$ action is free on $\mu_{\Gamma}^{-1}(\eta)$ and all orbits are closed.) The resulting diffeomorphism takes a point $z \in \Bet^{\uu, \R \times \C^*}(\Gamma, \eta)$ to the $(\C^*)^{\ver}$-orbit of its image.
\end{proof}

The virtue of the symplectic reduction picture of the Betti space is that it is more readily comparable to the Dolbeault space.  
More precisely we would like to compare a retracted version: 

\begin{definition}
$ \Bet^{<}(\Gamma, \eta) \colonequals (\Bet^{<})^{\uu, \R \times \C^*}(\Gamma, \eta)$. 
\end{definition}

\begin{theoremdef} \label{thm:basicnaht}
There is a diffeomorphism $\frak{F}_{\Gamma}\colon \Dol(\Gamma, \eta) \to \Bet^{<}(\Gamma, \eta)$ making the following diagram commutative.
\[ \begin{CD}
	\Dol(\Gamma, \eta) @> \frak{F}_{\Gamma} >> \Bet^{<}(\Gamma, \eta) \\
	@ V \mu^{\uu}_{\Dol, \operatorname{res}} \times q_{\operatorname{res}}  VV   @ VV \mu_{\Bet, \operatorname{res}}^{\C^*} \times \mu_{\Bet, \operatorname{res}}^{\R} V\\
	\H_1(\Gamma, \uutimesC)_\eta @> \momenthomomorphism_{\Gamma} >> \H_1(\Gamma, \C^* \times \R)_\eta \\
\end{CD} \]
\end{theoremdef}
Here  $\momenthomomorphism_{\Gamma}$ is the isomorphism induced by $\momenthomomorphism$.
\begin{proof}
Let $\basichodge\colon  \Dol \to \Bet^{<}$ be the diffeomorphism in Lemma \ref{lem:tateNAHT}. By construction, it is $\uu$-equivariant and covers a group isomorphism $\momenthomomorphism\colon \uutimesC \to \R \times \C^*$. Hence it induces a diffeomorphism 
\[ \frak{F}_{\Gamma} \colon \Dol^{(\uu, \uutimesC)}(\Gamma, \eta) \cong (\Bet^{<})^{\uu, \R \times \C^*}(\Gamma, \eta) . \]
covering the group isomorphism $\momenthomomorphism\colon \H_1(\Gamma, \uutimesC) \to \H_1(\Gamma, \R \times \C^*).$
\end{proof}

We see from Proposition \ref{prop: h1h1 structure} that $\Dol(\Gamma, \eta)$ is an $(\H^1(\Gamma ;\U_1), \H_1(\Gamma ;\uutimesC))$-space and $\Bet^{<}(\Gamma, \eta)$ is an $(\H^1(\Gamma ;\U_1), \H_1(\Gamma ;\R \times \C^*))$-space. The isomorphism $\momenthomomorphism\colon \uutimesC \to \R \times \C^*$ allows us to view both as $(\H^1(\Gamma ;\U_1), \H_1(\Gamma ;\uutimesC))$-spaces. 
\begin{lemma} \label{lem:HodgerespGMstruct}
$\frak{F}_{\Gamma} \colon \Dol(\Gamma, \eta) \xrightarrow{\sim} \Bet^{<}(\Gamma, \eta)$ is an isomorphism of $(\H^1(\Gamma ;\U_1), \H_1(\Gamma ;\uutimesC))$-spaces.
\end{lemma}
\begin{proof}
This follows from $\U_1$ equivariance of Lemma \ref{lem:tateNAHT}. 
\end{proof} 

\begin{remarque}
Note that $\basichodge$ was not unique, and thus nor is $\frak{F}_{\Gamma}$. However, the group of smooth maps $\H_1(\Gamma\homsep \uutimesC) \to \H^1(\Gamma\homsep \U_1)$ (not respecting any group structure) acts transitively on the set of choices. We will not need this fact in what follows.
\end{remarque}

\begin{proposition} \label{prop:betretract}
The retraction $\Bet \to \Bet^<$ in Lemma \ref{lem:zbethomretract} descends to a diffeomorphism and homotopy retraction from $\Bet(\Gamma, \eta)$ to $\Bet^{<}(\Gamma, \eta)$ . 
\end{proposition}  

\begin{proof}
We have the following diagram. 
\begin{equation}
	\begin{tikzcd}
		\mu_{\Gamma}^{-1}(\eta) / \U_1^{\ver}  & 
		\mu_{\Gamma}^{-1}(\eta)  \arrow[r, hook] \arrow[d] \arrow[l] & 
		\Bet^{\ed} \arrow[d] \\
		& 
		d_{\Gamma}^{-1}(\eta)  \arrow[r, hook] &
		C_1(\Gamma\homsep \uutimesC) = (\uutimesC)^{\ed}
	\end{tikzcd}
\end{equation}

Pick a diffeomorphism $\psi \colon \C^{\ed} \to \D^{\ed}$ which preserves $\R$-lines through the origin. Linear interpolation between $z$ and $\psi(z)$ defines a deformation retraction $r  \colon \C^{\ed} \times [0,1] \to \D^{\ed}$. Let $R \colon (\uutimesC)^{\ed} \times  [0,1] \to (\uutimesC)^{\ed}$ be the induced deformation retraction, which is constant in the $(\U_1)^{\ed}$ factor. Since $R$ preserves lines in $\C^{\ed}$, $R$ preserves $d_{\Gamma}^{-1}(\eta)$ and its stratification by coordinate hyperplanes.  As in Lemma \ref{lem:zbethomretract}, this induces a $\U_1^{\ed}$-equivariant retraction of $\mu_{\Gamma}^{-1}(\eta)$ onto $ \mu_{\Gamma}^{-1}(\eta) \cap (\Bet^{<})^{\ed}$. Passing to $\U_1$-quotients, we obtain the desired retraction. 
\end{proof}

Combining Proposition \ref{prop:betretract}, Theorem/Definition \ref{thm:basicnaht} and Lemma \ref{lem:dolbretract}, we see that $\Bet(\Gamma, \eta)$ also retracts onto $q_{\Gamma}^{-1}(0)$, viewed as a subset of $\Bet^<(\Gamma, \eta)$ via Theorem/Definition \ref{thm:basicnaht}.

\subsection{$P=W$} \label{sec:PW}

\begin{proposition} \label{prop: leftdoltobetti}
The following diagram commutes. 
\begin{equation}
	\label{leftdoltobetti}
	\begin{tikzcd} 
		\Bet(\Gamma \setminus e)  \arrow[r] &  \Bet(\Gamma) \\
		\Dol(\Gamma \setminus e)  \arrow[r] \arrow[u, "\frak{F}_{\Gamma \setminus e} "] &  
		\Dol(\Gamma) \arrow[u, "\frak{F}_{\Gamma}"]  
	\end{tikzcd}
\end{equation}
\end{proposition} 
\begin{proof}
The horizontal arrows can be defined in both cases by applying Lemma \ref{lem:deletioninclusionfromGMpoint} to the edge $e$. This expresses $\Bet(\Gamma \setminus e)$ and $\Dol(\Gamma \setminus e)$ as subquotients of  $\Dol^{\ed}$ and $\Bet^{\ed}$ respectively. The maps $\frak{F}_{\Gamma}$ and $\frak{F}_{\Gamma \setminus e}$ are induced from 
the product map $\frak{F}^{\ed} : \Dol^{\ed} \to \Bet^{\ed}$.
\end{proof} 

\begin{remark} \label{rem:smalldiagcommutessodolbcommutes2}
Using Proposition \ref{prop: leftdoltobetti} we could instead prove the commutativity 
of the $\Dol$ deletion maps $a^{\Dol}_e$
(Lemma \ref{lem:smalldiagcommutessodolbcommutes}) by importing 
the result from the corresponding result on the maps $a^{\Bet}_e$ 
(Lemma \ref{lem:gysincommutes}), or conversely. 
\end{remark}

\begin{corollary} \label{prop:delsareequal}
The map $\frak{F}_\Gamma^*$ identifies the Betti and Dolbeault deletion filtrations.
\end{corollary}
\begin{proof} 
The deletion filtrations are determined by Gysin maps associated to the horizontal
arrows of Diagram \eqref{leftdoltobetti}. 
\end{proof}

\begin{proposition} \label{prop:dolbdelequalspervLer}
$\delfilt_k \H^{\bullet}(\Dol(\Gamma)\homsep \Q) = P_{k}\H^{\bullet}(\Dol(\Gamma)\homsep \Q).$ In other words, the Dolbeault deletion filtration equals the perverse Leray filtration.
\end{proposition}
\begin{proof}
We know that the Dolbeault deletion filtration is bounded by the perverse Leray filtration by Corollary \ref{cor:delbounded}. Thus it is enough to show that the filtrations are abstractly isomorphic, i.e. that there exists an isomorphism of vector spaces taking the deletion filtration to the perverse Leray filtration. We produce this isomorphism via $\frak{F}$. 

Namely, by Corollary \ref{prop:delsareequal}, the Dolbeault deletion filtration is isomorphic to the Betti deletion filtration. In turn, the Betti deletion filtration is isomorphic to the $\CKS$ filtration by Proposition \ref{cksdelequalsbettidel}. Finally, Theorem \ref{thm:abstisom} shows that the $\CKS$ and perverse Leray filtrations are isomorphic. \end{proof}

\begin{remark}
To show Proposition \ref{prop:dolbdelequalspervLer} 
without appeal to the $\Dol \subset \Bet$ comparison, we could argue by induction if 
we fixed an isomorphism between the 
cohomology of the $\Dol$-space and the $\CKS$ complex which intertwines the deletion
maps.  This may be possible, since both the $\CKS$ complex and the deletion map are built from
nearby-vanishing cycle operations, but we have not done it here. 
\end{remark}

\begin{theorem}[$P=W$] \label{thm:wequals2p}
The map $\frak{F}_\Gamma^*$ identifies $W_{2k} \H^{\bullet}(\Bet(\Gamma)\homsep\Q)$ with $P_{k}\H^{\bullet}(\Dol(\Gamma)\homsep\Q)$. 
\end{theorem}
\begin{proof}
We have identified both filtrations with the deletion filtrations on the respective spaces in Theorem \ref{deletionisweight} and Proposition \ref{prop:dolbdelequalspervLer}. The result thus follows from Corollary \ref{prop:delsareequal}.
\end{proof}

\subsection{Intertwining of deletion-contraction sequences} \label{sec:hodgedeletionintertwines}  

\begin{theorem} \label{Hodgeintertwinesdelcon}
The following diagram commutes:
\begin{equation} \label{bigdiag}
	\begin{tikzcd} 
		\ \arrow[r, "c_e^{\Bet}"] &  \ \H^{\bullet-2}(\Bet(\Gamma \setminus e) \homsep \Q)(-1) \arrow[d, "\frak{F}_{\Gamma \setminus e}^*"]
		\arrow[r, "a_e^{\Bet}"]  &  \H^{\bullet}(\Bet(\Gamma)\homsep \Q) \arrow[d, "\frak{F}_{\Gamma}^*"] \arrow[r, "b_e^{\Bet}"]   &  \H^{\bullet}( \Bet(\Gamma / e)\homsep \Q) \arrow[r, "c_e^{\Bet} "]  \arrow[d, "\frak{F}_{\Gamma / e}^*"]
		& \  \\
		\ \arrow[r, "c_e^{\Dol}"]  &  \H^{\bullet-2}(\Dol(\Gamma \setminus e) \homsep \Q) \{-1\} \arrow[r, "a_e^{\Dol}"] 
		& \H^{\bullet}(\Dol(\Gamma) \homsep \Q) \arrow[r, "b_e^{\Dol}"] 
		& \H^{\bullet}(\Dol(\Gamma / e) \homsep \Q) \arrow[r, "c_e^{\Dol} "]  & \ \\
	\end{tikzcd}
\end{equation}
\end{theorem} 
\begin{proof}
We write 
$t^{\Dol}_e$ for the map 
given by the composition $\Dol(\Gamma) \to \H_1(\Gamma, \uutimesC) \to \uutimesC$, 
where the second map extracts the coefficient of the edge $e$. 
Similarly have a map 
$t^{\Bet}_e$ from a composition 
$\Bet(\Gamma) \to \H_1(\Gamma, \C^* \times \R) \to \C^* \times \R$.  
Because  $\frak{F}_{\Gamma}$ covers the map $\momenthomomorphism_{\Gamma}: \H_1(\Gamma, \uutimesC) \to \H_1(\Gamma, \C^* \times \R)$, 
it follows that we may fill in the dashed arrow below: 

\begin{equation}
	\begin{tikzcd} \label{eq:minidiagdolbtobetti}
		\Bet(\Gamma / e) = \Bet(\Gamma / e) \star_{\Gm} (\Bet \setminus \bS_\Bet) \arrow[r, equals]  & 
		\{t^{\Bet}_e \notin 1 \times \R^{<0} \} \arrow[r, hook] &
		\Bet(\Gamma / e) \star_{\Gm} \Bet = \Bet(\Gamma) \\
		\Dol(\Gamma / e) \UUCstar (\Dol \setminus \bS_\Dol) \arrow[r, equals]  & 
		\{t^{\Dol}_e \neq 1 \times 0 \}  \arrow[r, hook] \arrow[u, hook, dashed]
		&
		\Dol(\Gamma / e) \UUCstar\Dol  = \Dol(\Gamma) \arrow[u, hook, "\frak{F}_{\Gamma}"]  
	\end{tikzcd}
\end{equation}

From \eqref{eq:minidiagdolbtobetti} we obtain a map of long exact sequences of pairs
(taking notations from Section \ref{sec:LESpair} and using in particular Lemma \ref{lem:bridgetothecoastlemma}),  
which we compose with Diagram \eqref{pairstodolb}, to obtain: 
\begin{equation} \label{diag:monsterdiag}
	\hspace{-10mm}
	\begin{tikzcd}
		\H^{\bullet-2}(\Bet(\Gamma \setminus e)\homsep \Q)  \arrow[d, "(\frak{F}_{\Gamma}|_{\Dol(\Gamma / e) \UUCstar \bS_\Dol})^*"]  \arrow[r, "a^{\Bet}_e"]  & \H^{\bullet}(\Bet(\Gamma)\homsep \Q)  \arrow[d, "\frak{F}_{\Gamma}^*"] \arrow[r, "b^{\Bet}_e"] & \H^{\bullet}(\Bet(\Gamma / e), \Q) \arrow[d, "(\frak{F}_{\Gamma, \operatorname{res}})^*"] \arrow[r, "c^{\Bet}_e"] & \
		\\	\H^{\bullet-2}(\Dol(\Gamma / e) \UUCstar \bS_\Dol \homsep \Q) \arrow[r, "A^{\Dol}_e"] \arrow[d, equals] 
		& \H^{\bullet}(\Dol(\Gamma)\homsep \Q) \arrow[r, "B^{\Dol}_e"] 
		& \H^{\bullet}(\Dol(\Gamma / e) \UUCstar (\Dol \setminus \bS_\Dol) \homsep \Q) \arrow[r, "C^{\Dol}_e"]  \arrow[d, " i_\epsilon^* (\kappa_\circ^*)^{-1}"] & \ \\
		\H^{\bullet-2}(\Dol(\Gamma \setminus e),\Q) \arrow[r, "a^{\Dol}_e"] 
		& \H^{\bullet}(\Dol(\Gamma),\Q) \arrow[u, equals] \arrow[r, "b^{\Dol}_e"] 
		&   \H^{\bullet}(\Dol(\Gamma / e), \Q) \arrow[r, "c^\Dol_e"]  & \ \\
	\end{tikzcd} 
\end{equation}
Proposition \ref{prop: leftdoltobetti} implies that the composition of the left vertical arrows
is simply $\frak{F}_{\Gamma \setminus e}^*$.  To complete the 
proof of the theorem, it remains only to show that 
$\frak{F}_{\Gamma / e}^* =  i_\epsilon^* (\kappa_\circ^*)^{-1} \frak{F}_{\Gamma, \operatorname{res}}^*$. 

We will consider the following diagram. 
\begin{equation} \label{diagextended}
	\begin{tikzcd}
		\Bet(\Gamma / e) \arrow[rr, equals]  
		& 
		& \Bet(\Gamma / e) \star_{\Gm} (\Bet \setminus \bS_\Bet) \arrow[r, "\kappa' "]
		&  \Bet(\Gamma / e) \star_{\U_1} (\Bet \setminus \bS_\Bet)
		\\
		\Dol(\Gamma / e) \arrow[u, "\frak{F}_{\Gamma /e }"] \arrow[r, "i_\epsilon"] 
		& \Dol(\Gamma / e)  \star_{\U_1} (\Dol \setminus \bS_\Dol) 
		& \arrow[l, "\kappa_\circ"] \Dol(\Gamma / e) \UUCstar (\Dol \setminus \bS_\Dol) \arrow[u, "\frak{F}_{\Gamma}"]  \arrow[r, "\kappa_\circ"]
		&    \Dol(\Gamma / e)  \star_{\U_1} (\Dol \setminus \bS_\Dol)  \arrow[u, "\frak{F}_{\times, \operatorname{res}}"]
	\end{tikzcd}
\end{equation}
We have not yet introduced the maps $\kappa'$ and $\frak{F}_{\times, \operatorname{res}}$
or explained how we can $\star_{\U_1}$ on the upper left corner; we do so now. 
Recall from Proposition \ref{prop:betproperties} \eqref{other structure} that 
$\Bet$ can also be regarded as a $(\U_1, \C^* \times \R =
\uutimesC)$ space; the $\Bet(\Gamma)$ etc. spaces have corresponding structures. 
Moreover, the $\star_{(\C^*, \C^*)}$ and $\star_{(\U_1, \uutimesC)}$ type convolutions
agree for these spaces, by Kempf-Ness.  In particular, 
$$\Bet(\Gamma / e) = \Bet(\Gamma / e) \star_{\Gm} (\Bet \setminus \bS_\Bet) \stackrel{\phi}{=} 
\Bet(\Gamma / e) \star_{\uutimesC} (\Bet \setminus \bS_\Bet)$$  
The map $\kappa'$ is defined to be
the natural map from a $\star_{\uutimesC}$ product to a $\star_{\U_1}$ product. 
The map $\frak{F}_{\times, \operatorname{res}}$ the restriction of the pullback map
obtained from the $\star_{\U_1}$ version of \eqref{eq:minidiagdolbtobetti}. 
It is clear from the definitions that the left hand square of \eqref{diagextended} 
commutes. 

The map $\frak{F}_{\times, \operatorname{res}}$ is an embedding; we 
will show it is a homotopy equivalence.  By construction, the retraction $\Bet \to \Bet^{<}$ given by Lemma \ref{lem:zbethomretract} is a map of $(\uu, \uu)$ spaces. The same is true for the induced retraction $\Bet \setminus \bS_\Bet \to \Bet^{<} \setminus (\bS_{\Bet} \cap \Bet^{<})$. Likewise, the retraction $\Bet(\Gamma / e) \to \Bet^{<}(\Gamma / e)$ is a map of $\H^1(\Gamma /e, \uu) \times \H_1(\Gamma / e, \uu)$-spaces. The product of these retracts induces a retraction of $\Bet(\Gamma / e) \star_{\uu} (\Bet \setminus \bS_\Bet)$ onto the image of $\frak{F}_{\times, \operatorname{res}}$.  
In particular, $\frak{F}_{\times, \operatorname{res}}^*$ is an isomorphism. 

We claim all the maps in Diagram \eqref{diagextended} induce isomorphisms on cohomology. 
Indeed, we have seen this now for all but $\kappa'$, for which it follows by commutativity
of the diagram. 

We would be finished if the maps $\kappa' \circ \frak{F}_{\Gamma / e}$ and 
$\frak{F}_{\times, \operatorname{res}} \circ i_\epsilon$ were equal; unfortunately, 
they are not.  To complete the proof of the theorem, we will show these
two maps are homotopic. 

Recall 
$\Bet(\Gamma / e) \star_{\U_1} (\Bet \setminus \bS_\Bet)$ is obtained from some moment fiber
$\mu_{\U_1}^{-1}(\zeta) \subset \Bet(\Gamma / e) \times (\Bet \setminus \bS_{\Bet})$ by a free $\U_1$ quotient. 
To distinguish the notation, 
we will write $\mu_{\U_1}^{-1}(\zeta)_\Dol \subset \Dol(\Gamma / e)  \times \Dol$ for the corresponding moment fiber in the product of Dolbeault spaces.

Define:
\begin{eqnarray*}
	\tilde{f}_+: \Bet(\Gamma / e) & \to &  [\Gm \times \Gm] = \Bet \setminus \bS_\Bet \\
	y & \mapsto & (1, \zeta / \mu_{\Gm}(y))
\end{eqnarray*} 
Consider $q^{-1}(\epsilon) \subset \Dol$ where $\epsilon \in \D$ is a small positive real number. We have an isomorphism $q^{-1}(\epsilon) \cong [\U_1 \times \U_1]$ as a $(\U_1, \U_1)$-manifold.

Any two such isomorphisms differ by multiplication by a smooth section of the moment map $[\U_1 \times \U_1] \to \U_1$. In particular, there is a unique such isomorphism which, when composed with $\frak{F} : \Dol \to \Bet$, intertwines the coordinates of $\U_1 \times \U_1$ with the angle coordinates of $\C^* \times \C^*: \frak{F}(e^{i\theta_1}, e^{i\theta_2}) = (r_1e^{i\theta_1}, r_2 e^{i \theta_2} ) = (x,xy+1)$. Note that this is automatic for the second coordinate by the construction of $\frak{F}$, but not for the first coordinate.

Define:
\begin{eqnarray*}
	\tilde{f}_-  \colon \Dol(\Gamma / e) & \to & [\U_1 \times \U_1] \xrightarrow{\sim} q^{-1}(\epsilon) \subset \Dol \setminus \bS_\Dol \\ 
	x & \mapsto & (1, \zeta / \mu_{\U_1}(x)_\Dol)
\end{eqnarray*} 
We define a similar map $f_-  \colon \Bet(\Gamma / e) \to q^{-1}(\epsilon)$ by the same formula. 

It follows from the definitions that the following diagram commutes, excluding the dashed ``$=$".  

\begin{equation} \label{diag:escher}
	\begin{tikzcd}
		\Dol(\Gamma / e) \arrow[d, "id \times \tilde{f}_-"] \arrow[r, "\frak{F}_{\Gamma / e}"] \arrow[bend right=80, dd, "i_{\epsilon}", swap] & \Bet(\Gamma / e)
		\arrow[r, dashed, equals] \arrow[d, "id \times f_-"] & \Bet(\Gamma / e)  \arrow[d, "id \times \tilde{f}_+", swap]  
		\arrow[dd, bend left = 60, equals] \\
		\arrow[d] \mu^{-1}_{\U_1}(\zeta)_{\Dol}  \arrow[r, "\frak{F}_{\Gamma / e} \times \frak{F}"]  & \arrow[d] \mu_{\U_1}^{-1}(\zeta) & \arrow[l] \arrow[d] \mu_{\C^*}^{-1}(\zeta) \\
		\Dol(\Gamma / e) \star_{\U_1} (\Dol \setminus \bS_{\Dol}) \arrow[r, "\frak{F}_{\times, \operatorname{res}}"]  & \Bet(\Gamma / e)  \star_{\U_1} (\Bet \setminus \bS_\Bet) &  \arrow[l, "\kappa' ", swap] \Bet(\Gamma / e)  \star_{\C^*} (\Bet \setminus \bS_\Bet) \\
	\end{tikzcd}
\end{equation} 

To construct a homotopy between the compositions of the maps around the outside of Diagram \ref{diag:escher}, 
it now suffices to  show that the maps $\tilde{f}_+,f_-$ are homotopic via a family of maps $ \Bet(\Gamma / e)  \to \Bet \setminus \bS_\Bet$ whose graph lies in $\mu_{\U_1}^{-1}(\zeta)$.

We pick a deformation retraction $r_t$ of $\C$ onto $\epsilon$, chosen so that it restricts to a deformation retraction of $\C \setminus \R^{\leq 0}$ onto $\epsilon$. For instance, we may suppose $\epsilon$ lies on the positive real axis, and define $r_t(z) = t\epsilon + (1-t)z$.

The space $\Bet \setminus \bS_\Bet$ is a $\U_1$-bundle over $\C \times \uu \setminus (\R^{\leq 0} \times 1)$, with a trivialisation given by $\arg(x)$.
Using this trivialisation, we can lift $r_t$ to a deformation retraction $\tilde{r}_t$ of $\Bet \setminus \bS_\Bet$ onto $q^{-1}(\epsilon)$ which keeps both $\uu$-coordinates unchanged. 

By definition, both $\tilde{f}_+$ and $f_-$ have image contained in the image of the unit section of the $\U_1$ bundle, and their projections to the $\U_1$-factor both coincide with $\zeta / \mu_{\U_1} = \zeta / e^{i\arg(\mu_{\C^*})}$.
Then $\tilde{r}_t \circ \tilde{f}_+$ is the desired homotopy from $\tilde{f}_+$ to $f_-$. 
\end{proof}

\begin{corollary} \label{cor:dolbdelconstricpreserves}
The Dolbeault deletion-contraction sequence strictly preserves the perverse Leray filtration.
\end{corollary}
\begin{proof}
Follows immediately from Theorem \ref{Hodgeintertwinesdelcon}, Theorem \ref{thm:wequals2p} and the corresponding fact for the Betti deletion-contraction sequence. \end{proof}

\appendix

\section{Recollections on cohomology and filtrations} 

\subsection{The long exact sequence of a pair} \label{sec:LESpair} Let $A \xrightarrow{f} X$ be an embedding of topological spaces. The pullback $\H^{\bullet}(X, \Q) \xrightarrow{f^*} \H^{\bullet}(A, \Q)$ extends to a long exact sequence 
$$... \to \H^{\bullet}(X, A, \Q) \xrightarrow{\comap(f)^*} \H^{\bullet}(X, \Q) \xrightarrow{f^*} \H^{\bullet}(A, \Q) \to ...$$ where $\comap(f)$ is the map of pairs $X, \emptyset \to X, A$. Given a second embedding $B \xrightarrow{} Y$ and a map of pairs given by a commutative diagram
\begin{equation} \label{diag:generalmapofpairs}
\begin{tikzcd}
	A \rar \arrow[d, "F_{\operatorname{res}}"] & X \arrow[d, "F"] \\
	B \rar & Y \\
\end{tikzcd}
\end{equation}
we obtain a map of long exact sequences of pairs
\begin{equation}
\begin{tikzcd}
	\  \rar & \H^{\bullet}(Y,B,\Q) \arrow[d, "F_{\relmap}^*"] \rar & \H^{\bullet}(Y,\Q) \arrow[d, "F^*"] \rar & \H^{\bullet}(B,\Q) \arrow[d, "F_{\operatorname{res}}^*"] \rar & \ \\
	\ \rar & \H^{\bullet}(X,A,\Q) \rar & \H^{\bullet}(X,\Q) \rar & \H^{\bullet}(A,\Q) \rar & \ \\
\end{tikzcd}
\end{equation} 
\begin{lemma} \label{lem:bridgetothecoastlemma}
If $A,B$ are codimension $d$ submanifolds of $X,Y$, we can identify $F_{\relmap}^*$ with \linebreak $(F|_{X \setminus A})^* \colon \H^{\bullet - d}(Y \setminus B,\Q) \to \H^{\bullet - d}(X \setminus A,\Q)$.
\end{lemma}

\subsection{The residue exact triangle} \label{app:residuetriang}
Recall that if $X$ is a topological space and $p \colon V \to X $ is a closed subset, and $q \colon X\setminus V \to X$
its open complement, then there is the Verdier dual exact triangle of a pair: 
$$p_* p^! \C_V \to \C_X \to q_* \C_{X \setminus V} \xrightarrow{[1]}$$

In case $X$ is a smooth complex manifold and $D$ is a smooth divisor, this becomes

$$ \C_D[-2] \to \C_X \to q_* \C_{X \setminus D} \xrightarrow{[1]}$$ 

We can take analytic 
de Rham resolutions of the constant sheaves to obtain  

\begin{equation} \label{residue exact triangle} 
p_* \Omega_{D}^\bullet [-2] \to \Omega_X^\bullet \to q_* \Omega^\bullet_{X \setminus D} \xrightarrow{[1]}
\end{equation} 

In this setting one can replace $q_* \Omega^\bullet_{X \setminus D}$ by the complex of differential forms with log poles along $D$, 
which we denote $\Omega^\bullet_{X \langle D \rangle}$.  Having made this replacement, 
the connecting map may be identified with the residue (see e.g. \cite[II.3]{D0}).  Thus 
$$\Omega_X^\bullet \xrightarrow{\sim} {}^D \Omega_{X}^\bullet := \mathrm{Cone}(\Omega^\bullet_{X\langle D \rangle} \xrightarrow{res}  p_* \Omega^\bullet_D[-1])[-1]$$
where the map is the inclusion of forms into log forms $\Omega_X^\bullet \hookrightarrow  \Omega^\bullet_{X\langle D \rangle}$.  
Note that ${}^D \Omega_{X}^\bullet$ (tautologically) 
allows us to replace the residue exact triangle \eqref{residue exact triangle} with the (quasi-isomorphic) triangle associated to the exact sequence of complexes
\begin{equation} \label{residue exact sequence}
0 \to p_* \Omega^\bullet_D[-2] \to  {}^D \Omega_{X}^\bullet \to \Omega^\bullet_{X\langle D \rangle} \to 0
\end{equation} 

\vspace{2mm}

More generally, if $X$ is a topological space with an increasing filtration by closed subsets 
$X_n \subset X_{n-1} \subset \cdots \subset X_0 = X$, then we may iterate this procedure to obtain
an expression for $\C_{X}$ as a twisted complex on the sum of shifts of the star pushforwards of 
the $\C_{X_i \setminus X_{i+1}}$.   

When $X$ is a complex manifold, $D = \bigcup D_k$ is a  normal crossings divisor, 
and $X_i$ above is the codimension $i$ intersections of the $D_k$, this complex can be 
explicitly described in terms of differential forms with log poles, as in the case of a single divisor 
above. This construction is presumably standard, 
but we did not find a convenient reference, so give some details here.  For convenience we assume $X$ affine 
and pass to termwise global sections of de Rham complexes. 

Again we write $\Omega^\bullet_{X\langle D \rangle}$ for the complex of holomorphic differential forms 
with log poles along $D$.  Recall this means that 
in coordinates where $D$ is cut out by $\prod_{i} z_i = 0$, the sheaf
$\Omega^1_{X\langle D \rangle}$ is locally free and generated over $\Omega^1_X$ by $d \log z_i$ and 
$\Omega^\bullet_{X \langle D \rangle}$ is the exterior algebra on $\Omega^1_{X\langle D \rangle}$, 
here equipped with the de Rham differential.

We fix  notation for indexing divisors. 
Given  $J \subset \{1,...,n\}$, let $\generalstratum_\generalsubedges$ be the intersection of components $D_j$ for $j \in J$, and $D^J$ be their union. Write $J^c$ for the complement of $J$.
Note that $\closedpart_J \colonequals \generalstratum_\generalsubedges \cap D^{J^c}$ is a normal crossings divisor in $\generalstratum_\generalsubedges$; let $\openpart_J \colonequals \generalstratum_\generalsubedges \setminus \closedpart_J$ be its complement.

Suppose $j \notin J$. We take the residue of a form in $ \Omega^{\bullet}_{\generalstratum_\generalsubedges \langle \closedpart_J \rangle}$ along the component $D_j \cap D^J$ of $\closedpart_J$:  
\[ \operatorname{res}_j  \colon \Omega^{l}_{\generalstratum_\generalsubedges \langle \closedpart_J \rangle} \to \Omega^{l-1}_{D_{J \cup j} \langle \closedpart_{J \cup j} \rangle}. \]

\begin{definition} \label{omega x d} 
Let
${}^D\Omega_{X}^{2k, l} \colonequals \bigoplus_{|J|=k} \Omega^{l}_{\generalstratum_\generalsubedges \langle \closedpart_J \rangle}$
and ${}^D\Omega^\bullet_{X} = \bigoplus {}^D \Omega_{X}^{2k, l}$.  The latter carries 
the de Rham differential $d_{dR}$ of bidegree $(0,1)$, 
and an endomorphism $d_{\operatorname{res}}$, given by the sum of all residue maps, 
of bidegree $(2, -1)$. 
\end{definition}

\begin{lemma}
We have $d_{dR}^2 = d_{\operatorname{res}}^2 = (d_{dR} + d_{\operatorname{res}})^2 = 0$.
\end{lemma}
\begin{proof}
$d_{dR}^2=0$ is of course standard. To show $d_{\operatorname{res}}^2 = 0$ , we must check that for any two distinct  $j,j'$, the corresponding residue maps in $d_{\operatorname{res}}$ anti-commute. The two different compositions correspond to integration over a $2$-torus with the two opposite orientations, which implies the result.

To verify $(d_{dR} + d_{\operatorname{res}})^2 = 0$, it remains to check that $d_{dR} d_{\operatorname{res}} = -  d_{\operatorname{res}} d_{dR}$. 
Let us focus on the term $\operatorname{res}_{J \to J'}$ of $d_{\operatorname{res}}$ taking the residue along $z_1$. We can locally write any form as a sum of terms $f(z) \frac{dz_1}{z_1} \omega$ or $f(z) \omega$, 
where $\omega$ is a $d_{dR}$-closed 
form nonsingular along $z_1$. In the first case, we have:

$$\operatorname{res} d_{dR}f(z) \frac{dz_1}{z_1} \omega = -df(z) \omega \qquad \qquad 
d_{dR} \operatorname{res} f(z) \frac{dz_1}{z_1} \omega = df(z) \omega$$ 
Here we have used the same notation for a form nonsingular along $z_1$ and its restriction to $z_1=0$. 
In the second case, both sides vanish. This concludes the proof.
\end{proof}

We henceforth regard ${}^D \Omega^\bullet_{X}$ as a singly-graded complex (the sum of the previous gradings) 
equipped with the differential $d_{dR} + d_{\operatorname{res}}$.  This complex retains a filtration
by the size of $J$ (the first degree of the bidegree).  

As  $(\generalstratum_\generalsubedges,  \closedpart_J)$ and $(X \setminus \generalstratum_\generalsubedges, D \setminus \generalstratum_\generalsubedges)$
are again pairs of a space and normal crossings divisor, we also have complexes 
${}^{\closedpart_J} \Omega^\bullet_{\generalstratum_\generalsubedges}$ and ${}^{D \setminus \generalstratum_\generalsubedges} \Omega^\bullet_{X \setminus \generalstratum_\generalsubedges} $.  
We additionally define 
${}^D \Omega^\bullet_{X \langle \generalstratum_\generalsubedges \rangle} \subset {}^{D \setminus \generalstratum_\generalsubedges} \Omega^\bullet_{X \setminus \generalstratum_\generalsubedges}$ as the subcomplex of forms whose extensions to the closure of the relevant stratum in $X$ have log poles
along $\generalstratum_\generalsubedges$, when this is a divisor.  The inclusion 
${}^D \Omega^\bullet_{X \langle \generalstratum_\generalsubedges \rangle} \subset {}^{ D \setminus \generalstratum_\generalsubedges} \Omega^\bullet_{X \setminus \generalstratum_\generalsubedges}$, being built from inclusions of log de Rham complexes in de Rham complexes, 
is a quasi-isomorphism.  

\begin{proposition} \label{LESdeletion}
For any $J$, there is an exact sequence of complexes
\begin{equation} \label{eq:exactsequencecousins}
	0 \to {}^{\closedpart_J}  \Omega^\bullet_{\generalstratum_\generalsubedges}[-2|J|] \to {}^D\Omega^\bullet_{X} \to 
	{}^D \Omega^\bullet_{X \langle  \generalstratum_\generalsubedges \rangle}  \to 0
\end{equation}
where the first map is the inclusion of summands from Def. \ref{omega x d}, and the second map is the projection onto the complementary summands. 
\end{proposition}
\begin{proof}
There is evidently such an exact sequence of underlying graded spaces; indeed a split one, which remains split upon
imposing the differential $d_{dR}$ which preserves the terms associated to any given stratum.  The differential $d_{res}$ 
carries terms associated to some stratum to terms associated to its closure, hence still respects the above maps.  
\end{proof} 
\begin{remark}
The corresponding projection map
${}^D \Omega^\bullet_{X} \to {}^{D \setminus \generalstratum_\generalsubedges} \Omega^\bullet_{X \setminus \generalstratum_\generalsubedges}$ 
is not surjective, since
forms in the codomain are not required to have simple poles along $\generalstratum_\generalsubedges$.  
\end{remark} 

We have a natural inclusion $\Omega_X^{\bullet} \to {}^D\Omega^{\bullet}_{X}$, defined by the inclusion $\Omega^{\bullet}_{X} \to \Omega^{\bullet}_{X \langle D \rangle}$, the latter being a summand of the underlying graded vector space of 
${}^D \Omega^{\bullet}_{X}$, and $d_{\operatorname{res}}$ restricted to the image of 
$\Omega_X^\bullet$ being trivial.

\begin{proposition} \label{prop:cousincomputescohomology}
The inclusion of complexes $\Omega^{\bullet}_{X} \to {}^D\Omega^{\bullet}_{X}$ is a quasi-isomorphism. 
\end{proposition}
\begin{proof}
We proceed by induction on the number of components of $D$. The statement is 
tautologous when $D$ has no components. Choose a component $D_j$ of $D$. 	
Consider the diagram: 
\begin{equation*} \label{mapofcousintriangles}
	\begin{tikzcd}
		\Omega^{\bullet}_{D_j}[-2]  \arrow[d] \arrow[r, dashed] &  \Omega^{\bullet}_{X} \arrow[d] \arrow[r] & \Omega^{\bullet}_{X \langle D_j \rangle} \arrow[d] \arrow[r]  &   \  \\
		{}^{\closedpart_j} \Omega^{\bullet}_{D_j} [-2] \arrow[r]    &	{}^D \Omega^{\bullet}_{X} \arrow[r]  &  {}^D \Omega^{\bullet}_{X \langle D_j \rangle} \arrow[r]  & \
	\end{tikzcd}
\end{equation*}
Here, the upper triangle is \eqref{residue exact triangle}, the lower triangle is the triangle associated to \eqref{eq:exactsequencecousins},
the vertical arrows are the inclusions of forms into log forms on the appropriate stratum, the dashed arrow is dashed as a reminder that
it is defined only in the derived category (and the commutativity of the left square must also be understood in the derived category).  

By induction the left vertical map is a quasi-isomorphism.  The right vertical map maps to 
$\Omega^{\bullet}_{X \setminus D_j} \to  {}^{D \setminus D_j} \Omega^{\bullet}_{X \setminus D_j}$ by quasi-isomorphisms, and this latter
map is itself a quasi-isomorphism by induction.  Hence the center vertical map is a quasi-isomorphism as well. 
\end{proof}

\begin{corollary} \label{cor:lesofapairdesc}
The exact triangle $\C_{D_j}[-2] \to \C_X \to j_* \C_{X \setminus D_j}$ is quasi-isomorphic
to the triangle ${}^{\closedpart_j} \Omega^{\bullet}_{D_j}[-2] \to {}^D \Omega^{\bullet}_{X} \to {}^{ D} \Omega^{\bullet}_{X \langle D_j \rangle}$.  
\end{corollary}

\subsection{Filtrations}

If $V$ is a vector space with an increasing filtration $F$, we write the steps of the filtration as 
$$ \cdots \subset F_{-1} V \subset F_0 V \subset F_1 V \subset \cdots$$

We recall 

\begin{definition}
Let $V, W$ be filtered vector spaces. A map $g: V \to W$ is said to be: 
\begin{enumerate} 
	\item compatible with the filtrations if $g(F_k V) \subset F_k W$
	\item strictly compatible with the filtrations if $F_k W \cap g(V) = g(F_k V)$
\end{enumerate} 
We will also synonymously say the map (strictly) preserves or (strictly) respects the filtration. 
\end{definition}

The significance of the strictness condition is: 

\begin{lemma} \label{lem:strictmapsexactassocg}
Let $... \to V_{-1} \xrightarrow{a} V_0 \xrightarrow{b} V_1 \to ...$ be a long exact sequence of filtered vector spaces, whose maps strictly preserve the filtrations. Then the sequence defined by the associated graded spaces $... \to \gr_k V_{-1} \xrightarrow{\gr(a)} \gr_k V_0 \xrightarrow{\gr(b)} \gr_k V_1 \to ...$ is also exact. \end{lemma}
\begin{proof}
We first show check that the kernel of $\gr(b)$ contains the image of $\gr(a)$. Given $u \in F_i V$, write $\gr_i(u)$ for the associated element of $\gr_i V$. Consider $w \in F_k V_{-1}$. By exactness of the original sequence, $b(a(w))=0$. It follows that $\gr(b)(\gr(a)(\gr_k(w)))=0$.

We now show that the image of $\gr(a)$ contains the kernel of $\gr(b)$. Suppose $v \in F_k V_0$ satisfies $\gr(b)(\gr_k(v)) = 0$. By definition, this means $b(v) \in F_{k-1}V_1$. By strictness of $b$, there exists $v' \in F_{k-1} V_0$ such that $b(v') = b(v)$. By exactness of the sequence, there exists $w \in V_{-1}$ such that $a(w) = v-v'$. By strictness, there exists $w' \in F_k V_{-1}$ with the same property. It follows that $\gr(a)(\gr_k(w')) = \gr_k(v - v') = \gr_k(v)$. 
\end{proof}

\begin{caution}
For a fixed short exact sequence $0 \to A \to B \to C \to 0$ of vector spaces, and fixed filtrations on $A, C$, there are many filtrations on $B$ such that the maps strictly preserve filtrations. Indeed let $A$ and $C$ be one-dimensional, with $\gr_1 A = A, \gr_0 C = C$. Then the filtration on $B$ is determined by the subspace $F_0 B \cong C$. The only condition on $F_0 B$ is that is must intersect the image of $A$ trivially.
\end{caution}

\subsection{Weight filtration}

According to \cite{D1, D2, D3}, cohomology of an algebraic varieties carry various filtrations, 
strictly preserved by pullback with respect to any morphism of algebraic varieties.  Of relevance to us here is 
the weight filtration, defined on the rational cohomology, and denoted $W_\bullet \H^\bullet(X, \Q)$.  It is an increasing filtration,
with $W_\bullet \H^n(X, \Q)$ supported in degrees $[0, 2n]$ in general, and in degrees $[0, n]$ and $[n, 2n]$ if $X$ is
projective and smooth respectively.  

The weight filtration of a smooth variety $X$ is defined by choosing a normal crossings compactification.  Then 
the complex cohomology of $X$ is computed by a complex of differential forms with log singularities 
along the boundary, and $W_{\bullet+k}\H^{\bullet}(X,\Q)$ is generated by the subsheaf of forms singular along at most $k$ different boundary components near any given point. 

In this sense, the size of $Gr^W_{\bullet+k}\H^{\bullet}(X,\Q)$ for $k>0$ is a measure of the non-compactness of $X$. 

\subsection{Perverse Leray filtration} \label{app:perverse}
Let $B$ be a topological space. For a complex of sheaves $K$ on $B$, one can define a filtration on the (hyper)cohomology of $K^\bullet$ by cutting off
the complex:  

$$P_k \H^{\bullet}(B; K) \colonequals Image(\H^{\bullet}(B; \tau^{\le k} K) \to \H^{\bullet}(B; K))$$

In case one has another $t$-structure available --- in our case the middle perverse $t$-structure for constructible sheaves on algebraic varieties ---
one gets a similar filtration by using the truncations of the $t$-structure.  We term the resulting filtration the 
{\em perverse filtration}.  

In the setting where one has a map $\pi \colon X \to B$ and $K = \pi_* F$, the perverse filtration on $K$ is the filtration
which arises on $\H^{\bullet}(X; F)$ from the perverse t-structure Leray spectral sequence.  
Thus it is called the
{\em perverse Leray filtration} on $\H^{\bullet}(X; F)$.  See \cite{dCHM} and the references therein for
discussion of this filtration.  

\begin{convention} \label{perverse covention}
Let $f: X \to Y$ be a map of algebraic or complex analytic spaces.  When we discuss 
the perverse Leray filtration $\H^\bullet(X, \Q)$ associated by $f$, i.e. the perverse filtration on $\H^\bullet(f_* \Q)$, we always shift the filtration so that $P_{-1} = 0$ and $1 \in P_0 \H^\bullet(X, \Q)$.  

In some circumstances, we may wish to further shift the filtration.  We will write 
$ \H^\bullet(X, \Q)\{n\}$ to indicate we have shifted the filtration by $n$ steps, i.e. 
$P_i  \H^\bullet(X, \Q)\{n\} =  P_{i+n} \H^\bullet(X, \Q)$. 
\end{convention}

Pullback and pushforward operations generally respect only `half' of the perverse t-structure.  As a consequence, the perverse Leray filtration is not generally preserved by base-change.  (In particular, 
the perverse Leray filtration of the base-change to a point is always trivial, in the sense of agreeing
with the filtration by cohomological degree.)  However, base-change which is ``transverse to the 
singularities of the sheaf'' does respect perversity.  This can be precisely formulated in the 
language of microsupport of Kashiwara-Schapira, as we now recall.

\begin{proposition}  \label{noncharacteristic perversity} \cite[Cor. 10.3.16]{KS} Let $N \subset M$ be an inclusion of smooth complex submanifolds and 
$\cF$ a perverse sheaf on $M$.  
If $N$ is noncharacteristic for $\cF$, i.e., $ss(\cF) \cap T_N^*M \subset T_M^*M$, then $\cF|_{N}[\dim N -\dim M]$ is perverse.  
\end{proposition}

\begin{corollary} \label{transverse perverse} 
Consider a diagram of complex manifolds
\begin{equation} 
	\begin{tikzcd}
		A \arrow[d] \arrow[r, "f"] & \arrow[d]  C 
		\\
		B \arrow[r] & D
	\end{tikzcd}
\end{equation}
Then $\H^\bullet(C\homsep \C)$ and $\H^\bullet(A \homsep \C)$ have perverse
Leray filtrations coming from the maps to $B$ and $D$.  If the vertical maps are proper, and the diagram is Cartesian
$(A = B \times_D C)$ and transverse, then $f^*: \H^\bullet(C\homsep \C) \to \H^\bullet(A \homsep \C)$ 
strictly preserve the perverse Leray filtrations. 
\end{corollary} 
\begin{proof}
Follows from Proposition \ref{noncharacteristic perversity} and the standard microsupport estimate 
for proper pushforward (\cite[Prop. 5.4.4]{KS}).  
\end{proof}

\section{A holomorphic embedding $\Dol_J \subset \Bet$}

Here we recall a more canonical choice of the embedding $\frak{F}$ of Lemma \ref{lem:tateNAHT}. 
The construction is due to Yang Li. We do not 
logically depend on this result, but it provides a better analogy with the hyperk\"ahler 
aspects of the nonabelian Hodge diffeomorphism. 

The sphere of complex structures compatible with the hyperk\"ahler metric on $\Dol$ are parametrized by a choice of unit vector $s \in \R^3$, as described by \eqref{eq:complexstructures}. So far, we have considered the complex structure $I_{e_1}$ on $\Dol$, with respect to which the fibers of $\Dol \to \D$ are complex subvarieties. On the other hand, the complex structures $I_{a e_2 + b e_3}$ for different $(a,b) \in S^1$ define biholomorphic manifolds, because the function $V$ is symmetric under rotations of the $e_2, e_3$-plane. They were described in \cite{Li}. For concreteness, let $J$ be the complex structure associated to $e_3$, and let $\Dol_J$ denote the manifold $\Dol$ with this choice of complex structure. Let $\eta = u_1 + i u_2$ be the complex part of the moment map. The map $\exp(2 \pi i \eta) : \Dol_J \to \C^*$ is holomorphic with image the annulus $\exp(-2 \pi r) < |z| < \exp(2 \pi r)$. Its fibers are complex surfaces, which are topologically annuli away from $1 \in \C^*$. Holomorphic coordinates on these fibers can be constructed as follows. 
\begin{theorem} \cite[Section 1.3.2]{Li} \label{li embedding} 
There is a holomorphic embedding $\Dol_J \to \C^2 \setminus \{ xy = 1 \}$, identifying $\exp(2\pi i \eta)$ with $1 - xy$ and $u_3$ with a strictly monotonic function of $|x|^2 - |y|^2$ (for fixed $\eta$).
\end{theorem}
\begin{proof}
We summarize Li's argument here, adapting the notation slightly to match our own, and refer to \cite{Li} for details. The first task is to define holomorphic functions $x, y : \Dol_J \to \C$ satisfying $xy = 1 - \exp(2 \pi i \eta)$. Let $\theta_\infty \in \C$ and consider the following one-forms on $\tilde{\Dol}_J \setminus \{ \eta \in \Z \}$ : 
\[ \zeta_x = V d u_3 + i\theta + \left( \frac{\pi i}{2} - i\theta_\infty + \lim_{k \to \infty} \sum_{-k}^{k} \left(  \frac{1}{2(\eta + n)} - \frac{u_3}{2(\eta + n)\sqrt{u_3^2 + |\eta + n|^2}} \right) \right) d\eta \]
\[ \zeta_y = -V d u_3 - i\theta + \left( \frac{\pi i}{2} + i\theta_\infty +  \lim_{k \to \infty} \sum_{-k}^{k}  \left(  \frac{1}{2(\eta + n)} + \frac{u_3}{2(\eta + n)\sqrt{u_3^2 + |\eta + n|^2}} \right) \right) d \eta. \]
The apparently singularities of the second term at $\eta \in \Z$ can be eliminated if we assume $u_3 \neq 0$. On the other hand, the singularities at $\eta \in \Z, u_3 = 0$ cannot be removed. A direct calculation shows that away from this locus $\zeta_x, \zeta_y$ are closed $\Z$-invariant $(1,0)$-forms.

The line integrals $\int \zeta_x, \int \zeta_y$ therefore define local holomorphic functions on $\tilde{\Dol}_J$ away from $\eta \in \Z$. Moreover, the period around a $\U_1$-orbit is $2 \pi i$, and for a unique choice of $\theta_\infty$ the period spanning a $\Z$-translation also vanishes. One therefore has global holomorphic functions
\[ x = \exp(\int \zeta_x), \ \ y = \exp(\int \zeta_y). \]
on $\Dol_J$ defined away from $u_3 \leq 0, \eta =0$ and $u_3 \geq 0, \eta = 0$ respectively. Moreover, a direct calculation shows 
\begin{equation} \label{eq:complementsequation} \zeta_x + \zeta_y = d \log (1 - \exp(2 \pi i \eta)). \end{equation} 
and therefore
\[ xy = 1 - \exp(2 \pi i \eta) \]
wherever the left-hand side is defined. This in turn proves that $x,y$ extend to holomorphic functions on all of $\Dol_J$, vanishing along $u_3 \leq 0, \eta =0$ and $u_3 \geq 0, \eta = 0$ respectively.

Finally, we check that the map $(x,y) : \Dol_J \to \C^2$ is an embedding. It is enough to verify that it restricts to an embedding on any fiber of $\exp(2 \pi i \eta)$. Each such fiber is a $\U_1$-bundle over an interval $a < u_3 <  b$. The restriction of $\zeta_x$ to such a fiber equals $V du_3 + i\theta$. It follows that $x$ restricts to $\int (V du_3 + i\theta)$. The injectivity of $x$ follows from the fact that $V$ is everywhere positive and $\theta$ is a connection one-form. The claim regarding $u_3 = u_3(|x|^2 - |y|^2, \eta)$ likewise follows from the fact that $V$ is everywhere positive.
\end{proof}

\end{document}